\documentclass{article}
\usepackage{amssymb,amsmath,amsthm,ascmac}
\usepackage{enumerate}

\newtheorem{thm}{Theorem}

\newtheorem{lem}{Lemma}[section]

\newtheorem{rem}{Remark}[section]

\newcommand{\dis}{\displaystyle}

\newcommand{\R}{{\Bbb R}}
\newcommand{\N}{{\Bbb N}}

\newcommand{\pa}{\partial}

\makeatletter

\@addtoreset{equation}{section}
\makeatother

\usepackage{fancyhdr}
\title{\Large\sf A type II blowup for the six dimensional energy critical heat equation}
\author{Junichi Harada
\\{\small Faculty of Education and Human Studies, Akita University}
\\[1mm]{\small email: harada-j@math.akita-u.ac.jp}}

\begin{document}
\maketitle
\thispagestyle{empty}

\begin{abstract}
We study blowup solutions of the 6D energy critical heat equation
$u_t=\Delta u+|u|^{p-1}u$ in $\R^n\times(0,T)$.
A goal of this paper is to show the existence of type II blowup solutions
predicted by Filippas, Herrero and Vel\'azquez \cite{FilippasHV}.
The dimension six is a border case whether a type II blowup can occur or not.
Therefore
the behavior of the solution is quite different from other cases.
In fact,
our solution behaves like
 \[
 u(x,t)\approx
 \begin{cases}
 \lambda(t)^{-2}{\sf Q}(\lambda(t)^{-1}x) & \text{in the inner region: }
 |x|\sim\lambda(t),
 \\
 -(p-1)^\frac{1}{p-1}(T-t)^{-\frac{1}{p-1}} & \text{in the selfsimilar region: }
 |x|\sim\sqrt{T-t}
 \end{cases}
 \]
with $\lambda(t)=(1+o(1))(T-t)^\frac{5}{4}|\log(T-t)|^{-\frac{15}{8}}$.
The local energy
$E_\text{loc}(u)
 =\frac{1}{2}\|\nabla u\|_{L^2(|x|<1)}^2-\frac{1}{3}\|u\|_{L^3(|x|<1)}^3$
of the solution goes to $-\infty$.
\end{abstract}


\section{Introduction}
This paper is concerned with blowup solutions for the semilinear heat equation.
 \begin{equation}\label{1.1}
 \begin{cases}
 u_t = \Delta u+|u|^{p-1}u
 &
 \text{in } \R^n\times(0,T),\\
 u(x,0)=u_0(x)
 &
 \text{on } \R^n.
 \end{cases}
 \end{equation}
It is well known that
for any bounded continuous initial data $u_0$,
\eqref{1.1} admits a unique maximal classical solution $u(x,t)$ in $t\in(0,T)$. 
When the maximal existence time $T$ is finite,
we say that a solution blows up in a finite time $T$,
which is equivalent to $\limsup_{t\to T}\|u(t)\|_\infty=\infty$.
Blowup solutions are classified into two classes.
 \begin{align*}
 \limsup_{t\to T}(T-t)^\frac{1}{p-1}\|u(t)\|_\infty
 &<\infty
 \qquad (\text{type I}),
 \\
 \limsup_{t\to T}(T-t)^\frac{1}{p-1}\|u(t)\|_\infty
 &=\infty
 \qquad (\text{type II}).
 \end{align*}
In this paper,
we are interested in a type II blowup mechanism,
which is much more delicate than a type I case.
In the study of blowup problems,
there are two important critical exponents defined by
 \[
 p_{\text{S}}=\frac{n+2}{n-2},
 \hspace{10mm}
 p_{\text{JL}}=
 \begin{cases}
 \infty & \text{if } n\leq10,
 \\ \dis
 1+\frac{4}{n-4-2\sqrt{n-1}} & \text{if } n\geq11.
 \end{cases}
 \]
A blowup for the case $1<p<p_\text{S}$ is almost completely understood
(\cite{Giga,Giga2,Giga4,Quittner,Filippas,HerreroV,Velazquez,MerleZ}).
Every blowup solution is a type I,
and has the same local profile:
 \[
 u(x,t) \approx (p-1)^{-\frac{1}{p-1}}(T-t)^{-\frac{1}{p-1}}
 \qquad
 \text{near the sigular point}.
 \]
It is known that
for the case $p_S<p<p_{JL}$,
only a type I blowup occurs under a radially symmetric setting (\cite{MatanoM}).
The first type II blowup is discovered by Herrero and Vel\'azquez
\cite{HerreroV2,HerreroV3} (see also \cite{Mizoguchi}) for $p>p_\text{JL}$.
In this pioneering work,
they give a widely applicable approach for the existence of blowup solutions with the prescribed asymptotic profile,
and construct infinitely many positive radially symmetric type II blowup solutions with the exact blowup rates.
Very recently Seki \cite{Seki} proves the existence of a type II blowup
for the critical case $p=p_{\text{JL}}$ (see also \cite{Seki2}).
Another type II blowup is found in the energy critical case $p=p_{\text{S}}$
by Filippas, Herrero and Vel\'azquez \cite{FilippasHV}.
They formally obtain sign changing type II blowup solutions
by using the matched asymptotic expansion technique.
The blowup rate of their solutions are given by
 \begin{align}\label{1.2}
 \|u(t)\|_\infty
 &\sim
 \begin{cases}
 (T-t)^{-k} & n=3,
 \\
 (T-t)^{-k}|\log(T-t)|^\frac{2k}{2k-1} & n=4,
 \\
 (T-t)^{-k} & n=5,
 \\
 (T-t)^{-\frac{1}{4}}|\log(T-t)|^{-\frac{15}{8}} & n=6,
 \end{cases}
 \hspace{7.5mm}
 \|u(t)\|_\infty
 \sim
 \begin{cases}
 (T-t)^{-k} & n=3,
 \\
 (T-t)^{-k}|\log(T-t)|^\frac{2k}{2k-1} & n=4,
 \\
 (T-t)^{-3k} & n=5,
 \\
 (T-t)^{-\frac{5}{2}}|\log(T-t)|^\frac{15}{4} & n=6,
 \end{cases}
 \\ \nonumber
 & \hspace{25mm}
 (\text{original})
 \hspace{67mm}
 (\text{corrected})
 \end{align}
where $k=1,2,3,\cdots$.
The list on the right side in \eqref{1.2} is a corrected version of the original one
(see Appendix of \cite{Harada2} for details).
The first rigorous proof for the existence of a type II blowup
is given by Schweyer \cite{Schweyer}.
He constructs a type II blowup solution for $n=4$
by a different approach from
the matched asymptotic expansion technique \cite{FilippasHV} and \cite{HerreroV2,HerreroV3}.
Its blowup rate coincides with $k=1$ in \eqref{1.2}.
Very recently del Pino, Musso, Wei and their collaborators
succeed to treat blowup problems for the energy critical case
by developing a new method (\cite{Cortazar,delPino, delPino2}).
They obtain a type II blowup solution for $n=5$ with the same blowup rate as $k=1$
in \eqref{1.2}.
After that,
the author \cite{Harada2} proves the existence of a type II blowup corresponding to $n=5$ and $k\geq2$ in \eqref{1.2}
by using their method.
In this paper,
we investigate the possibility of a type II blowup for $n=6$.
This is a border case whether a type II blowup can occur or not.
In fact,
Collot, Merle and Rapha\"el \cite{Collot} prove that
if the initial data is close to the ground states,
then the blowup is a type I for a higher dimensional case $n\geq7$.
We will see that
a type II blowup occurs for $n=6$
and
its asymptotic profile is quite different from other cases.

\section{Main result}
Let ${\sf Q}_\lambda(x)$ be the positive radially symmetric stationary solution given by
 \[
 {\sf Q}_\lambda(x)
 =
 \lambda^{-\frac{n-2}{2}
 }\left( 1+\frac{1}{n(n-2)}\frac{|x|^2}{\lambda^2} \right)^{-\frac{n-2}{2}}
 \qquad
 (\text{ground state}).
 \]
We introduce an ODE type blowup solution.
 \[
 \Theta(x,t)
 =
 (T-t)^{-\frac{1}{p-1}}
 \left( (p-1)+\frac{\alpha}{\tau}e_1(z) \right)^{-\frac{1}{p-1}},
 \qquad z=\frac{x}{\sqrt{T-t}}, \quad \tau=|\log(T-t)|,
 \]
where $\alpha>0$ and $e_1(z)$ is the quadratic polynomial in $|z|$
(see Section \ref{S4.2} for details).
This function coincides with the ODE solution near the singular point.
\[
 \Theta(x,t)
 \approx
 (p-1)^{-\frac{1}{p-1}}(T-t)^{-\frac{1}{p-1}}
 \qquad
 \text{in } |z|< R
\]
for any fixed $R>0$.
Furthermore
we define a cut off function to connect two solutions.
 \[
 \chi_1
 =
 \begin{cases}
 1 & \text{if } |z|<\tau^{-1}, \\
 0 & \text{if } |z|>2\tau^{-1},
 \end{cases}
 \qquad z=\frac{x}{\sqrt{T-t}}, \quad \tau=|\log(T-t)|.
 \]
\begin{thm}\label{Thm1}
 Let $n=6$ and $p=p_{\rm S}$.
 There exist $T>0$ and a radially symmetric solution
 $u(x,t)\in C(\R^6\times[0,T))\cap C^{2,1}(\R^6\times(0,T))$
 of \eqref{1.1}  satisfying the following properties$:$
 \begin{enumerate}[\rm(i)]
 \item The function $u(x,t)$ is given by
 \begin{equation}\label{2.1}
 u(x,t)
 =
 \underbrace{
 {\sf Q}_{\lambda(t)}(x)
 -
 \frac{\frac{5}{4}+\frac{15}{8|\log(T-t)|}}{T-t}T_1(y)\chi_1
 -
 \Theta(x,t)(1-\chi_1)}_{=:u_{\rm app}(x,t)}
 +
 v(x,t),
 \qquad
 y=\frac{x}{\lambda(t)},
 \end{equation}
 where  $T_1(y)$ is a bounded function defined in Section {\rm\ref{S4.1}}.
 \item
 $\lambda(t)=(1+o(1))(T-t)^{\frac{5}{4}}|\log(T-t)|^{-\frac{15}{8}}$.
 \item There exist $c,K>0$ such that
 \[
 |v(x,t)|<
 c
 \begin{cases}
 \left( 1+|z|^2 \right)\tau^{-\frac{3}{2}}(T-t)^{-1}
 & {\rm for}\ |x|<K\sqrt\tau\sqrt{T-t},
 \\[1mm] \dis
 \tau^{-^\frac{7}{16}}(T-t)^{-1}|z|^{-\frac{1}{8}}
 & {\rm for}\ |x|>K\sqrt\tau\sqrt{T-t},
 \end{cases}
 \qquad
 z=\frac{x}{\sqrt{T-t}},
 \quad
 \tau=|\log(T-t)|.
 \]
 \item
 Let $E_{\rm loc}(u)$ be the local energy defined by
 \[
 E_\text{loc}(u)
 =
 \frac{1}{2}\int_{|x|<1}|\nabla u|^2dx
 -
 \frac{1}{p+1}\int_{|x|<1}|u|^{p+1}dx.
 \]
 Then $\dis\lim_{t\to T}E_{\rm loc}(u(t))=-\infty$.
 \end{enumerate}
\end{thm}
\begin{rem}
The blowup rate of this solution is given by
 \[
 \|u(t)\|_\infty
 \sim
 \|{\sf Q}_{\lambda(t)}\|_\infty
 \sim
 \lambda(t)^{-\frac{n-2}{2}}
 \sim
 (T-t)^{-\frac{5}{2}}|\log(T-t)|^\frac{15}{4}.
 \]
 This gives the same blowup rate as \eqref{1.2}.
\end{rem}
\begin{rem}
As in \eqref{2.1},
the solution behaves like
 \[
 u(x,t)
 \approx
 \begin{cases}
 {\sf Q}_{\lambda(t)}(x) & |x|\sim\lambda(t),\\
 -\Theta(x,t) & |x|\sim\sqrt{T-t}.
 \end{cases}
 \]
This is a characteristic property of this solution.
As a consequence,
{\rm(iv)} is formally derived from
 \[
 E_\text{loc}(u(t))
 \sim
 E_\text{loc}({\sf Q})+E(\Theta(t)),
 \qquad
 E_\text{loc}({\sf Q})>0,
 \qquad
 \lim_{t\to T}E_\text{loc}(\Theta(t))=-\infty.
 \]
\end{rem}
\begin{rem}
From the blowup criterion $($Theorem {\rm2.1} {\rm\cite{Giga3}}$)$,
the solution blows up only on the origin.
\end{rem}
\begin{rem}
We have no idea whether there exist solutions with a different blowup rate
for $n=6$.
\end{rem}
Our approach is a combination of
the matched asymptotic expansion technique {\rm\cite{FilippasHV}} and
the inner-outer gluing method developed in {\rm\cite{Cortazar,delPino,delPino2}}.
We first construct an approximate solution $u_\text{app}(x,t)$ (see \eqref{2.1})
along \rm\cite{FilippasHV}, and
next prove the existence of a solution close to the approximate solution
by the the inner-outer gluing method.
To do that,
we investigate the behavior of the remainder $v(x,t)$ defined in \eqref{2.1}.
Since
the solution $u(x,t)$ has two characteristic length:
 \begin{enumerate}[\rm(i)]
 \item inner region $|x|\sim\lambda(t)$ \qquad (Section \ref{S6}),
 \item selfsimilar region $|x|\sim\sqrt{T-t}$ \qquad (Section \ref{S7}), 
 \end{enumerate}
we discuss separately.
In the inner region,
we analyze the behavior of the solution near the ground state ${\sf Q}$.
This part is treated in essentially the same way as \cite{delPino}.
On the other hand,
the solution $u(x,t)$ behaves like
 \begin{equation}\label{2.2}
 u(x,t)\approx-(p-1)^{-\frac{1}{p-1}}(T-t)^{-\frac{1}{p-1}}
 \qquad
 \text{in the selfsimilar region}.
 \end{equation}
In other words,
our solution blows up in both directions at the same time $t=T$.
 \[
 \lim_{t\to T}\sup_{|x|\sim\lambda(t)}u(x,t)=\infty,
 \qquad
 \lim_{t\to T}\inf_{|x|\sim\sqrt{T-t}}u(x,t)=-\infty.
 \]
Hence a problem in this paper is to construct a solution
whose selfsimilar part behaves as above.
To obtain such a solution,
we first introduce additional parameters and carefully choose them to control the blowup time of the selfsimilar part of the solution.
After that,
we compute the effect from the inner part and the errors from the approximate procedure,
and confirm that the remainder $v(x,t)$ is sufficiently small.
We note that
since the solution $u(x,t)$ behaves like \eqref{2.2},
our problem is reduced to the ODE type blowup problem with the error terms from the inner part.
We finally compute the local energy of the solution.
This is derived from pointwise estimates of $u(x,t)$ and $\nabla u(x,t)$
(Section \ref{S8}).

\section{Preliminary}
\subsection{Notations}
\label{S3.1}
Throughout this paper,
$\eta(\xi)\in C^\infty(\R)$ stands for a standard cut off function satisfying
\[
 \eta(\xi)=
 \begin{cases}
 1 & \text{if } \xi<1,\\
 0 & \text{if } \xi>2.
 \end{cases}
 \]
The nonlinear function in \eqref{1.1} is denoted by
 \[
 f(u)=|u|^{p-1}u,
 \qquad
 p=\frac{n+2}{n-2}.
 \]

\subsection{Linearization around the ground state}
\label{S3.2}
Let us consider the eigenvalue problem related to the linearized problem around
the ground state ${\sf Q}(y)={\sf Q}_\lambda(y)|_{\lambda=1}$.
 \begin{equation}\label{3.1}
 -H_y\psi=\mu\psi \qquad\text{in } \R^n,
 \end{equation}
where the operator $H_y$ is define by
 \[
 H_y=\Delta_y+V(y),
 \qquad
 V(y)=f'({\sf Q}(y))=p{\sf Q}(y)^{p-1}.
 \]
We recall that the operator $H_y$ has a negative eigenvalue $\mu_1<0$ and a zero eigenvalu (see Proposition 5.5 p. 37 \cite{Duyckaerts}).
We denote by $\psi_1(r)$ a positive radially symmetric eigenfunction
associated to the negative eigenvalue with $\psi_1(0)=1$.
It is know that
there exists $C>0$ such that (see p. 18 \cite{Cortazar})
 \[
 \psi_1(r)
 <
 C\left( 1+r \right)^{-\frac{n-1}{2}}
 e^{-\sqrt{|\mu_1|}\,r}.
 \]
The eigenfunction associated to the zero eigenvalue is explicitly given by
 \[
 \Lambda_y{\sf Q}(y)
 =
 \left( \frac{n-2}{2}+y\cdot\nabla_y \right){\sf Q}(y)
 =
 -
 \left. \frac{\pa}{\pa\lambda}{\sf Q}_\lambda \right|_{\lambda=1}.
 \]

\subsection{Perturbated linearized problem}
\label{S3.3}
We next consider the eigenvalue problem \eqref{3.1} in a bounded but very large domain.
 \begin{equation}\label{3.2}
 \begin{cases}
 -H_y\psi=\mu\psi & \text{in } B_R,
 \\
 \psi=0 & \text{on } \pa B_R,
 \\
 \psi \text{ is radially symmetric}.
 \end{cases}
 \end{equation}
We denote the $i$th eigenvalue of \eqref{3.2} by $\mu_i^{(R)}$
and the associated eigenfunction by $\psi_i^{(R)}$.
We normalize $\psi_i^{(R)}(r)$ as  $\psi_i^{(R)}(0)=1$.
Most of lemmas stated in this subsection are proved in Section 7 \cite{Cortazar}
(see also Section 3.2 \cite{Harada2}).
 \begin{lem}[Lemma 3.1 \cite{Harada2}]
 \label{L3.1}
 There exists $c>0$ independent of $R>0$ such that
 \[
 0
 <
 \psi_1^{(R)}(r)
 <
 c
 \left( 1+r \right)^{-\frac{n-1}{2}}
 e^{-\sqrt{|\mu_1|}\,r}
 \quad\text{\rm for } r\in(0,R).
 \]
 \end{lem}
 \begin{lem}[Lemma 7.2 \cite{Cortazar}, Lemma 3.3 \cite{Harada2}]
 \label{L3.2}
 Let $n\geq5$.
 There exists $c>0$ independent of $R>0$ such that
 \[
 \mu_2^{(R)}
 >
 cR^{-(n-2)}.
 \]
 \end{lem}

\subsection{Behavior of the Laplace equation with a perturbation term}
\label{S3.4}
Consider a radially symmetric solution of
 \[
 \Delta_yp+(1-\chi_M)V(y)p=0 \quad\text{in } \R^n,
 \qquad
 \chi_M=\eta\left( \frac{|y|}{M} \right).
 \]
 Let $p_M(r)$ be a radially symmetric solution of this problem satisfying
 $p_M(r)=1$ for $r<M$.
 \begin{lem}[see proof of Lemma 7.3 \cite{Cortazar}, Lemma 3.3 \cite{Harada2}]
 \label{L3.3}
 There exist $k\in(0,1)$ and $M_1>0$ such that if $M>M_1$
 \[
 k<p_M(r)\leq1
 \quad\text{\rm for}\ r\in(0,\infty).
 \]
 \end{lem}

\subsection{The Schauder estimate for parabolic equations}
\label{S3.5}
In this subsection,
we consider
 \begin{equation}\label{3.3}
 u_t=\Delta_xu+{\bf b}(x,t)\cdot\nabla_xu+V(x,t)u+f(x,t)
 \qquad\text{in } Q,
 \end{equation}
where $Q=B_2\times(0,1)$, $B_r=\{x\in\R^n;\ |x|<r\}$.
The coefficients are assumed to be
 \begin{equation}
 \tag{A1}
 {\bf b}(x,t)\in(L^\infty(Q))^n,
 \quad
 V(x,t)\in L^\infty(Q)
 \qquad\text{with }
 \|{\bf b}\|_{L^\infty(Q)}+\|V\|_{L^\infty(Q)}<M.
 \end{equation}
For $p,q\in[1,\infty]$,
we define
 \[
 \|f\|_{L^{p,q}(Q)}
 =
 \begin{cases}
 \dis
 \left( \int_0^1\|f(t)\|_{L^p(B_2)}^qdt \right)^\frac{1}{q}
 & \text{if } q\in[1,\infty),
 \\
 \dis
 \sup_{t\in(0,1)}\|f(t)\|_{L^p(B_2)}
 & \text{if } q=\infty.
 \end{cases}
 \]

 \begin{lem}[Exercise 6.5 p. 154 \cite{Lieberman}, Theorem 8.1 p. 192 \cite{Ladyzenskaja}]
 \label{L3.4}
 Let $p,q\in(1,\infty]$ satisfy $\frac{n}{2p}+\frac{1}{q}<1$ and assume {\rm(A1)}.
 There exists $c>0$ depending on $p$, $q$, $n$ and $M$ such that
 \begin{enumerate}[{\rm (i)}]
 \item
 if $u(x,t)$ is a weak solution of \eqref{3.3},
 then
 \[
 \|u\|_{L^\infty(B_1\times(\frac{1}{2},1))}
 <
 c\left( \|u\|_{L^2(Q)}+\|f\|_{L^{p,q}(Q)} \right),
 \]
 \item
 if $u(x,t)\in C(B_2\times[0,1))$ is a weak solution of \eqref{3.3} with $u(x,t)|_{t=0}=0$,
 then
 \[
 \|u\|_{L^\infty(B_1\times(0,1))}
 <
 c\left( \|u\|_{L^2(Q)}+\|f\|_{L^{p,q}(Q)} \right).
 \]
 \end{enumerate}
 \end{lem}

 \begin{lem}[Theorem 4.8 p. 56 \cite{Lieberman}, Theorem 11.1 p. 211 \cite{Ladyzenskaja}]
 \label{L3.5}
 Let $p,q\in(1,\infty]$ satisfy $\frac{n}{p}+\frac{2}{q}<1$ and assume {\rm(A1)}.
 There exists $c>0$ depending on $p$, $q$, $n$ and $M$ such that
 \begin{enumerate}[{\rm (i)}]
 \item
 if $u(x,t)\in C^{2,1}(Q)$ is a solution of \eqref{3.3},
 then
 \[
 \|\nabla u\|_{L^\infty(B_1\times(\frac{1}{2},1))}
 <
 c\left( \|u\|_{L^\infty(Q)}+\|f\|_{L^{p,q}(Q)} \right),
 \]
 \item
 if $u(x,t)\in C^{2,1}(Q)\cap C^{2,1}(B_2\times[0,1))$ is a solution of \eqref{3.3}
 with $u(x,t)|_{t=0}=0$,
 then
 \[
 \|\nabla u\|_{L^\infty(B_1\times(0,1))}
 <
 c\left( \|u\|_{L^\infty(Q)}+\|f\|_{L^{p,q}(Q)} \right).
 \]
 \item
 if $u(x,t)\in C^{2,1}(\bar Q)$ is a solution of \eqref{3.3} with
 $u(x,t)|_{t=0}=0$ and $u(x,t)|_{\pa B_2}=0$,
 then
 \[
 \|\nabla u\|_{L^\infty(Q)}
 <
 c\left(
 \|u\|_{L^\infty(Q)}+\|f\|_{L^{p,q}(Q)}
 \right).
 \]
 \end{enumerate}
 \end{lem}

\subsection{Local behavior of the heat equation}
\label{S3.6}
To describe a local behavior of solutions to the heat equation,
we often consider
 \[
 \theta_\tau=A_z\theta
 \qquad\text{in } \R^n\times(0,\infty),
 \qquad\text{where }
 A_z=\Delta_z-\frac{z}{2}\cdot\nabla_z.
 \]
We introduce a wighted $L^2$ space.
 \begin{align*}
 L_\rho^2(\R^n)
 :=
 \{f\in L_\text{loc}^2(\R^n);\ \|f\|_\rho<\infty\},
 \qquad
 \|f\|_\rho^2=\int_{\R^n}f(z)^2\rho(z)dz,
 \qquad
 \rho(z)=e^{-\frac{|z|^2}{4}}.
 \end{align*}
The inner product is defined by
 \[
 (f_1,f_2)_\rho=\int_{\R^n}f_1(z)f_2(z)\rho(z)dz.
 \]
The corresponding eigenvalue problem in a radially symmetric setting is given by
 \[
 -A_ze_i=\lambda_ie_i
 \qquad\text{in } L_{\rho,\text{rad}}^2(\R^n).
 \]
It is known that
 \begin{itemize}
 \item $\lambda_i=i$ \ ($i=0,1,2,\cdots$) and
 \item $e_i(z)\in L_{\rho,\text{rad}}^2(\R^n)$ is the $2i$th-degree polynomial.
 \end{itemize}
We normalize the eigenfunction $e_i(z)$ as
 \[
 e_i(z)
 =
 {\sf c}_i|z|^{2k}
 +
 \sum_{l=0}^{i-1}
 {\sf c}_i^{(l)}|z|^{2l}
 \qquad\text{with }
 {\sf c}_i>0 \text{ and } \|e_i\|_\rho=1.
 \]
The first three eigenfunctions $e_0(z)$, $e_1(z)$ and $e_2(z)$ are explicitly written as
 \begin{equation}\label{3.4}
 \begin{array}{l}
 e_0(z)
 =
 {\sf c}_0,
 \\[1mm]
 e_1(z)
 =
 {\sf c}_1\left( |z|^2-2n \right),
 \\[1mm]
 e_2(z)
 =
 {\sf c}_2\left( |z|^4-(4n+8)|z|^2+(4n^2+8n) \right).
 \end{array}
 \end{equation}
 From this relation,
 we see that
 \begin{align}\label{3.5}
 (e_1^2,e_1)_\rho
 &=
 {\sf c}_1^2((|z|^2-2n)^2,e_1)_\rho
 \nonumber
 \\
 &=
 {\sf c}_1^2(|z|^4-4n|z|^2+4n^2,e_1)_\rho
 \nonumber
 \\
 &=
 {\sf c}_1^2(|z|^4-(4n+8)|z|^2+(4n^2+8n),e_1)_\rho
 +
 {\sf c}_1^2(8|z|^2-8n,e_1)_\rho
 \nonumber
 \\
 &=
 \frac{{\sf c}_1^2}{{\sf c}_2}(e_2,e_1)_\rho
 +
 {\sf c}_1^2(8|z|^2-16n,e_1)_\rho
 +
 {\sf c}_1^2(16n,e_1)_\rho
 \nonumber
 \\
 &=
 8{\sf c}_1.
 \end{align}
 In the same manner,
 we get
 \begin{equation}\label{3.6}
 (|\nabla_ze_1|,e_1)_\rho={\sf c}_1^2(4|z|^2,e_1)_\rho=4{\sf c}_1.
 \end{equation}
We finally recall a fundamental inequality (see (6) p. 4218 \cite{Harada}):
 \begin{equation}\label{3.7}
 \||z|f\|_\rho<c(\|f\|_\rho+\|\nabla_zf\|_\rho).
 \end{equation}

\subsection{The linearization around the ground state}
\label{S3.7}
We here provide a radially symmetric solution of the linearized problem around the ground state.
 \begin{equation}\label{3.8}
 H_yT=g,
 \end{equation}
where $H_y=\Delta_y+V(y)$ and $V(y)=f'({\sf Q}(y))$.
We recall that
 \begin{align*}
 \Lambda_y{\sf Q}
 =
 \left( \frac{n-2}{2}+y\cdot\nabla_y \right){\sf Q}
 =
 \frac{n-2}{2}\left( 1-\frac{|y|^2}{n(n-2)} \right)
 \left( 1+\frac{|y|^2}{n(n-2)} \right)^{-\frac{n}{2}}
 \end{align*}
gives a radially symmetric solution of $H_yT=0$.
Let $\Gamma$ be another radially symmetric solution of $H_yT=0$ given by
 \begin{align*}
 \Gamma(r)
 &=
 (\Lambda_y{\sf Q})
 \int_R^r
 \frac{dr}{(\Lambda_y{\sf Q})^2r^{n-1}}
 \qquad \text{for } r>R,
 \end{align*}
where $R>0$ is a large constant such that $\Lambda_y{\sf Q}(r)<0$ for $r>R$.
For simplicity we write $\Lambda_y{\sf Q}$ as
 \[
 \Lambda_y{\sf Q}
 =
 -\underbrace{\frac{\{n(n-2)\}^\frac{n}{2}}{2n}}_{=:\kappa}
 r^{-(n-2)}+O(r^{-n}).
 \]
The asymptotic behavior of $\Gamma(r)$ is given by
 \begin{align}\label{3.9}
 \Gamma(r)
 &=
 (\Lambda_y{\sf Q})
 \int_R^r
 \frac{dr}{(\Lambda_y{\sf Q})^2r^{n-1}}
 \nonumber
 \\
 &=
 -\kappa^{-1}
 (r^{-(n-2)}+O(r^{-n}))
 \int_R^r
 \left( r^{n-3}+O(r^{n-5}) \right)dr
 \nonumber
 \\
 &=
 \frac{-\kappa^{-1}}{n-2}
 +
 \begin{cases}
 O(r^{-1}) & (n=3),
 \\
 O(r^{-2}\log r) & (n=4),
 \\
 O(r^{-2}) & (n\geq5).
 \end{cases}
 \end{align}
In the same manner,
the derivative of $\Gamma(r)$ is computed as
 \begin{align}\label{3.10}
 \pa_r\Gamma(r)
 &=
 (\Lambda_y{\sf Q})_r
 \int_R^r
 \frac{dr}{(\Lambda_y{\sf Q})^2r^{n-1}}
 +
 \frac{1}{(\Lambda_y{\sf Q})r^{n-1}}
 \nonumber
 \\
 &=
 \frac{1}{(\Lambda_y{\sf Q})r^{n-1}}
 \left(
 (\Lambda_y{\sf Q})(\Lambda_y{\sf Q})_rr^{n-1}
 \int_R^r
 \frac{dr}{(\Lambda_y{\sf Q})^2r^{n-1}}
 +1
 \right)
 \nonumber
 \\
 &=
 \frac{1}{(\Lambda_y{\sf Q})r^{n-1}}
 \left(
 \left( -(n-2)r^{-(n-2)}+O(r^{-n}) \right)
 \int_R^r\left( r^{n-3}+O(r^{n-5}) \right)dr
 +1
 \right)
 \nonumber
 \\
 &=
 \begin{cases}
 O(r^{-2}) & (n=3),
 \\
 O(r^{-3}\log r) & (n=4),
 \\
 O(r^{-3}) & (n\geq5).
 \end{cases}
 \end{align}
 From these relations,
 it holds that
 \begin{align*}
 (\Gamma_r(\Lambda_y{\sf Q})-\Gamma(\Lambda_y{\sf Q})_r)r^{n-1}
 &=
 \text{const}
 \\
 &=
 \lim_{r\to\infty}
 (\Gamma_r(\Lambda_y{\sf Q})-\Gamma(\Lambda_y{\sf Q})_r)r^{n-1}
 \\
 &=
 0-\frac{-\kappa^{-1}}{n-2}\cdot(n-2)\kappa
 =1.
 \end{align*}
 From this identity,
 we easily see that
 \begin{equation}\label{3.11}
 T
 =
 -\Gamma
 \int_0^r
 (\Lambda_y{\sf Q})g
 r^{n-1}dr
 +
 (\Lambda_y{\sf Q})
 \int_0^r
 \Gamma gr^{n-1}dr
 \end{equation}
 gives a solution of \eqref{3.8}.

\section{Formal derivation of blowup speed}
\label{S4}
We here repeat the argument in \cite{FilippasHV} p. 2970 - p. 2972
to construct an approximate solution, and derive its blowup rate.

\subsection{Inner region}
\label{S4.1}
We look for a solution of the form:
 \[
 u(x,t)\approx \lambda(t)^{-\frac{n-2}{2}}{\sf Q}(y), \qquad y=\frac{x}{\lambda(t)}
 \]
near the singular point.
The unknown function $\lambda(t)$ represents a characteristic length of
the inner region: $|x|\sim\lambda(t)$.
We change variables.
\[
 u(x,t)
 =
 \lambda^{-\frac{n-2}{2}}({\sf Q}(y)+\epsilon(y,t)),
 \qquad y=\frac{x}{\lambda(t)}.
\]
The function $\epsilon(y,t)$ solves
\[
 \lambda^2\epsilon_t
 =
 H_y\epsilon
 +
 \sigma\Lambda_y\epsilon
 +
 \sigma\Lambda_y{\sf Q}
 +
 N,
 \qquad\sigma=\lambda\dot\lambda,
\]
where $H_y=\Delta_y+f'({\sf Q}(y))$ and
\[
 \Lambda_y=\frac{n-2}{2}+y\cdot\nabla_y,
 \qquad
 N=f({\sf Q}+\epsilon)-f({\sf Q})-f'({\sf Q})\epsilon.
\]
To obtain more precious asymptotic behavior of the solution,
we rewrite
\[
 u(x,t)
 =
 \lambda^{-\frac{n-2}{2}}({\sf Q}(y)+\sigma(t)T_1(y)+\epsilon_1(y,t)),
 \qquad
 \sigma=\lambda\dot\lambda.
\]
The function $\epsilon_1(y,t)$ solves
\[
 \lambda^2\pa_t\epsilon_1
 +
 \lambda^2\dot\sigma T_1
 =
 H_y\epsilon_1
 +
 \sigma(\Lambda_y{\sf Q}+H_yT_1)
 +
 \sigma^2\Lambda_yT_1
 +
 \sigma\Lambda_y\epsilon_1
 +
 N.
\]
We choose $T_1(y)$ as a solution of $H_yT_1+\Lambda_y{\sf Q}=0$.
From this choice,
we can expect $|\epsilon_1|\ll\sigma$.
Therefore
the solution should behave as
 \begin{equation}\label{4.1}
 u(x,t)\approx
 \lambda^{-\frac{n-2}{2}}{\sf Q}(y)
 +
 \lambda^{-\frac{n-2}{2}}\sigma T_1(y),
 \qquad\sigma=\lambda\dot\lambda
 \end{equation}
in the inner region.
We here compute $T_1(y)$ for a later argument.
 From \eqref{3.11},
 $T_1(y)$ is written by
 \[
 T_1
 =
 \underbrace{
 -\Gamma
 \int_0^r
 (\Lambda_y{\sf Q})^2
 r^{n-1}dr}_{=:T_1'}
 +
 \underbrace{
 (\Lambda_y{\sf Q})
 \int_0^r
 \Gamma(\Lambda_y{\sf Q})r^{n-1}dr}_{=:T_1''},
 \qquad
 r=|y|.
 \]
When $n\geq5$,
 $T_1'$ and $T_1''$ are given by
 \begin{align*}
 T_1'
 &=
 -\Gamma\int_0^r
 (\Lambda_y{\sf Q})^2
 r^{n-1}dr
 \\
 &=
 -\Gamma\int_0^\infty
 (\Lambda_y{\sf Q})^2
 r^{n-1}dr
 +
 \Gamma\int_r^\infty
 (\Lambda_y{\sf Q})^2
 r^{n-1}dr
 \\
 &=
 -\Gamma\int_0^\infty
 (\Lambda_y{\sf Q})^2
 r^{n-1}dr
 +
 O(r^{-(n-4)}),
 \\
 T_1''
 &=
 O(r^{-(n-4)}).
 \end{align*}
 The integral in $T_1'$ is computed as
 \begin{align*}
 \int_0^\infty
 (\Lambda_y{\sf Q})^2
 r^{n-1}dr
 &=
 \left( \frac{n-2}{2} \right)^2
 \int_0^\infty
 \left( 1-\frac{r^2}{n(n-2)} \right)^2
 \left( 1+\frac{r^2}{n(n-2)} \right)^{-n}r^{n-1}dr
 \\
 &=
 \left( \frac{n-2}{2} \right)^2
 \{n(n-2)\}^\frac{n}{2}
 \int_0^\infty
 \frac{(1-t^2)^2t^{n-1}}{(1+t^2)^n}dt
 \\
 &=
 \left( \frac{n-2}{2} \right)^2
 \frac{\{n(n-2)\}^\frac{n}{2}}{2}
 \int_0^\infty \frac{(1-s)^2s^\frac{n-2}{2}}{(1+s)^n}ds.
 \end{align*}
When $n=6$,
the integrand is written as
 \begin{align*}
 \frac{(1-s)^2s^2}{(1+s)^6}
 &=
 \frac{(s+1-2)^2(s+1-1)^2}{(1+s)^6}
 \\
 &=
 \frac{(s+1)^4-6(s+1)^3+13(s+1)^2-12(s+1)+4}{(1+s)^6}
 \\
 &=
 \frac{1}{(s+1)^2}
 +
 \frac{-6}{(s+1)^3}
 +
 \frac{13}{(s+1)^4}
 +
 \frac{-12}{(s+1)^5}
 +
 \frac{4}{(s+1)^6}.
 \end{align*}
 Therefore
 we get
 \begin{align*}
 \int_0^\infty
 (\Lambda_y{\sf Q})^2
 r^{n-1}dr
 =
 -\left(
 -1+3-\frac{13}{3}+3-\frac{4}{5}
 \right)
 =
 \frac{2}{15}
 \qquad (n=6).
 \end{align*}
This relation and \eqref{3.9} give
 \begin{align*}
 T_1'
 &=
 -\Gamma
 \int_0^\infty
 (\Lambda_y{\sf Q})^2
 r^{n-1}dr
 +
 O(r^{-(n-4)})
 \\
 &=
 \frac{1}{n-2}\frac{2n}{\{n(n-2)\}^\frac{n}{2}}
 \left( \frac{n-2}{2} \right)^2
 \frac{\{n(n-2)\}^\frac{n}{2}}{2}
 \frac{2}{15}
 +
 O(r^{-2})
 +
 O(r^{-(n-4)})
 \\
 &=
 \frac{4}{5}+O(r^{-2})
 \qquad (n=6).
 \end{align*}
We finally obtain
 \begin{equation}\label{4.2}
 T_1(r)=\frac{4}{5}+O(r^{-2}),
 \qquad r=|y|
 \qquad (n=6).
 \end{equation}
Furthermore
since $\pa_r\Gamma=O(r^{-3})$ (see \eqref{3.10}),
it holds that
 \begin{equation}\label{4.3}
 \pa_rT_1(r)=O(r^{-3}),
 \qquad r=|y|
 \qquad (n=6).
 \end{equation}

\subsection{Selfsimiliar region}
\label{S4.2}
 We next investigate the asymptotic behavior of the solution in the selfsimilar region.
 In this region,
 we use the selfsimilar variables.
 \[
 z=\frac{x}{\sqrt{T-t}},
 \qquad
 T-t=e^{-\tau}.
 \]
 We assume that the solution behaves like
 \begin{equation}\label{4.4}
 u(x,t)
 \approx
 -\Theta(x,t)
 \qquad\text{in the selfsimilar region},
 \end{equation}
where $\Theta(x,t)$ is an explicit function given by
 \begin{align}\label{4.5}
 \Theta(x,t)
 =
 (T-t)^{-\frac{1}{p-1}}
 \left( (p-1)+\frac{\alpha}{\tau}e_1(z) \right)^{-\frac{1}{p-1}}
 \end{align}
with
 \begin{equation}\label{4.6}
 \alpha
 =
 \frac{2(p-1)^2}{p}\frac{1}{(e_1^2,e_1)_\rho}>0.
 \end{equation}
The constant $(e_1^2,e_1)_\rho$ is given by \eqref{3.5}.
This function $\Theta(x,t)$ is an approximation of an ODE blowup solution,
which is introduced in \cite{HerreroV, Velazquez}.
For a later matching procedure,
we compute the asymptotic behavior of $\Theta(x,t)$ as $|z|\to0$.
From \eqref{3.4} - \eqref{3.5} and \eqref{4.6},
it holds that
 \begin{align}\label{4.7}
 \Theta(x,t)
 &=
 \{(p-1)(T-t)\}^{-\frac{1}{p-1}}
 \left( 1+\frac{\alpha}{p-1}\frac{{\sf c}_1(|z|^2-2n)}{\tau} \right)^{-\frac{1}{p-1}}
 \nonumber
 \\
 &=
 \{(p-1)(T-t)\}^{-\frac{1}{p-1}}
 \left(
 1-\frac{\alpha}{(p-1)^2}\frac{{\sf c}_1(|z|^2-2n)}{\tau}
 +
 O\left(\frac{|z|^4+1}{\tau^2}\right)
 \right)
 \nonumber
 \\
 &=
 \{(p-1)(T-t)\}^{-\frac{1}{p-1}}
 \left(
 1-\frac{2}{p}\frac{1}{8{\sf c}_1}
 \frac{-2n{\sf c}_1}{\tau}
 +
 O\left(\frac{|z|^2}{\tau}\right)
 +
 O\left(\frac{|z|^4+1}{\tau^2}\right)
 \right)
 \nonumber
 \\
 &\approx
 \{(p-1)(T-t)\}^{-\frac{1}{p-1}}
 \left(
 1+\frac{n}{2p\tau}
 \right)
 \qquad
 \text{as } |z|\to0.
 \end{align}

\subsection{Matching condition}
In Section \ref{S4.1} - Section \ref{S4.2},
we derive the asymptotic form of the solution in the inner region (see \eqref{4.1})
and the selfsimilar region (see \eqref{4.4}) respectively.
We here provide a condition such that
those two solutions coincide at the boundary of the two regions.
The formulas \eqref{4.1} and \eqref{4.4} give the following matching condition.
 \[
 \lambda^{-\frac{n-2}{2}}{\sf Q}(y)
 +
 \lambda^{-\frac{n-2}{2}}\sigma T_1(y)
 \approx
 -\Theta(x,t)
 \qquad
 \text{as } |y|\to\infty \ \text{ and } \ |z|\to0.
 \]
From \eqref{4.2} and \eqref{4.7},
this condition is rewritten as
 \[
 \frac{\{n(n-2)\}^\frac{n}{2}}{\lambda^\frac{n-2}{2}}
 \frac{1}{|y|^{n-2}}
 +
 \frac{4}{5}
 \frac{\sigma}{\lambda^\frac{n-2}{2}}
 \approx
 -
 \{(p-1)(T-t)\}^{-\frac{1}{p-1}}
 \left(
 1+\frac{n}{2p\tau}
 \right)
 \qquad
 (n=6)
 \]
as $|y|\to\infty$ and $|z|\to0$.
We here assume that
the first term on the left-hand side is negligible since $|y|\to\infty$.
Then
when $n=6$,
the matching condition is given by
\[
 \frac{\dot\lambda}{\lambda}
 =
 -
 \frac{5}{4}
 \left(
 1+\frac{3}{2\tau}
 \right)
 (T-t)^{-1}.
\]
From this equation,
we obtain
 \begin{equation}\label{4.8}
 \lambda
 =
 (T-t)^\frac{5}{4}\tau^{-\frac{15}{8}},
 \qquad
 \tau=|\log(T-t)|.
 \end{equation}

\section{Formulation}
\label{S5}
In this section,
we set up our problem as in the proof of Theorem 1 \cite{delPino}
(see also Section 5 \cite{Harada2}).

\subsection{Setting}
\label{S5.1}
As in the previous section,
we look for solutions of the form:
 \begin{equation}\label{5.1}
 u(x,t)
 =
 \lambda^{-\frac{n-2}{2}}{\sf Q}(y)
 +
 \lambda_0^{-\frac{n-2}{2}}\sigma T_1(y)\chi_1
 -
 \Theta(x,t)\chi_2
 +
 v(x,t),
 \qquad
 y=\frac{x}{\lambda},
 \end{equation}
where $\lambda=\lambda(t)$ is an unknown function,
$\chi_i$ ($i=1,2$) is a cut off function defined by
 \[
 \chi_1
 =
 \eta\left( \tau|z| \right),
 \qquad
 \chi_2=1-\chi_1,
 \qquad
 z=\frac{x}{\sqrt{T-t}},
 \qquad
 \tau=|\log(T-t)|
 \]
and $\lambda_0$, $\sigma$ are functions given by (see \eqref{4.8})
 \begin{align}\label{5.2}
 \lambda_0
 =
 (T-t)^\frac{5}{4}\tau^{-\frac{15}{8}},
 \hspace{10mm}
 \sigma
 =
 -\left( \frac{5}{4}+\frac{15}{8\tau} \right)
 (T-t)^\frac{3}{2}\tau^{-\frac{15}{4}}.
 \end{align}
We remark that
the second term $\lambda_0^{-\frac{n-2}{2}}\sigma T_1$ in \eqref{5.1}
is slightly different from that of \eqref{4.1},
where $\lambda^{-\frac{n-2}{2}}\sigma T_1$ is used.
This gives a better estimate in our argument (see Lemma \ref{L7.2}).
 To derive an equation for $v(x,t)$,
 we compute the derivative of $\Theta(x,t)$ (see \eqref{4.5}).
 \begin{align*}
 \pa_t\Theta
 &=
 \frac{1}{p-1}
 (T-t)^{-\frac{p}{p-1}}
 \left( (p-1)+\frac{\alpha}{\tau}e_1 \right)^{-\frac{1}{p-1}}
 \\
 &\qquad
 -
 \frac{1}{p-1}
 (T-t)^{-\frac{p}{p-1}}
 \left( (p-1)+\frac{\alpha}{\tau}e_1 \right)^{-\frac{p}{p-1}}
 \left(
 -\frac{\alpha}{\tau^2}e_1
 +
 \frac{\alpha}{\tau}
 \frac{z}{2}\cdot\nabla_ze_1
 \right),
 \\
 \Delta_x\Theta
 &=
 \frac{-1}{p-1}
 (T-t)^{-\frac{p}{p-1}}
 \left( (p-1)+\frac{\alpha}{\tau}e_1 \right)^{-\frac{p}{p-1}}
 \frac{\alpha}{\tau}
 \Delta_ze_1
 \\
 &\qquad
 +
 \frac{p}{(p-1)^2}
 (T-t)^{-\frac{p}{p-1}}
 \left( (p-1)+\frac{\alpha}{\tau}e_1 \right)^{-\frac{2p-1}{p-1}}
 \frac{\alpha^2}{\tau^2}
 |\nabla_ze_1|^2.
 \end{align*}
 This gives
 \begin{align}\label{5.3}
 \pa_t\Theta-\Delta_x\Theta-\Theta^p
 &=
 \underbrace{
 \frac{\alpha}{p-1}
 (T-t)^{-\frac{p}{p-1}}
 \left( (p-1)+\frac{\alpha}{\tau}e_1 \right)^{-\frac{p}{p-1}}
 \frac{e_1}{\tau^2}
 }
 \\
 &\qquad
 \underbrace{
 -\frac{p\alpha^2}{(p-1)^2}
 (T-t)^{-\frac{p}{p-1}}
 \left( (p-1)+\frac{\alpha}{\tau}e_1 \right)^{-\frac{2p-1}{p-1}}
 \frac{|\nabla_ze_1|^2}{\tau^2}
 }_{=:\mu(x,t)}.
 \nonumber
 \end{align}
We compute the derivative of $u(x,t)$.
\begin{align*}
 u_t
 &=
 -\frac{\dot\lambda}{\lambda^\frac{n}{2}}\Lambda_y{\sf Q}
 -
 \Theta_t\chi_2
 +
 v_t
 +
 g_0+g_1,
 \\
 \Delta_xu
 &=
 -
 \frac{f({\sf Q})}{\lambda^\frac{n+2}{2}}
 -
 \frac{\sigma\Lambda_y{\sf Q}}{\lambda_0^\frac{n-2}{2}\lambda^2}\chi_1
 -
 \frac{\sigma VT_1}{\lambda_0^\frac{n-2}{2}\lambda^2}\chi_1
 -
 (\Delta_x\Theta)\chi_2
 +
 \Delta_xv
 +
 g_2
 +
 g_3,
\end{align*}
where $g_i$ ($i=0,1,2,3$) is given by
\begin{align*}
 g_0
 &=
 -\frac{\sigma}{\lambda_0^\frac{n-2}{2}}
 \left( \frac{n-2}{2}\frac{\dot\lambda_0}{\lambda_0}
 +
 \frac{\dot\lambda}{\lambda}y\cdot\nabla_y
 \right)
 T_1\cdot\chi_1
 +
 \frac{\dot\sigma T_1}{\lambda_0^\frac{n-2}{2}}\chi_1
 +
 \frac{\sigma T_1}{\lambda_0^\frac{n-2}{2}}\pa_t\chi_1,
 \\
 g_1
 &=
 -\Theta\pa_t\chi_2,
 \\
 g_2
 &=
 \frac{2\sigma}{\lambda_0^\frac{n-2}{2}\lambda}\nabla_yT_1\cdot\nabla_x\chi_1
 +
 \frac{\sigma T_1}{\lambda_0^\frac{n-2}{2}}\Delta_x\chi_1,
 \\
 g_3
 &=
 -2\nabla_x\Theta\cdot\nabla_x\chi_2
 -
 \Theta\Delta_x\chi_2.
\end{align*}
The remainder $v(x,t)$ satisfies
 \begin{align}\label{5.4}
 v_t
 &=
 \Delta_xv
 +
 \frac{V}{\lambda^2}(v-\Theta\chi_2)
 +
 \left(
 \frac{\dot\lambda}{\lambda^\frac{n}{2}}
 -
 \frac{\sigma}{\lambda_0^\frac{n-2}{2}\lambda^2}
 \right)
 (\Lambda_y{\sf Q})\chi_1
 +
 \frac{\dot\lambda}{\lambda^\frac{n}{2}}
 (\Lambda_y{\sf Q})\chi_2
 \\
 &\qquad
 +
 \mu
 +
 \underbrace{
 (-\mu\chi_1)
 +
 (-g_0)+(-g_1)+g_2+g_3+g_4}_{=:g}
 +
 N,
 \nonumber
 \end{align}
where
 \begin{align*}
 N
 &=
 f(\lambda^{-\frac{n-2}{2}}{\sf Q}+\lambda_0^{-\frac{n-2}{2}}\sigma T_1\chi_1
 -\Theta\chi_2+v)
 -
 f(\lambda^{-\frac{n-2}{2}}{\sf Q})
 \\
 &\qquad
 -
 f'(\lambda^{-\frac{n-2}{2}}{\sf Q})
 (\lambda_0^{-\frac{n-2}{2}}\sigma T_1\chi_1-\Theta\chi_2+v)
 -
 f(-\Theta\chi_2),
 \\
 g_4
 &=
 -f(-\Theta)\chi_2
 +
 f(-\Theta\chi_2).
 \end{align*}
We decompose $v(x,t)$ as
 \[
 v(x,t)
 =
 \lambda^{-\frac{n-2}{2}}\epsilon(y,t)\chi_\text{in}
 +
 w(x,t),
 \qquad
 y=\frac{x}{\lambda}.
 \]
The function $\epsilon(y,t)$ is defined on $(y,t)\in B_{2R}\times(0,T)$ and
 \[
 \chi_\text{in}=\chi\left( \frac{|y|}{R} \right).
 \]
The constant $R$ will be chosen large enough in Section \ref{S6} - Section \ref{S7}.
Plugging this into \eqref{5.4},
we get
 \begin{align}\label{5.5}
 \frac{\epsilon_t}{\lambda^{\frac{n-2}{2}}}
 \chi_\text{in}
 +
 w_t
 &=
 \frac{H_y\epsilon}{\lambda^{\frac{n+2}{2}}}
 \chi_\text{in}
 +
 \Delta w
 +
 \frac{V}{\lambda^2}(w-\Theta\chi_2)
 +
 \left(
 \frac{\dot\lambda}{\lambda^\frac{n}{2}}
 -
 \frac{\sigma}{\lambda_0^\frac{n-2}{2}\lambda^2}
 \right)
 (\Lambda_y{\sf Q})\chi_1
 \\
 &\qquad
 +
 \frac{\dot\lambda}{\lambda^\frac{n}{2}}
 (\Lambda_y{\sf Q})\chi_2
 +
 \frac{\dot\lambda}{\lambda^\frac{n}{2}}
 (\Lambda_y\epsilon)\chi_\text{in}
 +
 h_\text{in}
 +
 \mu
 +
 g
 +
 N,
 \nonumber
 \end{align}
 where
 \[
 h_\text{in}
 =
 \frac{1}{\lambda^\frac{n+2}{2}}
 \left(
 2\nabla_y\epsilon\cdot\nabla_y\chi_\text{in}
 +
 \epsilon\Delta_y\chi_\text{in}
 -
 \lambda^2\epsilon\pa_t\chi_\text{in}
 \right).
 \]
We introduce a parabolic system of $(\epsilon(y,t),w(x,t))$.
 \begin{equation}\label{5.6}
 \begin{cases}
 \lambda^2\epsilon_t
 =
 H_y\epsilon+G(\lambda,w)
 &
 \text{in } B_{2R}\times(0,T),
 \\
 \dis
 w_t
 =
 \Delta_xw
 +
 (1-\chi_\text{in})\frac{G(\lambda,w)}{\lambda^\frac{n+2}{2}}
 -
 \frac{V\Theta}{\lambda^2}\chi_2
 +
 \frac{\dot\lambda}{\lambda^\frac{n}{2}}
 (\Lambda_y{\sf Q})\chi_2
 +
 \frac{\dot\lambda}{\lambda^\frac{n}{2}}
 (\Lambda_y\epsilon)\chi_\text{in}
 \\ \dis
 \hspace{10mm}
 +
 h_\text{in}
 +
 \mu
 +
 g
 +
 N
 & \text{in } \R^n\times(0,T),
 \end{cases}
 \end{equation}
where
 \begin{align}\label{5.7}
 G(\lambda,w)
 &=
 \left(
 \lambda\dot\lambda
 -
 \left( \frac{\lambda}{\lambda_0} \right)^\frac{n-2}{2}
 \sigma
 \right)
 (\Lambda_y{\sf Q})
 \chi_1
 +
 \lambda^\frac{n-2}{2}Vw(x,t).
 \end{align}
We rewrite the second equation of \eqref{5.6} in a different expression.
 \begin{align}\label{5.8}
 w_t
 &=
 \Delta_xw
 +
 f'(-\Theta\chi_2)w
 +
 (1-\chi_\text{in})\frac{G(\lambda,w)}{\lambda^\frac{n+2}{2}}
 -
 \frac{V\Theta}{\lambda^2}\chi_2
 +
 \frac{\dot\lambda}{\lambda^\frac{n}{2}}
 (\Lambda_y{\sf Q})\chi_2
 \\
 &\qquad
 +
 \frac{\dot\lambda}{\lambda^\frac{n}{2}}
 (\Lambda_y\epsilon)\chi_\text{in}
 +
 h_\text{in}
 +
 \mu
 +
 g_\text{out}
 +
 N_\text{out}
 \qquad
 \text{in } \R^n\times(0,T),
 \nonumber
 \end{align}
where
 \begin{align*}
 N_\text{out}
 &=
 f\left(
 \lambda^{-\frac{n-2}{2}}{\sf Q}+\lambda_0^{-\frac{n-2}{2}}\sigma T_1\chi_1
 -\Theta\chi_2+\lambda^{-\frac{n-2}{2}}\epsilon\chi_\text{in}+w
 \right)
 -
 f(-\Theta\chi_2)
 \\
 &\qquad
 -
 f'(-\Theta\chi_2)
 \left( \lambda^{-\frac{n-2}{2}}{\sf Q}+\lambda_0^{-\frac{n-2}{2}}\sigma T_1\chi_1
 +\lambda^{-\frac{n-2}{2}}\epsilon\chi_\text{in}+w \right)
 \\
 &\qquad
 -
 f(\lambda^{-\frac{n-2}{2}}{\sf Q})
 -
 f'(\lambda^{-\frac{n-2}{2}}{\sf Q})
 \left( \lambda_0^{-\frac{n-2}{2}}\sigma T_1\chi_1-\Theta\chi_2
 +\lambda^{-\frac{n-2}{2}}\epsilon\chi_\text{in}+w \right),
 \\
 g_\text{out}
 &=
 g+g_5,
 \\
 g_5
 &=
 f'(-\Theta\chi_2)
 \left(
 \lambda^{-\frac{n-2}{2}}{\sf Q}+\lambda_0^{-\frac{n-2}{2}}\sigma T_1\chi_1
 +
 \lambda^{-\frac{n-2}{2}}\epsilon\chi_\text{in}
 \right).
 \end{align*}
We can verify that
$(\epsilon(y,t),w(x,t))$ solves \eqref{5.5}
if a solution $(\epsilon(y,t),w(x,t))$ of \eqref{5.6} is obtained.
This equation is the same one introduced in \cite{delPino}.
By a lack of boundary condition in the equation for $\epsilon(y,t)$ in \eqref{5.6},
the problem may not be uniquely solvable.
So we appropriately construct a solution $\epsilon(y,t)$
such that $\epsilon(y,t)$ decays enough in the region $|y|\sim R$
(see \eqref{6.23}).

\subsection{Fixed point argument}
\label{S5.2}
To construct a solution of \eqref{5.6},
we apply a fixed point argument.
We put
 \begin{align}\label{5.9}
 {\cal W}(x,t)
 &=
 \begin{cases}
 \dis
 \frac{e^\tau}{\tau^\frac{3}{2}}
 \left( 1+|z|^2 \right)
 & \text{for } |z|<K\sqrt{\tau},
 \\[3mm]
 \dis
 \frac{K^\frac{17}{8}e^{\tau}}{\tau^\frac{7}{16}}
 \frac{1}{|z|^\frac{1}{8}}
 & \text{for } |z|> K\sqrt{\tau}.
 \end{cases}
 \end{align}
The constant $K>0$ is chosen to be large enough later.
Let $\delta\in(0,T)$ and
$X_\delta$ be the space of all continuous functions on $\R^n\times[0,T-\delta]$ satisfying
 \[
 |w(x,t)|
 \leq
 {\cal W}(x,t)
 \quad
 \text{for }
 t\in[0,T-\delta].
 \]
We extend $w(x,t)$ to a continuous function on $\R^n\times[0,T]$.
 \[
 \bar{w}(x,t)=
 \begin{cases}
 w(x,t)
 & \text{if } (x,t)\in\R^n\times(0,T-\delta],
 \\
 w(x,T-\delta)\chi_\delta(t)
 & \text{if } (x,t)\in\R^n\times(T-\delta,T).
 \end{cases}
 \]
The function $\chi_\delta(t)$ is a cut off function satisfying
 \[
 \chi_\delta(t)
 =
 \begin{cases}
 1 & \text{if } t<T-\delta, \\
 0 & \text{if } t>T-\frac{\delta}{2}.
 \end{cases}
 \]
Furthermore we define
 \[
 \tilde{w}(x,t)
 =
 \begin{cases}
 \bar{w}(x,t) & \text{if } t\in[0,T-\delta],
 \\
 {\cal W}(x,t) & \text{if } t\in[T-\delta,T-\frac{\delta}{2}] \text{\ \ and\ \ }
 \bar{w}(x,t)>{\cal W}(x,t),
 \\
 \bar{w}(x,t) & \text{if } t\in[T-\delta,T-\frac{\delta}{2}] \text{\ \ and\  \ }
 |\bar{w}(x,t)|\leq{\cal W}(x,t),
 \\
 -{\cal W}(x,t) & \text{if } t\in[T-\delta,T-\frac{\delta}{2}] \text{\ \ and\ \ }
 \bar{w}(x,t)<-{\cal W}(x,t),
 \\
 0 & \text{if } t\in[T-\frac{\delta}{2},T].
 \end{cases}
 \]
From this definition,
we see that $\tilde{w}(x,t)\in C(\R^n\times[0,T])$ and
 \begin{equation}\label{5.10}
 |\tilde{w}(x,t)|
 \leq
 \begin{cases}
 {\cal W}(x,t) & \text{for } (x,t)\in\R^n\times(0,T-\frac{\delta}{2}),\\
 0 & \text{for } (x,t)\in\R^n\times(T-\frac{\delta}{2},T).
 \end{cases}
 \end{equation}
From this construction,
the mapping
 \[
 (w,\|\cdot\|_{C(\R^n\times[0,T-\delta])})
 \quad\mapsto\quad
 (\tilde w,\|\cdot\|_{C(\R^n\times[0,T])})
 \]
is continuous.
For given $w(x,t)\in X_\delta$,
we first determine $\lambda(t)$ by the orthogonal condition \eqref{6.1}.
Next
we construct $\epsilon(y,t)$ as a solution of (see \eqref{5.6})
 \[
 \lambda^2\epsilon_t=H_y\epsilon+G(\lambda,\tilde w)
 \qquad\text{in } B_{2R}\times(0,T).
 \]
After that
we solve the problem (see \eqref{5.8})
 \begin{align*}
 W_t
 &=
 \Delta_xW
 +
 f'(-\Theta\chi_2)W
 +
 (1-\chi_\text{in})
 \frac{G(\lambda,\tilde w)}{\lambda^\frac{n+2}{2}}
 -
 \frac{V\Theta}{\lambda^2}\chi_2
 +
 \frac{\dot\lambda}{\lambda^\frac{n}{2}}
 (\Lambda_y{\sf Q})\chi_2
 \\
 &\qquad
 +
 \frac{\dot\lambda}{\lambda^\frac{n}{2}}
 (\Lambda_y\epsilon)\chi_\text{in}
 +
 h_\text{in}
 +
 \mu
 +
 g_\text{out}
 +
 N_\text{out}
 \qquad
 \text{in } \R^n\times(0,T).
 \nonumber
 \end{align*}
By using this $W(x,t)$,
we define the mapping $w(x,t)\mapsto W(x,t)$.
The fixed point of this mapping gives a solution of \eqref{5.6}
in $\R^n\times(0,T-\delta)$ satisfying \eqref{5.10}.
Finally
we take $\delta\to0$ to obtain the solution stated in Theorem \ref{Thm1}.
For the rest of paper,
we investigate the solution mapping described in the above procedure.

\section{Inner solution}
\label{S6}
In this section,
we repeat the argument in Lemma 4.1 \cite{delPino} to define the mapping
$w(x,t)\in X_\delta\mapsto\epsilon(y,t)$ mentioned in Section \ref{S5.2}.
Throughout this section,
$\tilde{w}(x,t)\in C(\R^6\times[0,T])$
represents an extension of $w(x,t)$ defined in Section \ref{S5.2}.

\subsection{Notation 2}
\label{S6.1}
For the rest of this paper,
we assume that
 \begin{itemize}
 \item the constant $R>0$ is large enough and
 \item the blowup time $T$ is given by $T=e^{-R}$.
 \end{itemize}
Furthermore for convenience,
we write
 \[
 k_1\lesssim k_2
 \qquad(k_1,k_2>0),
 \]
if there is a universal constant $c>0$ independent of $R$, $\delta$
($\delta$ is defined in Section \ref{S5.2}) and $\mu$
($\mu$ is introduced in Section \ref{S6.3})
such that $k_1\leq ck_2$.

\subsection{Choice of $\lambda(t)$}
\label{S6.2}
We first define $\lambda(t)\in C^1([0,T])$ as a unique solution of
 \begin{equation}\label{6.1}
 \begin{cases}
 (G(\lambda(t),\tilde w(x,t)),\Lambda_y{\sf Q})_{L_y^2(\R^6)}=0
 \qquad\text{for } t\in(0,T),
 \\[2mm] \dis
 \lim_{t\to T}\left( \frac{\lambda}{\lambda_0} \right)
 =1.
 \end{cases}
 \end{equation}
We recall that $\lambda_0$ and $\sigma$ are defined by (see \eqref{5.2})
 \[
 \lambda_0
 =
 (T-t)^\frac{5}{4}\tau^{-\frac{15}{8}},
 \hspace{10mm}
 \sigma
 =
 -\left( \frac{5}{4}+\frac{15}{8\tau} \right)
 (T-t)^\frac{3}{2}\tau^{-\frac{15}{4}}.
 \]
 \begin{lem}\label{L6.1}
 It holds that
 \begin{align}\label{6.2}
 \dis
 \left| \lambda-\lambda_0 \right|
 &\lesssim
 \frac{\lambda_0}{\tau^\frac{1}{2}}
 \qquad \text{\rm for}\ t\in(0,T),
 \nonumber
 \\
 \dis
 \left|
 \lambda\dot\lambda
 -
 \left( \frac{\lambda}{\lambda_0} \right)^2\sigma
 \right|
 &\lesssim
 \frac{\lambda_0^2}{\tau^\frac{3}{2}(T-t)}
 \qquad \text{\rm for}\ t\in(0,T).
 \end{align}
 \end{lem}
\begin{proof}
From \eqref{5.7},
the relation \eqref{6.1} is rewritten as
 \[
 \begin{cases}
 \dis
 \left(
 \frac{\dot\lambda}{\lambda}
 -
 \frac{\sigma}{\lambda_0^2}
 \right)
 \|\Lambda_y{\sf Q}\|_{L_y^2(\R^6)}^2
 +
 (V\tilde w(x,t),\Lambda_y{\sf Q}))_{L_y^2(\R^6)}
 =
 0
 &
 \text{for } t\in(0,T),
 \\[4mm]
 \dis
 \lim_{t\to T}\left( \frac{\lambda}{\lambda_0} \right)
 =1.
 \end{cases}
 \]
We put $X(t)=\log\lambda(t)$ and $X(t)=X_0(t)+X_1(t)$ with $X_0=\log\lambda_0$.
Since $\dot X_0=\frac{\sigma}{\lambda_0^2}$,
it holds that
 \[
 \begin{cases}
 \dis
 \dot X_1
 \|\Lambda_y{\sf Q}\|_{L_y^2(\R^6)}^2
 +
 (V\tilde w(x,t),\Lambda_y{\sf Q}))_{L_y^2(\R^6)}
 =
 0
 &
 \text{for } t\in(0,T),
 \\[2mm]
 \dis
 X_1=0
 &
 \text{for } t=T.
 \end{cases}
 \]
Since $\tilde{w}(x,t)=0$ for $t\in(T-\frac{\delta}{2},T)$ (see \eqref{5.10}),
it is clear that $X_1(t)=0$ for $t\in(T-\frac{\delta}{2},T)$.
This implies $\lambda(t)=\lambda_0(t)$ for $t\in(T-\frac{\delta}{2},T)$.
Furthermore
due to \eqref{5.9} - \eqref{5.10},
the second term is computed as
 \begin{align*}
 \int_{\R^6}
 |V\tilde w \Lambda_y{\sf Q}|dy
 &=
 \int_{|z|<K\sqrt\tau}
 |V\tilde w \Lambda_y{\sf Q}|dy
 +
 \int_{|z|>K\sqrt\tau}
 |V\tilde w \Lambda_y{\sf Q}|dy
 \\
 &\lesssim
 \frac{(T-t)^{-1}}{\tau^\frac{3}{2}}
 \int_{|y|<\frac{K\sqrt\tau\sqrt{T-t}}{\lambda}}
 \frac{1+|z|^2}{1+|y|^8}dy
 +
 \frac{(T-t)^{-1}}{\tau^\frac{7}{16}}
 \int_{|y|>\frac{K\sqrt\tau\sqrt{T-t}}{\lambda}}
 \frac{1}{|z|^\frac{1}{8}}
 \frac{dy}{1+|y|^8}
 \\
 &\lesssim
 \frac{(T-t)^{-1}}{\tau^\frac{3}{2}}
 +
 \frac{\lambda^2(T-t)^{-2}}{\tau^\frac{3}{2}}
 \int_{|z|<K\sqrt\tau}
 \frac{|y|^2}{1+|y|^8}dy
 \\
 &\lesssim
 \left( 1+\lambda^2(T-t)^{-1} \right)
 \frac{(T-t)^{-1}}{\tau^\frac{3}{2}}.
 \end{align*}
From the above relations,
we easily see that there exists a unique solution $X_1(t)$
satisfying $|X_1(t)|\lesssim\tau^{-\frac{1}{2}}$.
This completes the proof.
\end{proof}
From Lemma \ref{L6.1},
it holds that
 \begin{equation}\label{6.3}
 \lambda
 =
 (1+o(1))(T-t)^\frac{5}{4}\tau^{-\frac{15}{8}},
 \hspace{10mm}
 \dot\lambda
 =
 (1+o(1))\left( -\frac{5}{4}-\frac{15}{8\tau} \right)
 (T-t)^{-1}\lambda.
 \end{equation}

\subsection{Construction of $\epsilon(y,t)$}
\label{S6.3}
Throughout this subsection,
$\lambda(t)$ represents the function given in Lemma \ref{L6.1}.
For simplicity,
we write
 \[
 G(t)
 =
 G(\lambda(t),\tilde w(x,t))
 =
 \left( \lambda\dot\lambda-\left( \frac{\lambda}{\lambda_0} \right)^2\sigma \right)
 (\Lambda_y{\sf Q})
 \chi_1
 +
 \lambda^2V\tilde w(x,t).
 \]
We define a radially symmetric function $g(|y|,t)$ on $\R^6\times(0,T)$ as a solution of
 \[
 \begin{cases}
 \dis
 -H_yg
 =
 G(t)
 \quad
 \text{in } \R^6,
 \\ \dis
 \lim_{|y|\to\infty}g(|y|,t)=0.
 \end{cases}
 \]
 Since
 $(G(t),\Lambda_y{\sf Q})_{L_y^2(\R^6)}=0$
 for $t\in(0,T)$ (see \eqref{6.1}),
 the radially symmetric solution $g(|y|,t)$ is given by
 (see Section \ref{S3.7})
 \[
 g(r,t)
 =
 \Gamma\int_0^r
 (\Lambda_y{\sf Q})
 G(t)r_1^5dr_1
 -
 (\Lambda_y{\sf Q})
 \int_0^r
 \Gamma G(t)r_1^5dr_1,
 \qquad r=|y|.
 \]
Therefore
we verify from \eqref{5.9} - \eqref{5.10} and \eqref{6.2} that
 \begin{align*}
 |G(t)|
 <
 \left| \lambda\dot\lambda-\left( \frac{\lambda}{\lambda_0} \right)^2\sigma \right|
 |\Lambda_y{\sf Q}|\chi_1
 +
 \lambda^\frac{n-2}{2}V|w|
 \lesssim
 \frac{\kappa}{1+|y|^4},
 \end{align*}
where $\kappa$ is defined by
 \begin{equation}\label{6.4}
 \kappa(t)=
 \frac{\lambda(t)^2}{\tau^\frac{3}{2}(T-t)}.
 \end{equation}
Since $(G(t),\Lambda_y{\sf Q})_{L_y^2(\R^6)}=0$ and
$|\Gamma(r)|\lesssim(1+r)^{-(n-2)}$ for $r>0$,
we get
 \begin{align}\label{6.5}
 |g(r,t)|
 &\lesssim
 \left|
 \Gamma\int_r^\infty
 (\Lambda_y{\sf Q})
 G(t)r_1^5dr_1
 \right|
 +
 \frac{\kappa}{1+|y|^2}
 \lesssim
 \frac{\kappa}{1+|y|^2},
 \qquad r=|y|.
 \end{align}
We introduce a new time variable $s$ defined by
 \[
 \frac{ds}{dt}
 =
 \frac{1}{\lambda(t)^2}, \qquad s(t)|_{t=0}=0.
 \]
From \eqref{6.3} - \eqref{6.4},
we see that
 \begin{align}\label{6.6}
 \frac{d\kappa}{ds}
 &=
 \frac{dt}{ds}\frac{d\kappa}{dt}
 =
 \frac{\lambda^4}{\tau^\frac{3}{2}(T-t)^2}
 \left( \frac{2\dot\lambda}{\lambda}(T-t)+1-\frac{3}{2\tau} \right)
 \nonumber
 \\
 &=
 -\left( \frac{3}{2}+\frac{21}{4\tau}+O(\tau^{-\frac{3}{2}}) \right)
 \frac{\lambda^4}{\tau^\frac{3}{2}(T-t)^2}
 \nonumber
 \\
 &=
 -(1+o(1))\left( \frac{3}{2}+\frac{21}{4\tau} \right)
 \tau^{-\frac{15}{4}}(T-t)^\frac{3}{2}\kappa.
\end{align}
From this relation,
it holds that for $s_2>s_1$
 \begin{align*}
 \kappa(s_2)-\kappa(s_1)
 &=
 \int_{s_1}^{s_2}
 \frac{d\kappa}{ds}ds
 >
 -2\tau_1^{-\frac{15}{4}}(T-t_1)^\frac{3}{2}\kappa(s_1)
 \int_{s_1}^{s_2}ds
 \\
 &>
 -T^\frac{3}{2}(s_2-s_1)\kappa(s_1),
 \end{align*}
where $t_1=t|_{s=s_1}$ and $\tau_1=\tau|_{t=t_1}$.
This implies
 \begin{equation}\label{6.7}
 \kappa(s_2)>\frac{1}{2}\kappa(s_1)
 \qquad\text{if } 0<s_2-s_1<\frac{1}{4T^\frac{3}{2}}.
 \end{equation}
Let $\mu_1^{(8R)}<0$ and $\psi_1^{(8R)}(y)\in H_0^1(B_{8R})$ be defined
in Section \ref{S3.3}.
A goal of this subsection is to construct a solution of
 \[
 \begin{cases}
 \pa_s E
 =
 H_yE+g\chi_{4R}
 &
 \text{in } B_{8R}\times(0,\infty),
 \\
 E=0
 &
 \text{on } \pa B_{8R}\times(0,\infty),
 \\
 \dis
 E=\frac{{\sf d}_\text{in}}{\mu_1^{(8R)}}\psi_1^{(8R)}
 &
 \text{for } s=0,
 \end{cases}
 \]
where $\chi_{4R}(y)=\eta(\frac{|y|}{4R})$.
The parameter ${\sf d}_\text{in}$ is determined below.
The desired solution $\epsilon(y,t)$ mentioned in Section \ref{S5.2}
is obtained by $\epsilon(y,t)=H_yE(y,t)$.
Let $M_1$ be the constant given in Lemma \ref{L3.3} and fix a large constant $M>M_1$
such that 
 \begin{equation}\label{6.8}
 |y|^2V(y)+|y|^\frac{7}{2}e^{-\sqrt{|\mu_1|}|y|}\ll1
 \qquad\text{for } |y|>2M.
 \end{equation}
To apply Lemma \ref{L3.5},
we consider the regularized problem.
 \[
 \begin{cases}
 \pa_s E^{(\mu)}
 =
 H_yE^{(\mu)}+g^{(\mu)}\chi_{4R}
 &
 \text{in } B_{8R}\times(0,\infty),
 \\
 E^{(\mu)}=0
 &
 \text{on } \pa B_{8R}\times(0,\infty),
 \\
 \dis
 E^{(\mu)}=\frac{{\sf d}_\text{in}^{(\mu)}}{\mu_1^{(8R)}}\psi_1^{(8R)}
 &
 \text{for } s=0,
 \end{cases}
 \]
where $g^{(\mu)}(y,s)\in C^\infty(\bar B_{8R}\times[0,\infty))$ is an approximation of
$g(y,s)$ satisfying
 \begin{itemize}
 \item $\dis g^{(\mu)}(y,s)=0$ \quad for $(y,s)\in B_{8R}\times(0,\mu)$,
 \item
 $\dis\lim_{\mu\to0}\sum_{i=0}^1
 \|\nabla_y^ig^{(\mu)}-\nabla_y^ig\|_{L^\infty(B_{8R}\times(0,s_1)}
 +\|\Delta_yg^{(\mu)}-\Delta_yg\|_{L^\infty(B_{8R}\times(0,s_1)}
 =0$
 \quad for any $s_1>0$,
 \item
 $\dis|\nabla_y^ig^{(\mu)}(y,s)|
 <\sup_{\max\{0,s-1\}<s'<s+1}\ \sup_{|y-y'|<1}|\nabla_y^ig(y',s')|$
 \quad for $(y,s)\in B_{8R}\times(0,\infty)$ \quad ($i=0,1$),
 \item
 $\dis|\Delta_yg^{(\mu)}(y,s)|<
 \sup_{\max\{0,s-1\}<s'<s+1}\ \sup_{|y-y'|<1}|\Delta_yg(y',s')|$
 \quad for $(y,s)\in B_{8R}\times(0,\infty)$.
 \end{itemize}
From this property, \eqref{6.5} and \eqref{6.7},
it holds that
 \begin{equation}\label{6.9}
 |g^{(\mu)}(y,s)|
 \lesssim
 \frac{1}{1+|y|^2}\sup_{\max\{0,s-1\}<s'<s+1}\kappa(s')
 \lesssim
 \frac{\kappa(s)}{1+|y|^2}.
 \end{equation}
In the same manner,
since $|\nabla_yg|=|\pa_rg|$ and $(\Delta_y+V(y))g=-G$,
we get
 \begin{equation}\label{6.10}
 |\nabla_yg^{(\mu)}(y,s)|
 \lesssim
 \frac{\kappa(s)}{1+|y|^3},
 \qquad
 |\Delta_yg^{(\mu)}(y,s)|
 \lesssim
 \frac{\kappa(s)}{1+|y|^4}.
 \end{equation}
We first consider
 \begin{equation}\label{6.11}
 \begin{cases}
 \pa_s E_1^{(\mu)}
 =
 \Delta_yE_1^{(\mu)}+(1-\chi_M)VE_1^{(\mu)}+g^{(\mu)}\chi_{4R}
 &
 \text{in } B_{8R}\times(0,\infty),
 \\
 E_1^{(\mu)}=0
 &
 \text{on } \pa B_{8R}\times(0,\infty),
 \\
 E_1^{(\mu)}=0
 &
 \text{for } s=0.
 \end{cases}
 \end{equation}
From a parabolic regularity theory,
we verify that
 \[
 E_1^{(\mu)}(y,s)\in C^\infty(\bar B_{8R}\times[0,\infty)).
 \]
We now investigate the regularity of the limiting function
$E_1=\lim_{\mu\to0}E_1^{(\mu)}$.
A parabolic $L^p$ theory (Theorem 7.32 p. 182 \cite{Lieberman}) gives
 \[
 \|E_1^{(\mu)}-E_1^{(\mu')}\|_{W_p^{2,1}(B_{8R}\times(0,s_1))}
 <
 c
 \|g^{(\mu)}\chi_{4R}-g^{(\mu')}\chi_{4R}\|_{L^p(B_{8R}\times(0,s_1))}
 \]
for any $1<p<\infty$ and $s_1>0$,
where
 \[
 \|E\|_{W_p^{2,1}(Q)}
 =\|E\|_{L^p(Q)}
 +\|\nabla_yE\|_{L^p(Q)}
 +\|D_y^2E\|_{L^p(Q)}
 +\|\pa_tE\|_{L^p(Q)}.
 \]
Furthermore we get from Lemma \ref{L3.5} (iii) that
 \[
 \|\nabla_yE_1^{(\mu)}-\nabla_yE_1^{(\mu')}\|_{L^\infty(B_{8R}\times(0,s_1))}
 <
 c
 \|g^{(\mu)}\chi_{4R}-g^{(\mu')}\chi_{4R}\|_{L^{p,q}(B_{8R}\times(0,s_1))}
 \]
for any $p,q\in(1,\infty]$ with $\frac{n}{p}+\frac{2}{q}<1$ and $s_1>0$.
Let
 \[
 e_1^{(\mu)}
 =
 \Delta_yE_1^{(\mu)}+(1-\chi_M)VE_1^{(\mu)}
 =
 H_y^{(M)}E_1^{(\mu)}.
 \]
This solves
 \begin{equation}\label{6.12}
 \begin{cases}
 \pa_s e_1^{(\mu)}
 =
 \Delta_ye_1^{(\mu)}+(1-\chi_M)Ve_1^{(\mu)}+H_y^{(M)}g^{(\mu)}\chi_{4R}
 &
 \text{in } B_{8R}\times(0,\infty),
 \\
 e_1^{(\mu)}=0
 &
 \text{on } \pa B_{8R}\times(0,\infty),
 \\
 e_1^{(\mu)}=0
 &
 \text{for } s=0.
 \end{cases}
 \end{equation}
In the same way as above,
we see that
 \[
 \|e_1^{(\mu)}-e_1^{(\mu')}\|_{W_p^{2,1}(B_{8R}\times(0,s_1))}
 <
 c
 \|H_y^{(M)}g^{(\mu)}\chi_{4R}-H_y^{(M)}g^{(\mu')}\chi_{4R}\|_{L^p(B_{8R}\times(0,s_1))}
 \]
for any $1<p<\infty$ and $s_1>0$,
and
 \[
 \|\nabla_ye_1^{(\mu)}-\nabla_ye_1^{(\mu')}\|_{L^\infty(B_{8R}\times(0,s_1))}
 <
 c
 \|H_y^{(M)}g^{(\mu)}\chi_{4R}-H_y^{(M)}g^{(\mu')}\chi_{4R}\|_{L^{p,q}(B_{8R}\times(0,s_1))}
 \]
for any $p,q\in(1,\infty]$ with $\frac{n}{p}+\frac{2}{q}<1$ and $s_1>0$.
Therefore
it holds that for any $p\in(1,\infty)$ and $s_1>0$
 \begin{align}\label{6.13}
 \begin{array}{cl}
 E_1^{(\mu)}\to E_1 & \text{in } W_p^{2,1}(B_{8R}\times(0,s_1)),
 \\[1mm]
 \Delta_yE_1^{(\mu)} \to \Delta_yE_1 & \text{in } W_p^{2,1}(B_{8R}\times(0,s_1)),
 \\[1mm]
 \nabla_yE_1^{(\mu)}\to\nabla_yE_1 & \text{in } L^\infty(B_{8R}\times(0,s_1)),
 \\
 \nabla_y\Delta_yE_1^{(\mu)}\to\nabla_y\Delta_yE_1 & \text{in }
 L^\infty(B_{8R}\times(0,s_1)).
 \end{array}
 \end{align}

\begin{lem}\label{L6.2}
 It holds that
 \begin{align*}
 |E_1^{(\mu)}(y,s)|
 &\lesssim
 \kappa(s)
 \log R
 \quad\text{\rm for}\ (y,s)\in B_{8R}\times(0,\infty),
 \\[1mm]
 |\nabla_yE_1^{(\mu)}(y,s)|
 &\lesssim
 \kappa(s)
 \frac{\log R}{1+|y|}
 \quad\text{\rm for}\ (y,s)\in B_{6R}\times(0,\infty).
 \end{align*}
\end{lem}

\begin{proof}
To construct a comparison function,
we put
 \[
 p(r)
 =
 p_1(r)\int_r^{16R}\frac{dr_1}{p_1(r_1)^2r_1^5}\int_0^{r_1}
 \frac{p_1(r_2)r_2^5}{1+r_2^2}dr_2,
 \qquad r=|y|,
 \]
where $p_1(r)$ is a radially symmetric function given in Lemma \ref{L3.3}.
The function $p(r)$ gives a positive radially symmetric solution of
 \[
 \begin{cases}
 \dis
 \Delta_yp+(1-\chi_M)Vp+\frac{1}{1+|y|^2}=0 & \text{in } B_{16R},
 \\
 p=0 & \text{on } \pa B_{16R}.
 \end{cases}
 \]
Since $k<p_1(r)<1$ for $r>0$ (see Lemma \ref{L3.3}),
there exists $k_1>0$ independent of $M$, $R$ such that
 \begin{equation}\label{6.14}
 0<p(r)<k_1\log R
 \qquad\text{for } r\in(0,16R).
 \end{equation}
We now check that $C_1\kappa(s)p(r)$ gives a super-solution of \eqref{6.11}
when $C_1\gg1$.
We see from \eqref{6.6} and \eqref{6.14} that
 \begin{align*}
 \left( \pa_s-\Delta_y-(1-\chi_M)V \right)
 \kappa p
 &=
 \frac{d\kappa}{ds}p
 +
 \frac{\kappa}{1+|y|^2}
 >
 \left(
 -2\tau^{-\frac{15}{4}}(T-t)^\frac{3}{2}p+\frac{1}{1+|y|^2}
 \right)
 \kappa
 \\
 &>
 \left(
 -k_1T^\frac{3}{2}\log R+\frac{1}{1+|y|^2}
 \right)
 \kappa.
 \end{align*}
Since $T=e^{-R}$,
if $R$ is large enough,
it holds that
 \begin{align*}
 k_1T^\frac{3}{2}\log R
 <
 \frac{k_1e^{-\frac{3}{2}R}\log R\cdot(1+(16R)^2)}{1+|y|^2}
 <
 \frac{\frac{1}{2}}{1+|y|^2}
 \qquad\text{for } |y|<16R.
 \end{align*}
Therefore
we obtain
 \begin{align*}
 (\pa_s-\Delta_y-(1-\chi_M)V)
 \kappa p
 >
 \frac{\kappa}{2(1+|y|^2)}
 \qquad\text{for } |y|<16R.
 \end{align*}
Since $|g^{(\mu)}|\lesssim\frac{\kappa}{1+|y|^2}$
(see \eqref{6.9}),
by a comparison argument,
we verify that there exists $C_1>0$ such that
 \begin{equation}\label{6.15}
 |E_1^{(\mu)}(y,s)|
 <
 C_1\kappa(s)p(r)
 <
 C_1k_1\kappa(s)\log R
 \qquad \text{for } (y,s)\in B_{8R}\times(0,\infty).
 \end{equation}
We next derive a gradient estimate.
From Lemma \ref{L3.5} (i) - (ii),
we deduce that
 \[
 \sup_{|y|<1}|\nabla_yE_1^{(\mu)}(y,s)|
 \lesssim
 \sup_{\max\{0,s-1\}<s'<s}
 \
 \sup_{|y'|<2}
 |E_1^{(\mu)}(y',s')|
 +
 \|g^{(\mu)}\chi_{4R}\|_{L^{p,q}(Q)},
 \]
where $Q=B_2\times(\max\{0,s-1\},s)$ and $p$, $q$ are as in Lemma \ref{L3.5}.
We apply \eqref{6.7}, \eqref{6.9} and \eqref{6.15} to get
 \begin{align*}
 \sup_{|y|<1}
 |\nabla_yE_1^{(\mu)}(y,s)|
 \lesssim
 \log R
 \sup_{\max\{0,s-1\}<s'<s}\kappa(s')
 \lesssim
 \kappa(s)
 \log R
 \qquad\text{for } s>0.
 \end{align*}
Next we fix $(y,s)\in(B_{6R}\setminus B_1)\times(0,\infty)$ and put
 \[
 r=|y|.
 \]
We consider two cases separately.
For the case (i) $s>\frac{r^2}{36}$,
we put
 \[
 \tilde E_1^{(\mu)}(Y,S)
 =
 E_1\left( y+\frac{rY}{6},s+\frac{r^2(S-1)}{36} \right),
 \qquad
 \tilde g^{(\mu)}(Y,S)
 =
 g^{(\mu)}
 \left( \left|y+\frac{rY}{6}\right|,s+\frac{r^2(S-1)}{36} \right).
 \]
The function $\tilde E_1^{(\mu)}(Y,s)$ is defined on $(Y,S)\in B_2\times(0,1)$
and satisfies
 \[
 \pa_S\tilde E_1^{(\mu)}
 =
 \Delta_Y\tilde E_1^{(\mu)}
 +
 \frac{r^2}{36}(1-\tilde\chi_M)\tilde V\tilde E_1^{(\mu)}
 +
 \frac{r^2}{36}
 \tilde g^{(\mu)}
 \tilde \chi_{4R}
 \qquad
 \text{in } B_2\times(0,1).
 \]
Since
$\frac{2}{3}r<|y+\frac{r}{3}Y|<\frac{4}{3}r$ for $|Y|<1$,
the potential term satisfies
 \[
 r^2(1-\tilde\chi_M)\tilde V
 \lesssim
 \frac{r^2}{1+r^4}
 \qquad\text{for } |Y|<1.
 \]
Therefore
Lemma \ref{L3.5} (i) gives
 \begin{align*}
 \sup_{\frac{1}{2}<S<1}\sup_{|Y|<1}|\nabla_Y\tilde E_1^{(\mu)}(Y,S)|
 &\lesssim
 \|\tilde E_1^{(\mu)}\|_{L_{(Y,S)}^\infty(B_2\times(0,1))}
 +
 \|r^2\tilde g^{(\mu)}\|_{L_{(Y,S)}^\infty(B_2\times(0,1))}.
 \end{align*}
From \eqref{6.9} and \eqref{6.15},
we note that
 \begin{align*}
 \|\tilde E_1^{(\mu)}\|_{L_{(Y,S)}^\infty(B_1\times(0,1))}
 +
 \|r^2\tilde g^{(\mu)}\|_{L_{(Y,S)}^\infty(B_1\times(0,1))}
 \lesssim
 (\log R+1)
 \sup_{s-\frac{r^2}{36}<s'<s}\kappa(s').
 \end{align*}
Since $\frac{r^2}{36}<6R^2=6(\log T)^2<\frac{1}{4}T^{-\frac{3}{2}}$ for $r<6R$,
it follows from \eqref{6.7} that
 \[
 \sup_{s-\frac{r^2}{36}<s'<s}\kappa(s')\lesssim\kappa(s).
 \]
Therefore
we obtain
 \[
 \sup_{\frac{1}{2}<S<1}\sup_{|Y|<1}|\nabla_Y\tilde E_1^{(\mu)}(Y,S)|
 \lesssim\kappa(s)\log R.
 \]
Since $\nabla_Y\tilde E_1^{(\mu)}(Y,S)
 =\frac{r}{6}\nabla_y\bar E_1^{(\mu)}(y+\frac{r}{6}Y,s+\frac{r^2(S-1)}{36})$,
we deduce that
 \[
 \frac{r}{6}|\nabla_yE_1^{(\mu)}(y,s)|
 =
 |\nabla_Y\tilde E_1^{(\mu)}(Y,S)|_{(Y,S)=(0,1)}|
 \lesssim
 \kappa(s)\log R.
 \]
This gives the desired estimate.
On the other hand,
for the case (ii) $0<s<\frac{r^2}{36}$,
we put
 \[
 \tilde E_1^{(\mu)}(Y,S)
 =
 E_1^{(\mu)}
 \left( y+\frac{rY}{6},\frac{r^2S}{36} \right),
 \qquad
 \tilde g^{(\mu)}(Y,S)
 =
 g^{(\mu)}
 \left( \left|y+\frac{rY}{6}\right|,\frac{r^2S}{36} \right).
 \]
The function $\tilde E_1^{(\mu)}(Y,S)$ is defined on $(Y,S)\in B_2\times(0,1)$
and satisfies
 \[
 \begin{cases}
 \dis
 \pa_S\tilde E_1^{(\mu)}
 =
 \Delta_Y\tilde E_1^{(\mu)}
 +
 \frac{r^2}{36}(1-\tilde\chi_M)\tilde V\tilde E_1^{(\mu)}
 +
 \frac{r^2}{36}
 \tilde g^{(\mu)}
 \tilde \chi_{4R}
 &
 \text{in } B_2\times(0,1),
 \\
 \tilde E_1^{(\mu)}=0
 &
 \text{for } S=0.
 \end{cases}
 \]
Since $\tilde E_1|_{S=0}=0$,
we get from Lemma \ref{L3.5} (ii) that
 \begin{align*}
 \sup_{0<S<1}\sup_{|Y|<1}|\nabla_Y\tilde E_1^{(\mu)}(Y,S)|
 &\lesssim
 \|\tilde E_1^{(\mu)}\|_{L_{(Y,S)}^\infty(B_2\times(0,1))}
 +
 \|r^2\tilde g^{(\mu)}\|_{L_{(Y,S)}^\infty(B_2\times(0,1))}.
 \end{align*}
Since $\frac{36}{r^2}s<1$ and $\frac{r^2}{36}<\frac{1}{4}T^{-\frac{3}{2}}$
we obtain
 \[
 \frac{r}{6}|\nabla_yE_1^{(\mu)}(y,s)|
 =
 |\nabla_Y\tilde E_1^{(\mu)}(Y,S)|_{(Y,S)=(0,\frac{36}{r^2}s)}|
 \lesssim
 \sup_{0<S<1}\kappa\left( \frac{r^2S}{36} \right)\log R
 \lesssim
 \kappa(s)\log R.
 \]
Therefore
the proof is completed.
\end{proof}

 \begin{lem}\label{L6.3}
 It holds that
 \begin{align*}
 |\Delta_yE_1^{(\mu)}(y,s)|
 &\lesssim
 \kappa(s)\left( \frac{1}{1+|y|^2}+\frac{\log R}{1+|y|^4} \right)
 \quad\text{\rm for}\ (y,s)\in B_{8R}\times(0,\infty),
 \\[1mm]
 |\nabla_y\Delta_yE_1^{(\mu)}(y,s)|
 &\lesssim
 \kappa(s)\left( \frac{1}{1+|y|^3}+\frac{\log R}{1+|y|^4} \right)
 \quad\text{\rm for}\ (y,s)\in B_{6R}\times(0,\infty).
 \end{align*}
\end{lem}

\begin{proof}
We set $e_1^{(\mu)}=\Delta_yE_1^{(\mu)}+(1-\chi_M)VE_1^{(\mu)}$.
This solves \eqref{6.12}.
We note from \eqref{6.9} - \eqref{6.10} that
\[
 |H_y^{(M)}g^{(\mu)}\chi_{4R}|
 =
 |(\Delta_y+(1-\chi_M)V)g^{(\mu)}\chi_{4R}|
 \lesssim
 \frac{\kappa}{1+|y|^4}.
 \]
Therefore
the argument in the proof of Lemma \ref{L6.2} shows
 \[
 |e_1^{(\mu)}(y,s)|
 \lesssim
 \frac{\kappa(s)}{1+|y|^2}
 \qquad\text{for } (y,s)\in B_{8R}\times(0,\infty).
 \]
Since
$|\Delta_yE_1^{(\mu)}|<|e_1^{(\mu)}|+|(1-\chi_M)VE_1^{(\mu)}|$,
this implies
 \begin{align*}
 |\Delta_yE_1^{(\mu)}|
 \lesssim
 \kappa(s)\left( \frac{1}{1+|y|^2}+\frac{\log R}{1+|y|^4} \right)
 \qquad \text{for } (y,s)\in B_{8R}\times(0,\infty).
 \end{align*}
Furthermore
we obtain a gradient estimate
in exactly the same way as in the proof of Lemma \ref{L6.2}.
 \begin{align*}
 \frac{|y|}{6}|\nabla_ye_1^{(\mu)}(y,s)|
 \lesssim
 \frac{\kappa(s)}{1+|y|^2}
 \qquad\text{for } (y,s)\in(B_{6R}\setminus B_1)\times(0,\infty).
 \end{align*}
This proves the second inequality in this lemma.
\end{proof}

Next we put
 \[
 E_2^{(\mu)}=E^{(\mu)}-E_1^{(\mu)}.
 \]
The function $E_2(y,s)$ solves
 \[
 \begin{cases}
 \pa_s E_2^{(\mu)}
 =
 H_yE_2^{(\mu)}+\chi_MVE_1^{(\mu)}
 &
 \text{in } B_{8R}\times(0,\infty),
 \\
 E_2^{(\mu)}=0
 &
 \text{on } \pa B_{8R}\times(0,\infty),
 \\
 \dis
 E_2^{(\mu)}=\frac{{\sf d}_\text{in}^{(\mu)}}{\mu_1^{(8R)}}\psi_1^{(8R)}
 &
 \text{for } s=0.
 \end{cases}
 \]
We recall that
$\mu_1^{(8R)}<0$ and $\psi_1^{(8R)}\in H_0^1(B_{8R})$ are defined
in Section \ref{S3.3}.
We take
 \begin{equation}\label{6.16}
 {\sf d}_\text{in}^{(\mu)}=-\mu_1^{(8R)}\int_{0}^\infty
 e^{\mu_1^{(8R)}s'}
 (\chi_MVE_1^{(\mu)}(s'),\psi_1^{(8R)})_{L_y^2(B_{8R})}ds'
 \end{equation}
and define ${\sf c}^{(\mu)}(s)$ by
 \[
 \begin{cases}
 \dis
 \frac{d{\sf c}^{(\mu)}}{ds}
 =
 -\mu_1^{(8R)}{\sf c}^{(\mu)}
 +
 (\chi_MVE_1^{(\mu)}(s),\psi_1^{(8R)})_{L_y^2(B_{8R})},
 \\
 \dis
 {\sf c}^{(\mu)}(0)=\frac{{\sf d}_\text{in}^{(\mu)}}{\mu_1^{(8R)}}.
 \end{cases}
 \]
The function ${\sf c}^{(\mu)}(s)$ is explicitly given by
 \begin{equation}\label{6.17}
 {\sf c}^{(\mu)}(s)=-\int_s^\infty e^{-\mu_1^{(8R)}(s-s')}
 (\chi_MVE_1^{(\mu)}(s'),\psi_1^{(8R)})_{L_y^2(B_{8R})}
 ds'.
 \end{equation}
Since $\frac{d\kappa}{ds}<0$ (see \eqref{6.6}) and $\mu_1^{(8R)}<\frac{\mu_1}{2}<0$,
we get from Lemma \ref{L3.1} and Lemma \ref{L6.2} that
 \begin{align}\label{6.18}
 |{\sf c}^{(\mu)}(s)|
 &\lesssim
 \kappa(s)\log R.
 \end{align}
We decompose $E_2^{(\mu)}(y,s)$ as
 \[
 E_2^{(\mu)}=\nu^{(\mu)}+{\sf c}^{(\mu)}(s)\psi_1^{(8R)}.
 \]
The function $\nu^{(\mu)}(y,s)$ satisfies
 \begin{equation}\label{6.19}
 \begin{cases}
 \pa_s\nu^{(\mu)}
 =
 H_y\nu^{(\mu)}
 +
 V\chi_ME_1^{(\mu)}
 -
 (\chi_MVE_1^{(\mu)},\psi_1^{(8R)})_{L_y^2(B_{8R})}
 \psi_1^{(8R)}
 &
 \text{in } B_{8R}\times(0,\infty),
 \\
 \nu^{(\mu)}=0
 &
 \text{on } \pa B_{8R}\times(0,\infty),
 \\
 \nu^{(\mu)}=0
 &
 \text{for } s=0.
 \end{cases}
 \end{equation}
Since ${\sf c}^{(\mu)}(s)\in C^\infty([0,\infty))$ 
and $E_1^{(\mu)}(y,s)=0$ for $s\in(0,\mu)$,
we note that $\nu^{(\mu)}(y,s)\in C^\infty(\bar B_{8R}\times[0,\infty))$.
Repeating the argument of the convergence of $E_1^{(\mu)}(y,s)$ (see \eqref{6.13}),
we get
 \begin{equation}\label{6.20}
 \nu^{(\mu)}\to \nu \qquad \text{in the same topology as \eqref{6.13}}.
 \end{equation}

\begin{lem}\label{L6.4}
 It holds that
 \begin{align*}
 |\nabla_y^i\nu^{(\mu)}(y,s)|
 &\lesssim
 \kappa(s)
 \frac{R^4\log R}{1+|y|^{\frac{7}{2}+i}}
 \qquad{\rm for}\ (y,s)\in B_{2R}\times(0,\infty)
 \qquad (i=0,1),
 \\
 |\nabla_y^i\Delta_y\nu^{(\mu)}(y,s)|
 &\lesssim
 \kappa(s)
 \frac{R^4\log R}{1+|y|^{\frac{11}{2}+i}}
 \qquad{\rm for}\ (y,s)\in B_{2R}\times(0,\infty)
 \qquad (i=0,1).
 \end{align*}
\end{lem}

\begin{proof}
Since $(\nu^{(\mu)}(s),\psi_1^{(8R)})_{L_y^2(B_{8R})}=0$ for $s\in(0,\infty)$,
from Lemma \ref{L3.2},
there exists $k>0$ such that
 \[
 (H_y\nu^{(\mu)}(s),\nu^{(\mu)}(s))_{L_y^2(B_{8R})}
 <
 -\frac{k}{R^4}\|\nu^{(\mu)}(s)\|_{L_y^2(B_{8R})}
 \qquad\text{for } s\in(0,\infty).
 \]
From this estimate and Lemma \ref{L6.2},
we get
 \begin{align*}
 e^{\frac{ks}{R^4}}\|\nu^{(\mu)}\|_{L_y^2(B_{8R})}^2
 &\lesssim
 R^4(\log R)^2\|\chi_MV\|_{L_y^2(B_{8R})}^2
 \int_0^se^{\frac{ks'}{R^4}}
 \kappa(s')^2
 ds'.
 \end{align*}
We compute the integral on the right-hand side.
From \eqref{6.6},
we observe that
 \begin{align*}
 \int_0^se^{\frac{ks'}{R^4}}
 \kappa(s')^2
 ds'
 &<
 \frac{R^4}{k}
 e^{\frac{ks}{R^4}}
 \kappa(s)^2
 -
 \frac{R^4}{k}
 \int_0^s
 e^{\frac{ks'}{R^4}}
 \frac{d}{ds'}
 \kappa(s')^2
 ds'
 \\
 &<
 \frac{R^4}{k}
 e^{\frac{ks}{R^4}}
 \kappa(s)^2
 +
 \frac{4R^4}{k}
 \int_0^s
 e^{\frac{ks'}{R^4}}
 \tau'^{-\frac{15}{4}}(T-t')^\frac{3}{2}
 \kappa(s')^2
 ds'
 \\
 &<
 \frac{R^4}{k}
 e^{\frac{ks}{R^4}}
 \kappa(s)^2
 +
 \frac{4R^4T^\frac{3}{2}}{k}
 \int_0^s
 e^{\frac{ks'}{R^4}}
 \kappa(s')^2
 ds'.
 \end{align*}
Since $T=e^{-R}$,
we deduce that
 \begin{align*}
 \int_0^se^{\frac{ks'}{R^4}}
 \kappa(s')^2
 ds'
 \lesssim
 R^4
 e^{\frac{ks}{R^4}}
 \kappa(s)^2.
 \end{align*}
This implies
 \[
 \|\nu^{(\mu)}(s)\|_{L_y^2(B_{8R})}
 \lesssim
 \kappa(s)R^4\log R.
 \]
Applying a local parabolic estimate in \eqref{6.19},
we get from \eqref{6.7} and Lemma \ref{L6.2} that
 \begin{align}\label{6.21}
 \|\nu^{(\mu)}(s)\|_{L_y^\infty(B_{6R})}
 &\lesssim
 \sup_{\max\{s-1,0\}<s'<s}\|\nu^{(\mu)}(s')\|_{L_y^2(B_{8R})}
 +
 \sup_{\max\{s-1,0\}<s'<s}\|VE_1(s')\|_{L_y^\infty(B_{2M})}
 \nonumber
 \\
 &\lesssim
 \kappa(s)R^4\log R.
 \end{align}
We now check that
$C_1M^\frac{7}{2}\kappa(s)R^4\log R|y|^{-\frac{7}{2}}$
gives a super solution for $y\in B_{8R}\setminus B_{2M}$ if $C_1\gg1$.
The computation is similar to the proof of
Lemma \ref{L6.2} and Lemma \ref{L6.3}.
We verify from \eqref{6.6} and \eqref{6.8} that
 \begin{align*}
 (\pa_s-H_y)
 \frac{\kappa R^4\log R}{|y|^\frac{7}{2}}
 &=
 \left(
 -2\tau^{-\frac{15}{4}}(T-t)^\frac{3}{2}
 +
 \frac{7}{4|y|^2}
 -
 V(y)
 \right)
 \frac{\kappa R^4\log R}{|y|^\frac{7}{2}}
 \\
 &>
 \left(
 -
 2\tau^{-\frac{15}{4}}T^\frac{3}{2}
 \frac{(8R)^2}{|y|^2}
 +
 \frac{7}{8|y|^2}
 \right)
 \frac{\kappa R^4\log R}{|y|^\frac{7}{2}}
  \\
 &>
 \frac{1}{2|y|^2}
 \frac{\kappa R^4\log R}{|y|^\frac{7}{2}}
 \qquad\text{for } y\in B_{8R}\setminus B_{2M}.
 \end{align*}
From Lemma \ref{L3.1} and Lemma \ref{L6.2},
the right-hand side of \eqref{6.19} is estimated as
 \begin{align*}
 \left|
 (V\chi_ME_1^{(\mu)},\psi_1^{(8R)})_{L_y^2(B_{8R})}
 \psi_1^{(8R)}(y)
 \right|
 \lesssim
 \frac{\kappa e^{-\sqrt{|\mu_1|}\cdot|y|}\log R}{|y|^\frac{5}{2}}
 \lesssim
 \frac{\kappa\log R}{|y|^\frac{7}{2}}
 \qquad\text{for } y\in B_{8R}\setminus B_{2M}.
 \end{align*}
Therefore
it holds that if $M\gg1$
 \begin{align*}
 (\pa_s-H_y)
 \frac{\kappa R^4\log R}{|y|^\frac{7}{2}}
 >
 (V\chi_ME_1,\psi_1^{(8R)})_{L_y^2(B_{8R})}
 \psi_1^{(8R)}
 \qquad\text{for } y\in B_{8R}\setminus B_{2M}.
 \end{align*}
From this estimate and \eqref{6.21},
a comparison argument shows
 \[
 |\nu^{(\mu)}(y,s)|
 <
 \frac{C_1M^\frac{7}{2}\kappa(s)R^4\log R}{|y|^\frac{7}{2}}
 \qquad\text{for } (y,s)\in B_{8R}\setminus B_{2M}\times(0,\infty).
 \]
By the same scaling argument as in the proof of Lemma \ref{L6.2},
we get
 \[
 |\nabla_y\nu^{(\mu)}(y,s)|
 \lesssim
 \frac{C_1M^\frac{7}{2}\kappa(s)R^4\log R}{|y|^\frac{9}{2}}
 \qquad\text{for } (y,s)\in B_{6R}\setminus B_{2M}\times(0,\infty).
 \]
Next we differentiate \eqref{6.19} with respect to $y_i$.
Then
$\nu_i^{(\mu)}(y,s)=\pa_{y_i}\nu^{(\mu)}(y,s)$ satisfies
 \[
 \pa_s\nu_i^{(\mu)}
 =
 H_y\nu_i^{(\mu)}
 +
 \pa_{y_i}V\cdot\nu^{(\mu)}
 -
 (\chi_MVE_1^{(\mu)},\psi_1^{(8R)})_{L_y^2(B_{8R})}
 \pa_{y_i}\psi_1^{(8R)}
 \]
for $(y,s)\in B_{8R}\setminus B_{2M}\times(0,\infty)$.
The same scaling argument as in the proof of Lemma \ref{L6.2} shows
 \begin{align*}
 |\nabla_y\nu_i^{(\mu)}(y,s)|
 &\lesssim
 \frac{CM^\frac{7}{2}\kappa(s)R^4\log R}{|y|^\frac{11}{2}}
 \qquad\text{for } (y,s)\in B_{4R}\setminus B_{2M}\times(0,\infty).
 \end{align*}
Furthermore
we repeat this procedure and obtain
 \begin{align*}
 |\nabla_y\Delta_y\nu^{(\mu)}(y,s)|
 &\lesssim
 \frac{CM^\frac{7}{2}\kappa(s)R^4\log R}{|y|^\frac{13}{2}}
 \qquad\text{for } (y,s)\in B_{2R}\setminus B_{2M}\times(0,\infty).
 \end{align*}
The proof is completed.
\end{proof}
We now put
 \[
 \epsilon^{(\mu)}
 =
 -H_yE^{(\mu)}
 =
 -H_y\left( E_1^{(\mu)}+E_2^{(\mu)} \right)
 =
 -H_y\left( E_1^{(\mu)}+\nu^{(\mu)}+{\sf c}^{(\mu)}(s)\psi_1^{(8R)} \right).
 \]
By definition of $E_1^{(\mu)}(y,s)$ and $\nu^{(\mu)}(y,s)$,
it satisfies
 \[
 \begin{cases}
 \pa_s\epsilon^{(\mu)}
 =
 H_y\epsilon^{(\mu)}
 -
 H_yg^{(\mu)}
 &
 \text{in } B_{2R}\times(0,\infty),
 \\
 \dis
 \epsilon^{(\mu)}={\sf d}_\text{in}^{(\mu)}\psi_1^{(8R)}
 &
 \text{for } s=0.
 \end{cases}
 \]
Furthermore
Lemma \ref{L6.2} - Lemma \ref{L6.4} and \eqref{6.18} imply
 \begin{align}\label{6.22}
 |\nabla_y^i\epsilon^{(\mu)}(y,s)|
 &\lesssim
 \kappa(s)\frac{R^4\log R}{1+|y|^{\frac{11}{2}+i}}
 \qquad\text{for } (y,s)\in B_{2R}\times(0,\infty)
 \qquad
 (i=0,1),
 \\
 |{\sf d}_\text{in}^{(\mu)}|
 &\lesssim
 \kappa(s)|_{s=0}\log R.
 \nonumber
 \end{align}
We note from \eqref{6.13} and \eqref{6.20} that
for any $p\in(1,\infty)$ and $s_1>0$
 \[
 \begin{array}{cl}
 \epsilon^{(\mu)}(y,t)
 \to
 \epsilon(y,t)
 & \text{in } W_p^{2,1}(B_{2R}\times(0,s_1)),
 \\[1mm]
 \nabla_y\epsilon^{(\mu)}(y,t)
 \to
 \nabla_y\epsilon(y,t)
 & \text{in } L^\infty(B_{8R}\times(0,s_1)).
 \end{array}
 \]
Therefore
since $-H_yg=G$ in $B_{2R}\times(0,\infty)$,
we conclude that
$\epsilon(y,t)$ gives a solution of
 \[
 \begin{cases}
 \pa_s\epsilon
 =
 H_y\epsilon
 +
 G(\lambda,\tilde{w})
 &
 \text{in } B_{2R}\times(0,\infty),
 \\
 \dis
 \epsilon={\sf d}_\text{in}\psi_1^{(8R)}
 &
 \text{for } s=0.
 \end{cases}
 \]
Furthermore
from \eqref{6.16} - \eqref{6.17} and \eqref{6.22},
it holds that
 \begin{align}\label{6.23}
 |\nabla_y^i\epsilon(y,s)|
 &\lesssim
 \kappa(s)\frac{R^4\log R}{1+|y|^{\frac{11}{2}+i}}
 \qquad\text{for } (y,s)\in B_{2R}\times(0,\infty)
 \qquad
 (i=0,1),
 \\
 |{\sf d}_\text{in}|
 &\lesssim
 \kappa(s)|_{s=0}\log R.
 \nonumber
 \end{align}

\section{Outer solution}
\label{S7}
In this section,
we solve a nonlinear problem (see \eqref{5.8}).
\begin{align*}
 W_t
 &=
 \Delta_xW
 +
 f'(-\Theta\chi_2)W
 \underbrace{
 +
 (1-\chi_\text{in})\frac{G(\lambda,\tilde w)}{\lambda^4}
 -
 \frac{V\Theta}{\lambda^2}\chi_2
 +
 \frac{\lambda\dot\lambda}{\lambda^4}
 (\Lambda_y{\sf Q})\chi_2}
 \\
 &\qquad
 \underbrace{
 +
 \frac{\lambda\dot\lambda}{\lambda^4}
 (\Lambda_y\epsilon)\chi_\text{in}
 +
 h_\text{in}
 }_{=F_\text{out}}
 +
 \mu
 +
 g_\text{out}
 +
 N_\text{out}
 \qquad
 \text{in } \R^6\times(0,T).
 \nonumber
 \end{align*}
For convenience,
we rewrite definition of each term on the right-hand side.
\begin{align*}
 G(\lambda,\tilde w)
 &=
 \left(
 \lambda\dot\lambda-\lambda_0^{-2}\lambda^2\sigma
 \right)
 (\Lambda_y{\sf Q})\chi_1
 +
 \lambda^2V\tilde w(x,t),
 \\
 h_\text{in}
 &=
 \lambda^{-4}
 \left(
 2\nabla_y\epsilon\cdot\nabla_y\chi_\text{in}+\epsilon\Delta_y\chi_\text{in}
 -
 \lambda^2\epsilon\pa_t\chi_\text{in}
 \right),
 \\
 \mu
 &=
 \alpha\tau^{-2}e_1
 \left( 1+\alpha\tau^{-1}e_1 \right)^{-2}
 e^{2\tau}
 -
 2\alpha^2\tau^{-2}|\nabla_ze_1|^2
 \left( 1+\alpha\tau^{-1}e_1 \right)^{-3}
 e^{2\tau},
 \\
 g_\text{out}
 &=
 (-\mu\chi_1)+(-g_0)+(-g_1)+g_2+g_3+g_4+g_5,
 \\[1mm]
 N_\text{out}
 &=
 f(\lambda^{-2}{\sf Q}+\lambda_0^{-2}\sigma T_1\chi_1
 -\Theta\chi_2+\lambda^{-2}\epsilon\chi_\text{in}+\tilde w)
 -
 f(-\Theta\chi_2)
 \\
 &\qquad
 -
 f'(-\Theta\chi_2)
 \left( \lambda^{-2}{\sf Q}+\lambda_0^{-2}\sigma T_1\chi_1
 +\lambda^{-2}\epsilon\chi_\text{in}+\tilde w \right)
 \\
 &\qquad
 -
 f(\lambda^{-2}{\sf Q})
 -
 f'(\lambda^{-2}{\sf Q})
 \left(
 \lambda_0^{-2}\sigma T_1\chi_1-\Theta\chi_2+\lambda^{-2}\epsilon\chi_\text{in}+\tilde w
 \right),
 \\
 g_0
 &=
 -\frac{\sigma}{\lambda_0^2}
 \left( 2\lambda_0^{-1}\dot\lambda_0
 +
 \lambda^{-1}\dot\lambda y\cdot\nabla_y
 \right)T_1
 \cdot\chi_1
 +
 \frac{\dot\sigma T_1}{\lambda_0^2}\chi_1
 +
 \frac{\sigma T_1}{\lambda_0^2}\pa_t\chi_1,
 \\
 g_1
 &=
 -\Theta\pa_t\chi_2,
 \\
 g_2
 &=
 \frac{2\sigma}{\lambda_0^2\lambda}\nabla_yT_1\cdot\nabla_x\chi_1
 +
 \frac{\sigma T_1}{\lambda_0^2}\Delta_x\chi_1,
 \\
 g_3
 &=
 -2\nabla_x\Theta\cdot\nabla_x\chi_2
 -
 \Theta\Delta_x\chi_2,
 \\
 g_4
 &=
 -f(-\Theta)\chi_2
 +
 f(-\Theta\chi_2),
 \\
 g_5
 &=
 f'(-\Theta\chi_2)
 \left( \lambda^{-2}{\sf Q}+\lambda_0^{-2}\sigma T_1\chi_1
 +\lambda^{-2}\epsilon\chi_\text{in} \right)
 \end{align*}
and
 \[
 \chi_1
 =
 \chi\left( \tau|z| \right), \qquad
 \chi_2
 =
 1-\chi_1, \qquad
 \chi_{\text{in}}
 =
 \chi\left( \frac{|y|}{R} \right).
 \]
A goal of this section is to show $W(x,t)\in X_\delta$ (see Section \ref{S5.2}).
We introduce a selfsimilar transformation.
 \[
 W(x,t)=e^\tau\varphi(z,\tau),
 \qquad
 z=\frac{x}{\sqrt{T-t}},
 \qquad
 T-t=e^{-\tau}.
 \]
The function $\varphi(z,\tau)$ solves
 \begin{equation}\label{7.1}
 \varphi_\tau
 =
 A_z\varphi
 +
 \left(
 1-\frac{2\alpha\tau^{-1}e_1+2\chi_1}{1+\alpha\tau^{-1}e_1}
 \right)
 \varphi
 +
 e^{-2\tau}
 \left(
 F_\text{out}
 +
 \mu
 +
 g_\text{out}
 +
 N_\text{out}
 \right)
 \qquad\text{in } \R^6\times(\tau_0,\infty),
 \end{equation}
where $\tau_0=-\log T$.
Throughout this section,
we assume that
$\lambda(t)$ and $\epsilon(y,t)$ are functions obtained in Section {\rm\ref{S6}}.

\subsection{Estimates of $F_\text{out}$, $g_\text{out}$ and $N_\text{out}$}
\label{S7.1}
Let ${\bf 1}_{z\in\Omega}(z)$ be a indicator function defined by
${\bf 1}_{z\in\Omega}(z)=1$ if $z\in\Omega$
and
${\bf 1}_{z\in\Omega}(z)=0$ if $z\not\in\Omega$.

 \begin{lem}\label{L7.1}
 It holds that
 \[
 |F_\text{\rm out}|
 \lesssim
 \begin{cases}
 \dis
 \frac{e^\tau}{\tau^\frac{3}{2}\lambda^2}
 \left(
 \frac{e^{-\tau}{\bf 1}_{|y|<2R}}{1+|y|^\frac{11}{2}}
 +
 \frac{\log |y|}{1+|y|^\frac{7}{2}}
 {\bf 1}_{R<|y|<2R}
 +
 \frac{{\bf 1}_{|y|>R}}{1+|y|^4}
 \right)
 +
 \tau^\frac{1}{4}e^{\frac{\tau}{2}}
 {\bf 1}_{\frac{1}{\tau}<|z|<\frac{2}{\tau}}
 &
 \text{\rm for}\ |z|<1,
 \\[4mm]
 \dis
 \tau^{-\frac{21}{4}}e^{\frac{\tau}{2}}\frac{1}{|z|^2}
 +
 \tau^{-\frac{15}{4}}e^{\frac{\tau}{2}}\frac{1}{|z|^4} &
 \text{\rm for}\ |z|>1.
 \end{cases}
 \]
 \end{lem}

 \begin{proof}
Since
$|G(\lambda,\tilde w)|\lesssim
 \frac{\lambda^2\tau^{-\frac{3}{2}}e^\tau+\lambda^2|\tilde w|}{1+|y|^4}$
(see \eqref{6.3}) and $|w(x,t)|<{\cal W}(x,t)$ (see \eqref{5.9} - \eqref{5.10}),
the first term in $F_\text{out}$ is estimated as
 \begin{align*}
 (1-\chi_\text{in})\frac{|G(\lambda,\tilde w)|}{\lambda^4}
 &\lesssim
 \left(
 \frac{e^\tau}{\tau^\frac{3}{2}\lambda^2}
 \frac{1}{|y|^4}
 +
 \frac{e^\tau}{\tau^\frac{3}{2}\lambda^2}
 \frac{1+|z|^2}{|y|^4}
 {\bf 1}_{|z|<K\sqrt\tau}
 +
 \frac{e^\tau}{\tau^\frac{7}{16}\lambda^2}
 \frac{1}{|y|^4}
 \frac{{\bf 1}_{|z|>K\sqrt\tau}}{|z|^\frac{1}{8}}
 \right)
 {\bf 1}_{|y|>R}
 \nonumber
 \\
 &\lesssim
 \frac{e^\tau}{\tau^\frac{3}{2}\lambda^2}
 \frac{{\bf 1}_{|y|>R}}{|y|^4}
 {\bf 1}_{|z|<1}
 +
 \frac{e^{3\tau}\lambda^2}{\tau^\frac{3}{2}}
 \frac{{\bf 1}_{1<|z|<K\sqrt\tau}}{|z|^2}
 +
 \frac{e^{3\tau}\lambda^2}{\tau^\frac{7}{16}}
 \frac{{\bf 1}_{|z|>K\sqrt\tau}}{|z|^\frac{33}{8}}
 \\
 &\lesssim
 \frac{e^\tau}{\tau^\frac{3}{2}\lambda^2}
 \frac{{\bf 1}_{|y|>R}}{|y|^4}
 {\bf 1}_{|z|<1}
 +
 \frac{e^{\frac{\tau}{2}}}{\tau^\frac{21}{4}}
 \frac{{\bf 1}_{1<|z|<K\sqrt\tau}}{|z|^2}
 +
 K^{-\frac{1}{8}}
 \frac{e^{\frac{\tau}{2}}}{\tau^\frac{17}{4}}
 \frac{{\bf 1}_{|z|>K\sqrt\tau}}{|z|^4}.
 \end{align*}
We next compute the second and the third term.
 \begin{align*}
 \frac{V\Theta}{\lambda^2}\chi_2
 +
 \frac{|\lambda\dot\lambda|}{\lambda^4}|\Lambda_y{\sf Q}|\chi_2
 &\lesssim
 \frac{e^\tau}{\lambda^2}\frac{1}{|y|^4}{\bf 1}_{|z|>\frac{1}{\tau}}
 \lesssim
 \tau^4e^{3\tau}\lambda^2{\bf 1}_{\frac{1}{\tau}<|z|<1}
 +
 e^{3\tau}\lambda^2
 \frac{{\bf 1}_{|z|>1}}{|z|^4}
 \nonumber
 \\
 &\lesssim
 \tau^\frac{1}{4}e^{\frac{\tau}{2}}
 {\bf 1}_{\frac{1}{\tau}<|z|<1}
 +
 \frac{e^{\frac{\tau}{2}}}{\tau^\frac{15}{4}}
 \frac{{\bf 1}_{|z|>1}}{|z|^4}.
 \end{align*}
As for the fourth  and the fifth term,
we get from \eqref{6.23} that
 \begin{align*}
 \frac{|\lambda\dot\lambda|}{\lambda^4}|\Lambda_y\epsilon|\chi_\text{in}
 &\lesssim
 \frac{e^\tau\kappa}{\lambda^2}
 \frac{R^4\log R}{1+|y|^\frac{11}{2}}
 {\bf 1}_{|y|<2R}
 \lesssim
 \frac{e^\tau}{\tau^\frac{3}{2}\lambda^2}
 \frac{e^{-\frac{3}{2}\tau}\tau^{-\frac{15}{4}}R^4\log R}{1+|y|^\frac{11}{2}}
 {\bf 1}_{|y|<2R}
 \nonumber
 \\
 &\lesssim
 \frac{e^\tau}{\tau^\frac{3}{2}\lambda^2}
 \frac{e^{-\tau}}{1+|y|^\frac{11}{2}}
 {\bf 1}_{|y|<2R},
 \\[1mm]
 |h_\text{in}|
 &\lesssim
 \frac{1}{\lambda^4}
 \left(
 \frac{|\nabla_y\epsilon|}{R}
 +
 \frac{|\epsilon|}{R^2}
 +
 |\lambda\dot\lambda\epsilon|
 \right)
 {\bf 1}_{R<|y|<2R}
 \\
 &\lesssim
 \frac{e^\tau}{\tau^\frac{3}{2}\lambda^2}
 \left(
 \frac{\log R}{R^\frac{7}{2}}
 +
 \frac{\log R}{R^\frac{7}{2}}
 +
 \frac{e^{-\frac{3}{2}\tau}}{\tau^\frac{15}{4}}
 \frac{\log R}{R^\frac{3}{2}}
 \right) 
 {\bf 1}_{R<|y|<2R}.
 \end{align*}
The proof is completed.
\end{proof}

 \begin{lem}\label{L7.2}
 It holds that
 \[
 |g_\text{\rm out}|
 \lesssim
 \begin{cases}
 e^{2\tau}
 {\bf 1}_{|z|<\frac{2}{\tau}}
 +
 \tau^\frac{1}{4}e^{\frac{\tau}{2}}
 {\bf 1}_{\frac{1}{\tau}<|z|<1}
 &
 \text{\rm for}\ |z|<1,
 \\[2mm]
 \dis
 \tau^{-\frac{15}{4}}
 e^{\frac{\tau}{2}}
 \frac{1}{|z|^4}
 &
 \text{\rm for}\ |z|>1.
 \end{cases}
 \]
 \end{lem}

\begin{proof}
The first term is estimated as
 \begin{align*}
 |\mu|\chi_1
 &\lesssim
 \tau^{-2}e^{2\tau}
 {\bf 1}_{|z|<\frac{2}{\tau}}.
 \end{align*}
By a direct computation,
we get
 \begin{align*}
 |(-g_0)&+(-g_1)+g_2+g_3|
 \lesssim
 \left|
 \frac{\sigma}{\lambda_0^2}
 \left( \frac{2\dot\lambda_0}{\lambda_0}
 +
 \frac{\dot\lambda}{\lambda}y\cdot\nabla_y
 \right)T_1
 \cdot\chi_1
 -
 \frac{\dot\sigma T_1}{\lambda_0^2}\chi_1
 -
 \frac{\sigma T_1}{\lambda_0^2}\pa_t\chi_1
 \right.
 \nonumber
 \\
 &
 \left.
 +\Theta\pa_t\chi_2
 +
 \frac{2\sigma}{\lambda_0^2\lambda}\nabla_yT_1\cdot\nabla_x\chi_1
 +
 \frac{\sigma T_1}{\lambda_0^2}\Delta_x\chi_1
 -
 2\nabla_x\Theta\cdot\nabla_x\chi_2
 -
 \Theta\Delta_x\chi_2
 \right|
 \nonumber
 \\
 &\lesssim
 e^{2\tau}
 {\bf 1}_{|z|<\frac{2}{\tau}}
 +
 \left|
 \frac{\sigma T_1}{\lambda_0^2}
 +
 \Theta
 \right|
 |\pa_t\chi_1|
 +
 \left|
 \frac{2\sigma}{\lambda_0^2\lambda}\nabla_yT_1
 +
 2\nabla_x\Theta
 \right|
 |\nabla_x\chi_1|
 +
 \left|
 \frac{\sigma T_1}{\lambda_0^2}
 +
 \Theta
 \right|
 |\Delta_x\chi_1|.
 \end{align*}
From \eqref{4.2} - \eqref{4.3}, \eqref{4.7} and \eqref{5.2},
we verify that
 \begin{align*}
 \left|
 \frac{\sigma T_1}{\lambda_0^2}+\Theta
 \right|
 &=
 \left|
 -\left( \frac{4}{5}+O\left( \frac{1}{|y|^2} \right) \right)
 \left( \frac{5}{4}+\frac{15}{8\tau} \right)
 e^\tau
 +
 \left(
 1+\frac{3}{2\tau}
 +
 O\left( \frac{|z|^2}{\tau} \right)
 +
 O\left( \frac{1}{\tau^2} \right)
 \right)e^\tau
 \right|
 \\
 &=
 \left|
 O\left( \frac{1}{|y|^2} \right)
  +
 O\left( \frac{|z|^2}{\tau} \right)
 +
 O\left( \frac{1}{\tau^2} \right)
 \right|
 e^\tau
 \lesssim
 \tau^{-2}e^\tau
 \qquad\text{for } \tau^{-1}<|z|<2\tau^{-1},
 \\
 \frac{\sigma}{\lambda^3}|\nabla_yT_1|
 &\lesssim
 \frac{\sigma}{\lambda^3}
 \frac{1}{|y|^3}
 \lesssim
 \tau^3e^{\frac{3}{2}\tau}\sigma
 =
 \tau^{-\frac{3}{4}}
 \qquad\text{for } \tau^{-1}<|z|<2\tau^{-1},
 \\
 |\nabla_x\Theta|
 &\lesssim
 \tau^{-1}e^{\frac{3}{2}\tau}|\nabla_ze_1|
 \lesssim
 \tau^{-1}e^{\frac{3}{2}\tau}|z|
 \lesssim
 \tau^{-2}e^{\frac{3}{2}\tau}
 \qquad\text{for } \tau^{-1}<|z|<2\tau^{-1}.
 \end{align*}
Since $|\pa_t\chi_1|\lesssim e^\tau$,
$|\nabla_x\chi_1|\lesssim \tau e^{\frac{\tau}{2}}$ and
$|\Delta_x\chi_1|\lesssim \tau^2e^\tau$,
combining these estimates,
we obtain
 \begin{align*}
 |(-g_0)+(-g_1)+g_2+g_3|
 \lesssim
 e^{2\tau}
 {\bf 1}_{|z|<\frac{2}{\tau}}.
 \end{align*}
The term $g_4$ is easily estimated as
 \[
 |g_4|\lesssim
 e^{2\tau}{\bf 1}_{|z|<\frac{2}{\tau}}.
 \]
We finally estimate $g_5$.
Since $\chi_\text{2}\chi_\text{in}=0$,
we see that
 \begin{align*}
 |g_5|
 &\lesssim
 \Theta\chi_2
 \left|
 \frac{1}{\lambda^2}\frac{1}{|y|^4}
 +
 \frac{\sigma T_1\chi_1}{\lambda_0^2}
 \right|
 \lesssim
 e^{3\tau}\lambda^2
 \frac{1}{|z|^4}
 {\bf 1}_{|z|>\frac{1}{\tau}}
 +
 e^{2\tau}
 {\bf 1}_{\frac{1}{\tau}<|z|<\frac{2}{\tau}}
 \nonumber
 \\
 &\lesssim
 \tau^\frac{1}{4}e^{\frac{\tau}{2}}
 {\bf 1}_{\frac{1}{\tau}<|z|<1}
 +
 \frac{e^{\frac{\tau}{2}}}{\tau^\frac{15}{4}}
 \frac{1}{|z|^4}
 {\bf 1}_{|z|>1}
 +
 e^{2\tau}
 {\bf 1}_{\frac{1}{\tau}<|z|<\frac{2}{\tau}}.
\end{align*}
The proof is completed.
\end{proof}

We next estimate $N_\text{out}$.
\begin{lem}\label{L7.3}
 It holds that
 \[
 |N_\text{\rm out}|
 \lesssim
 \begin{cases}
 \dis
 \tau^{-3}e^{2\tau}
 \frac{R^8(\log R)^2}{1+|y|^{11}}
 {\bf 1}_{|y|<2R}
 +
 e^{2\tau}{\bf 1}_{|z|<\frac{2}{\tau}}
 +
 \tau^{-3}e^{2\tau}
 &
 \text{\rm for}\ |z|<1,
 \\[2mm]
 \dis
 \tau^{-3}e^{2\tau}|z|^4
 &
 \text{\rm for}\ 1<|z|<K\sqrt\tau,
 \\[1mm]
 \dis
 \tau^{-\frac{7}{8}}e^{2\tau}
 \frac{1}{|z|^\frac{1}{4}}
 &
 \text{\rm for}\ |z|>K\sqrt\tau.
 \end{cases}
 \]
 \end{lem}
 \begin{proof}
We first note that the following  elementary relation.
 \[
 |f(a+b)-f(a)-f'(a)b|\lesssim b^2 \qquad \text{for } a,b\in\R.
 \]
Since $\chi_2=0$ for $|z|<\frac{1}{2\tau}$,
we get for $|z|<\frac{1}{2\tau}$
 \begin{align*}
 |N_\text{out}|
 &\lesssim
 \left|
 f\left(
 \frac{{\sf Q}}{\lambda^2}
 +
 \frac{\sigma T_1\chi_1}{\lambda_0^2}
 +
 \frac{\epsilon\chi_\text{in}}{\lambda^2}+\tilde w
 \right)
 -
 f\left( \frac{{\sf Q}}{\lambda^2} \right)
 -
 f'\left( \frac{{\sf Q}}{\lambda^2} \right)
 \left(
 \frac{\sigma T_1\chi_1}{\lambda_0^2}
 +
 \frac{\epsilon\chi_\text{in}}{\lambda^2}+\tilde w
 \right)
 \right|
 \\
 &\lesssim
 \left(
 \frac{\sigma T_1\chi_1}{\lambda_0^2}
 +
 \frac{\epsilon\chi_\text{in}}{\lambda^2}+\tilde w
 \right)^2
 \\
 &\lesssim
 e^{2\tau}{\bf 1}_{|z|<\frac{1}{2\tau}}
 +
 \frac{e^{2\tau}}{\tau^3}
 \frac{R^8(\log R)^2}{1+|y|^{11}}
 {\bf 1}_{|y|<2R}
 +
  \tau^{-3}e^{2\tau}
 {\bf 1}_{|z|<\frac{1}{2\tau}}.
 \end{align*}
Furthermore
since $\chi_\text{in}=0$ for $|z|>\frac{1}{2\tau}$,
it holds that for $\frac{1}{2\tau}<|z|<K\sqrt\tau$
 \begin{align*}
 |N_\text{out}|
 &\lesssim
 \left|
 f\left(
 \frac{{\sf Q}}{\lambda^2}
 +
 \frac{\sigma T_1\chi_1}{\lambda_0^2}
 -
 \Theta\chi_2
 +
 \tilde w
 \right)
 -
 f(-\Theta\chi_2)
 -
 f'(-\Theta\chi_2)
 \left(
 \frac{{\sf Q}}{\lambda^2}
 +
 \frac{\sigma T_1\chi_1}{\lambda_0^2}
 +
 \tilde w
 \right)
 \right|
 \\
 &\qquad
 +
 f\left(
 \frac{{\sf Q}}{\lambda^2}
 \right)
 +
 f'\left(
 \frac{{\sf Q}}{\lambda^2}
 \right)
 \left|
 \frac{\sigma T_1\chi_1}{\lambda_0^2}-\Theta\chi_2+\tilde w
 \right|
 \\
 &\lesssim
 \left(
 \frac{{\sf Q}}{\lambda^2}
 +
 \frac{\sigma T_1\chi_1}{\lambda_0^2}
 +
 \tilde w
 \right)^2
 +
 \left(
 \frac{{\sf Q}}{\lambda^2}
 \right)^2
 +
 \frac{{\sf Q}}{\lambda^2}
 \left|
 \frac{\sigma T_1\chi_1}{\lambda_0^2}-\Theta\chi_2+\tilde w
 \right|
 \\
 &\lesssim
 \frac{1}{\lambda^4}\frac{1}{|y|^8}
 +
 e^{2\tau}
 {\bf 1}_{|z|<\frac{2}{\tau}}
 +
 \tilde w^2
 +
 \frac{1}{\lambda^2}\frac{1}{|y|^4}
 \left(
 e^\tau{\bf 1}_{|z|<\frac{2}{\tau}}+e^\tau+|\tilde w|
 \right)
 \\
 &\lesssim
 e^{4\tau}\lambda^4
 \frac{1}{|z|^8}
 +
 e^{2\tau}{\bf 1}_{|z|<\frac{2}{\tau}}
 +
 \tau^{-3}e^{2\tau}
 (1+|z|^4)
 +
 e^{3\tau}\lambda^2
 \frac{1}{|z|^4}
 \\
 &\lesssim
 \tau^\frac{1}{2}e^{-\tau}
 +
 e^{2\tau}{\bf 1}_{|z|<\frac{2}{\tau}}
 +
 \tau^{-3}e^{2\tau}
 (1+|z|^4)
 +
 \tau^{\frac{1}{4}}
 e^\frac{\tau}{2}.
 \end{align*}
We finally investigate the case $|z|>K\sqrt\tau$.
 \begin{align*}
 |N_\text{out}|
 &\lesssim
 \left|
 f\left(
 \frac{{\sf Q}}{\lambda^2}
 -
 \Theta\chi_2
 +
 \tilde w
 \right)
 -
 f(-\Theta\chi_2)
 -
 f'(-\Theta\chi_2)
 \left(
 \frac{{\sf Q}}{\lambda^2}
 +
 \tilde w
 \right)
 \right|
 \\
 &\qquad
 +
 f\left(
 \frac{{\sf Q}}{\lambda^2}
 \right)
 +
 f'\left(
 \frac{{\sf Q}}{\lambda^2}
 \right)
 \left|
 -\Theta\chi_2+\tilde w
 \right|
 \\
 &\lesssim
 \left(
 \frac{{\sf Q}}{\lambda^2}
 +
 \tilde w
 \right)^2
 +
 \left(
 \frac{{\sf Q}}{\lambda^2}
 \right)^2
 +
 \frac{{\sf Q}}{\lambda^2}
 \left|
 -\Theta\chi_2+\tilde w
 \right|
 \\
 &\lesssim
 \frac{1}{\lambda^4}\frac{1}{|y|^8}
 +
 |\tilde w|^2
 +
 \frac{1}{\lambda^2}\frac{1}{|y|^4}
 \left(
 \tau e^\tau
 \frac{1}{|z|^2}
 +
 \frac{e^\tau}{\tau^\frac{7}{16}}
 \frac{1}{|z|^\frac{1}{8}}
 \right)
 \lesssim
 \frac{e^{2\tau}}{\tau^\frac{7}{8}}
 \frac{1}{|z|^\frac{1}{4}}.
 \end{align*}
The proof is completed.
 \end{proof}

From Lemma \ref{L7.1} - Lemma \ref{7.3},
we obtain $L_\rho^2(\R^6)$ estimates of
$F_\text{out}+\mu+g_\text{out}+N_\text{out}$.
 \begin{lem}\label{L7.4}
 It holds that
 \begin{align*}
 \|e^{-2\tau}F_\text{\rm out}\|_\rho
 +
 \|e^{-2\tau}g_\text{\rm out}\|_\rho
 +
 \|e^{-2\tau}N_\text{\rm out}\|_\rho
 +
 e^{-2\tau}|(\mu,e_1)_\rho|
 +
 \|e^{-2\tau}\mu^\bot\|_\rho 
 &\lesssim
 \tau^{-3},
 \end{align*}
where $\mu^\bot=\mu-(\mu,e_0)_\rho e_0-(\mu,e_1)_\rho e_1$.
 \end{lem}
 \begin{proof}
 From Lemma \ref{L7.1} and $|z|=e^\frac{\tau}{2}\lambda|y|$,
 we see that
 \begin{align*}
 \|e^{-2\tau}F_\text{\rm out}\|_\rho
 \lesssim
 \frac{e^{-\tau}}{\tau^\frac{3}{2}\lambda^2}
 \left( \frac{\lambda}{\sqrt{T-t}} \right)^3
 +
 \tau^\frac{1}{4}e^{-\frac{3}{2}\tau}
 +
 \tau^{-\frac{15}{4}}e^{-\frac{3}{2}\tau}
 \lesssim
 \tau^{-3}.
 \end{align*}
Since $\|{\bf 1}_{|z|<\frac{1}{\tau}}\|_\rho\lesssim\tau^{-\frac{n}{2}}$,
we get from Lemma \ref{L7.2} that
 \begin{align*}
 \|e^{-2\tau}g_\text{\rm out}\|_\rho
 \lesssim
 \tau^{-\frac{n}{2}}
 +
 \tau^\frac{1}{4}e^{-\frac{3}{2}\tau}
 +
 \tau^{-\frac{15}{4}}e^{-\frac{3}{2}\tau}
 \lesssim
 \tau^{-3}.
 \end{align*}
In the  same manner,
from Lemma \ref{L7.3},
we verify that
 \begin{align*}
 \|e^{-2\tau}N_\text{\rm out}\|_\rho
 \lesssim
 R^8(\log R)^2\tau^{-3}
 \left( \frac{\lambda}{\sqrt{T-t}} \right)^3
 +
 \tau^{-\frac{n}{2}}
 +
 \tau^{-3}
 +
 \tau^{-\frac{7}{8}}
 \|{\bf 1}_{|z|>K\sqrt\tau}\|_\rho
 \lesssim
 \tau^{-3}.
 \end{align*}
Finally we compute the last two terms.
We recall that
 \[
 \mu(z,\tau)
 =
 \frac{\alpha}{\tau^2}e_1
 \left( 1+\frac{\alpha}{\tau}e_1 \right)^{-2}
 e^{2\tau}
 -
 \frac{2\alpha^2}{\tau^2}|\nabla_ze_1|^2
 \left( 1+\frac{\alpha}{\tau}e_1 \right)^{-3}
 e^{2\tau}.
 \]
Since
$(e_1^2,e_1)_\rho=8{\sf c}_1$ (see \eqref{3.5}) and
$(|\nabla_ze_1|^2,e_1)_\rho=4{\sf c}_1$ (see \eqref{3.6}),
it holds that
 \[
 \alpha
 -
 2\alpha^2(|\nabla_ze_1|^2,e_1)_\rho=0,
 \quad\text{where }
 \alpha=\frac{1}{(e_1^2,e_1)_\rho}.
 \]
From this relation,
we easily see that $|(\mu,e_1)_\rho|\lesssim\tau^{-3}$.
Furthermore since $|\nabla_ze_1|^2\in$ span$\{e_0,e_1\}$,
 we obtain $\|e^{-2\tau}\mu^\bot\|_\rho\lesssim\tau^{-3}$.
 The proof is completed.
 \end{proof}

\subsection{Choice of parameters}
\label{S7.2}
We rewrite \eqref{7.1} in a different form.
 \begin{align}\label{7.2}
 \varphi_\tau
 &=
 A_z\varphi
 +
 \varphi
 -
 \sum_{i=0}^1
 \frac{2\alpha\tau^{-1}e_1+2\chi_1}{1+\alpha\tau^{-1}e_1}
 (\varphi,e_i)_\rho e_i
 -
 \frac{2\alpha\tau^{-1}e_1+2\chi_1}{1+\alpha\tau^{-1}e_1}
 \varphi^\bot
 \\
 &\qquad
 +
 e^{-2\tau}F_\text{out}
 +
 e^{-2\tau}\mu
 +
 e^{-2\tau}g_\text{out}
 +
 e^{-2\tau}N_\text{out}.
 \nonumber
 \end{align}
The symbol $^\bot$ represents the projection to the orthogonal complement of
span$\{e_0,e_1\}$ in $L_\rho^2(\R^6)$.
We solve this problem under the initial condition:
 \begin{equation}\label{7.3}
 \varphi(\tau_0)
 =
 ({\bf d}\cdot {\bf e})\chi_\text{out},
 \qquad
 \chi_\text{out}(z)=\eta\left( \frac{4|z|}{K\sqrt{\tau}} \right).
 \end{equation}
The parameter ${\bf d}=(d_0,d_1)$ is chosen later.
To construct a solution of this problem,
we here consider an auxiliary problem.
 \[
 \begin{cases}
 \dis
 \Phi_\tau
 =
 A_z\Phi
 +
 \Phi
 -
 \sum_{i=0}^1
 \frac{2\alpha\tau^{-1}e_1+2\chi_1}{1+\alpha\tau^{-1}e_1}
 (\Phi,e_i)_\rho e_i
 -
 \frac{2\alpha\tau^{-1}e_1+2\chi_1}{1+\alpha\tau^{-1}e_1}
 \nu
 \\[4mm]
 \hspace{10mm}
 +
 e^{-2\tau}F_\text{out}
 +
 e^{-2\tau}\mu
 +
 e^{-2\tau}g_\text{out}
 +
 e^{-2\tau}N_\text{out}
 & \text{in } \R^6\times(\tau_0,\infty),
 \\[2mm]
 \Phi(\tau_0)
 =
 ({\bf d}\cdot {\bf e})\chi_\text{out},
 \end{cases}
 \]
where $\nu(z,\tau)$ is a given function.
Then $a_j=(\Phi,e_j)_\rho$ satisfies
 \begin{align*}
 \dot a_j
 &=
 (1-j)
 a_j
 -
 \sum_{i=0}^1
 a_i
 \left(
 \frac{2\alpha\tau^{-1}e_1+2\chi_1}{1+\alpha\tau^{-1}e_1}e_i,e_j
 \right)_\rho
 \underbrace{
 -
 \left(
 \frac{2\alpha\tau^{-1}e_1+2\chi_1}{1+\alpha\tau^{-1}e_1}\nu,e_j
 \right)_\rho}_{=:q_j}
 \\
 &\qquad
 +
 \underbrace{
 (e^{-2\tau}F_\text{out},e_j)_\rho
 +
 (e^{-2\tau}\mu,e_j)_\rho
 +
 (e^{-2\tau}g_\text{out},e_j)_\rho
 +
 (e^{-2\tau}N_\text{out},e_j)_\rho}_{=:f_j}.
 \end{align*}
This ODE system is rewritten as
 \begin{equation}\label{7.4}
 \dot{\bf a}
 =
 B{\bf a}
 +
 {\bf q}
 +
 {\bf f},
 \end{equation}
where
 \[
 {\bf a}
 =
 \begin{pmatrix}
 a_0\\
 a_1
 \end{pmatrix},
 \qquad
 {\bf q}
 =
 \begin{pmatrix}
 q_0\\
 q_1
 \end{pmatrix},
 \qquad
 {\bf f}
 =
 \begin{pmatrix}
 f_0\\
 f_1
 \end{pmatrix},
 \qquad
 B=
 \begin{pmatrix}
 B_{11} & B_{12}
 \\
 B_{21} & B_{22}
 \end{pmatrix}.
 \]
Since $|(\chi_1e_i,e_j)_\rho|\lesssim\tau^{-6}$ and
$\alpha=\frac{1}{(e_1^2,e_1)_\rho}$,
it holds that
 \[
 B=
 \begin{pmatrix}
 1+\Delta_{11} & D\tau^{-1}+\Delta_{12}
 \\
 D\tau^{-1}+\Delta_{21} & -2\tau^{-1}+\Delta_{22}
 \end{pmatrix},
 \qquad
 D=-2\alpha e_0,
 \qquad
 |\Delta_{ij}|\lesssim\tau^{-2}.
 \]
We first consider the linear part of this equation.
We can choose two independent solutions of the linear problem $\dot{\bf a}=B{\bf a}$
satisfying
 \begin{align*}
 {\bf w}_1
 &=
 \begin{pmatrix}
 e^\tau \\
 D\tau^{-1}e^\tau
 \end{pmatrix}
 +
 \begin{pmatrix}
 O(\tau^{-1}e^\tau) \\
 O(\tau^{-2}e^\tau)
 \end{pmatrix}
 \qquad\text{for } \tau>\tau_0,
 \\[2mm]
 {\bf w}_2
 &=
 \begin{pmatrix}
 -D\tau^{-3} \\
 \tau^{-2}
 \end{pmatrix}
 +
 \begin{pmatrix}
 O(\tau^{-4}) \\
 O(\tau^{-3})
 \end{pmatrix}
 \qquad\text{for } \tau>\tau_0.
 \end{align*}
We define the fundamental matrix $M$ by
 \begin{align*}
 M
 &=
 ({\bf w}_1,{\bf w}_2)
 =
 \underbrace{
 \begin{pmatrix}
 e^\tau & -D\tau^{-3} \\
 D\tau^{-1}e^\tau & \tau^{-2}
 \end{pmatrix}}_{=M_0}
 +
 M_1,
 \qquad
 M_1=O(\tau^{-1}M_0),
 \end{align*}
where $M_1=O(\tau^{-1}M_0)$ represents a matrix satisfying
 \[
 |[M_1]_{ij}|
 \lesssim
 \tau^{-1}|[M_0]_{ij}|
 \qquad\text{for } i,j=0,1.
 \]
The inverse matrix $N=M^{-1}$ is given by
 \begin{align*}
 N
 =
 \underbrace{
 \begin{pmatrix}
 e^{-\tau} & D\tau^{-1}e^{-\tau} \\
 -D\tau & \tau^2
 \end{pmatrix}}_{N_0}
 +
 N_1,
 \qquad
 N_1=
 O(\tau^{-1}N_0).
 \end{align*}
A solution ${\bf a}(\tau)$ of \eqref{7.4} is expressed by
 \begin{align*}
 {\bf a}(\tau)
 &=
 M(\tau)\left(
 N(\tau_0){\bf a}(\tau_0)
 +
 \int_{\tau_0}^\tau N(\tau')({\bf q}+{\bf f})(\tau')d\tau'
 \right)
 \\
 &=
 M(\tau)\left(
 N(\tau_0){\bf a}(\tau_0)
 +
 \int_{\tau_0}^\tau I_1N(\tau')({\bf q}+{\bf f})(\tau')d\tau'
 +
 \int_{\tau_0}^\tau I_2N(\tau')({\bf q}+{\bf f})(\tau')d\tau'
 \right),
 \end{align*}
where $I_1$ and $I_2$ are given by
 \[
 I_1
 =
 \begin{pmatrix}
 1 & 0 \\
 0 & 0
 \end{pmatrix},
 \qquad
 I_2
 =
 \begin{pmatrix}
 0 & 0 \\
 0 & 1
 \end{pmatrix}.
 \]
We take
 \[
 {\bf a}(\tau_0)
 =
 -M(\tau_0)
 \int_{\tau_0}^\infty I_1N(\tau')({\bf q}+{\bf f})(\tau')d\tau'.
 \]
From this choice,
we deduce that
 \begin{align}\label{7.5}
 {\bf a}(\tau)
 =
 M(\tau)\left(
 -\int_\tau^\infty I_1N(\tau')({\bf q}+{\bf f})(\tau')d\tau'
 +
 \int_{\tau_0}^\tau I_2N(\tau')({\bf q}+{\bf f})(\tau')d\tau'
 \right).
 \end{align}
By a direct computation,
we verify that
 \begin{align*}
 M_0(\tau)I_1N_0(\tau')
 &=
 \begin{pmatrix}
 e^\tau & -D\tau^{-3} \\
 D\tau^{-1}e^\tau & \tau^{-2}
 \end{pmatrix}
 \begin{pmatrix}
 e^{-\tau'} & D\tau'^{-1}e^{-\tau'} \\
 0 & 0
 \end{pmatrix}
 \\
 &=
 \begin{pmatrix}
 e^{\tau-\tau'} & D\tau'^{-1}e^{\tau-\tau'} \\
 D\tau^{-1}e^{\tau-\tau'} & D^2\tau^{-1}\tau'^{-1}e^{\tau-\tau'}
 \end{pmatrix},
 \\[2mm]
 M_0(\tau)I_2N_0(\tau')
 &=
 \begin{pmatrix}
 e^\tau & -D\tau^{-3} \\
 D\tau^{-1}e^\tau & \tau^{-2}
 \end{pmatrix}
 \begin{pmatrix}
 0 & 0 \\
 -D\tau' & \tau'^2
 \end{pmatrix}
 \\
 &=
 \begin{pmatrix}
 D^2\tau^{-3}\tau' & -D\tau^{-3}\tau'^2 \\
 -D\tau^{-2}\tau' & \tau^{-2}\tau'^2
 \end{pmatrix}.
 \end{align*}
Therefore
it holds that
 \begin{align*}
 M(\tau)I_1N(\tau')
 &=
 \underbrace{M_0(\tau)I_1N_0(\tau')}_{=S_1}
 +
 O(\tau^{-1}S_1)
 +
 O(\tau'^{-1}S_1),
 \\[1mm]
 M(\tau)I_2N(\tau')
 &=
 \underbrace{M_0(\tau)I_2N(\tau')}_{=S_2}
 +
 O(\tau^{-1}S_2)
 +
 O(\tau'^{-1}S_2).
 \end{align*}
We introduce a space $Y$ for a fixed point argument.
 \[
 Y
 =
 \{\nu\in C([\tau_0,\infty);X);\ \|\nu\|_Y\leq1\},
 \]
where $X$ is the orthogonal complement of a subspace spanned by $\{e_0,e_1\}$ of
$L_\rho^2(\R^6)$.
The norm $\|\cdot\|_Y$ is defined by
 \[
 \|\nu\|_Y:=\sup_{\tau>\tau_0}\tau^2\|\nu(\tau)\|_\rho.
 \]
For given $\nu\in Y$,
we define ${\bf a}(\tau)$ by \eqref{7.5},
and choose a parameter ${\bf d}=(d_0,d_1)$ as
 \begin{equation}\label{7.6}
 (({\bf d}\cdot{\bf e})\chi_\text{out},e_i)_\rho=a_i(\tau_0)
 \qquad
 (i=0,1).
 \end{equation}
We define $\Phi^\bot$ as a unique solution of
 \begin{equation}\label{7.7}
 \begin{cases}
 \dis
 \pa_\tau\Phi^\bot
 =
 A_z\Phi^\bot
 +
 \Phi^\bot
 -
 \left(
 \frac{2\alpha\tau^{-1}e_1+2\chi_1}{1+\alpha\tau^{-1}e_1}
 ({\bf a}\cdot{\bf e})
 \right)^\bot
 -
 \left(
 \frac{2\alpha\tau^{-1}e_1+2\chi_1}{1+\alpha\tau^{-1}e_1}
 \nu
 \right)^\bot
 \\[4mm]
 \hspace{15mm}
 +
 e^{-2\tau}F_\text{out}^\bot
 +
 e^{-2\tau}\mu^\bot
 +
 e^{-2\tau}g_\text{out}^\bot
 +
 e^{-2\tau}N_\text{out}^\bot
 &
 \qquad\text{in } \R^6\times(\tau_0,\infty),
 \\[2mm]
 \Phi^\bot(\tau_0)
 =
 \left\{
 ({\bf d}\cdot {\bf e})\chi_\text{out}
 \right\}^\bot.
 \end{cases}
 \end{equation}
This procedure defines a mapping $\nu\in Y\mapsto\Phi^\bot\in Y$.
The fixed point of this mapping gives the desired solution.
We now provide a priori estimates of $\Phi^\bot$.
Since
$|f_0(\tau)|\lesssim\tau^{-2}$, $|f_1(\tau)|\lesssim\tau^{-3}$
(see Lemma \ref{L7.4})
and
$|q_0|\lesssim\tau^{-2}\|\nu\|_\rho$, $|q_1|\lesssim\tau^{-1}\|\nu\|_\rho$,
we verify from \eqref{7.5} that
 \begin{align}\label{7.8}
 |a_0|\lesssim\tau^{-2}+\tau^{-3}\log\tau\|\nu\|_Y,
 \qquad
 |a_1|\lesssim\tau^{-2}\log\tau+\tau^{-2}\log\tau\|\nu\|_Y.
 \end{align}
Since $e_1(z)={\sf c}_1(|z|^2-2n)$,
it holds that
 \begin{align*}
 \left|
 \left(
 \left(
 \frac{2\alpha\tau^{-1}e_1}{1+\alpha\tau^{-1}e_1}
 \nu
 \right)^\bot,
 \Phi^\bot
 \right)_\rho
 \right|
 &=
 \left|
 \left(
 \frac{2\alpha\tau^{-1}e_1}{1+\alpha\tau^{-1}e_1}\nu,
 \Phi^\bot
 \right)_\rho
 \right|
 \\
 &=
 \left|
 \left(
 \frac{2\alpha\tau^{-1}e_1}{1+\alpha\tau^{-1}e_1}\nu,
 \Phi^\bot
 \right)_{L_\rho^2(|z|<\tau^\frac{1}{4})}
 +
 \left(
 \frac{2\alpha\tau^{-1}e_1}{1+\alpha\tau^{-1}e_1}\nu,
 \Phi^\bot
 \right)_{L_\rho^2(|z|>\tau^\frac{1}{4})}
 \right|
 \\
 &\lesssim
 \tau^{-\frac{1}{2}}
 (|\nu|,|\Phi^\bot|)_{L_\rho^2(|z|<\tau^\frac{1}{4})}
 +
 (|\nu|,|\Phi^\bot|)_{L_\rho^2(|z|>\tau^\frac{1}{4})}.
 \end{align*}
We apply \eqref{3.7} to get
\begin{align*}
 \left|
 \left(
 \left(
 \frac{2\alpha\tau^{-1}e_1}{1+\alpha\tau^{-1}e_1}
 \nu
 \right)^\bot,
 \Phi^\bot
 \right)_\rho
 \right|
 &\lesssim
 \tau^{-\frac{1}{2}}
 \|\nu\|_\rho\|\Phi^\bot\|_\rho
 +
 \tau^{-\frac{1}{4}}
 \|\nu\|_\rho\||z|\Phi^\bot\|_\rho
 \\
 &\lesssim
 \tau^{-\frac{1}{2}}
 \|\nu\|_\rho\|\Phi^\bot\|_\rho
 +
 \tau^{-\frac{1}{4}}
 \|\nu\|_\rho\|\Phi^\bot\|_{H_\rho^1}.
 \end{align*}
Therefore from Lemma \ref{L7.4} and \eqref{7.8},
we deduce that
 \begin{align*}
 \frac{1}{2}\frac{d}{d\tau}
 \|\Phi^\bot\|_\rho^2
 &<
 -\frac{1}{2}
 \|\Phi^\bot\|_\rho^2
 +
 c
 \left(
 \tau^{-2}
 |{\bf a}|^2
 +
 \tau^{-\frac{9}{2}}
 \|\nu\|_Y^2
 +
 \tau^{-6}
 \right)
 \\
 &<
 -\frac{1}{2}
 \|\Phi^\bot\|_\rho^2
 +
 c\tau^{-\frac{9}{2}}
 \quad\text{if } \nu\in Y.
 \end{align*}
This gives
 \[
 \|\Phi^\bot(\tau)\|_\rho
 <
 e^{-\frac{1}{2}(\tau-\tau_0)}\|\Phi^\bot(\tau_0)\|_\rho+c\tau^{-\frac{9}{4}}.
 \]
From definition of $\Phi^\bot(\tau_0)$
(see \eqref{7.7})
and $\chi_\text{out}=1$ for $|z|<\frac{K}{4}\sqrt\tau$
(see \eqref{7.3} for definition of $\chi_\text{out}$),
we easily see that
 \begin{align*}
 \|\Phi^\bot(\tau_0)\|_\rho
 &=
 \|\left\{ ({\bf d}\cdot{\bf e})\chi_\text{out} \right\}^\bot\|_\rho
 =
 \|\left\{ ({\bf d}\cdot{\bf e})(\chi_\text{out}-1) \right\}^\bot\|_\rho
 \\
 &\lesssim
 |{\bf d}|
 \cdot
 \|\,|{\bf e}|(\chi_\text{out}-1)\|_\rho
 \lesssim
 |{\bf d}|
 \left( \int_{|z|>\frac{K}{4}\sqrt{\tau}}|z|^4e^{-\frac{|z|^2}{4}}dz \right)^\frac{1}{2}
 \\
 &\lesssim
 |{\bf d}|
 \left( K\sqrt\tau_0 \right)^\frac{n+2}{2}
 e^{-\frac{K^2}{128}\tau_0}.
 \end{align*}
Since ${\bf d}$ is defined  by \eqref{7.6},
we verify from \eqref{7.8} that
 \begin{equation}\label{7.9}
 |{\bf d}|
 \lesssim
 |{\bf a}(\tau_0)|
 \lesssim
 \tau_0^{-2}\log\tau_0
 \qquad\text{if } \nu\in Y.
 \end{equation}
Therefore
it holds that $\|\Phi^\bot(\tau_0)\|_\rho\lesssim e^{-\frac{1}{2}\tau_0}$ if $\nu\in Y$.
From this estimate,
we deduce that
 \begin{align*}
 \|\Phi^\bot(\tau)\|_\rho
 <
 e^{-\frac{1}{2}(\tau-\tau_0)}\|\Phi^\bot(\tau_0)\|_\rho
 +
 c\tau^{-\frac{9}{4}}
 \lesssim
 \tau^{-\frac{9}{4}}
 \qquad\text{if } \nu\in Y.
 \end{align*}
This relation shows that the mapping $\nu\in Y\mapsto \Phi^\bot\in Y$ is well defined.
We next verify that this mapping is contractive in $Y$.
Let $\nu_i\in Y$ ($i=1,2$).
We define $({\bf a}_i(\tau),{\bf d}_i,\Phi_i^\bot(z,\tau))$ ($i=1,2$)
by \eqref{7.5} - \eqref{7.7}.
Then the function $\Phi_1^\bot-\Phi_2^\bot$ satisfies
 \[
 \begin{cases}
 \dis
 (\Phi_1^\bot-\Phi_2^\bot)_\tau
 =
 A_z(\Phi_1^\bot-\Phi_2^\bot)
 +
 (\Phi_1^\bot-\Phi_2^\bot)
 \\[2mm] \dis
 \hspace{5mm}
 -
 \left(
 \frac{2\alpha\tau^{-1}e_1+2\chi_1}{1+\alpha\tau^{-1}e_1}
 \{({\bf a}_1-{\bf a}_2)\cdot{\bf e}\}
 \right)^\bot
 -
 \left(
 \frac{2\alpha\tau^{-1}e_1+2\chi_1}{1+\alpha\tau^{-1}e_1}(\nu_1-\nu_2)
 \right)^\bot
 & \text{in } \R^6\times(\tau_0,\infty),
 \\[6mm]
 (\Phi_1^\bot-\Phi_2^\bot)(\tau_0)
 =
 \left\{
 (({\bf d}_1-{\bf d}_2)\cdot{\bf e})\chi_\text{out}
 \right\}^\bot.
 \end{cases}
 \]
In the same manner as above,
we obtain the same differential inequality.
 \begin{align*}
 \frac{1}{2}\frac{d}{d\tau}
 \|\Phi_1^\bot-\Phi_2^\bot\|_\rho^2
 &<
 -\frac{1}{2}
 \|\Phi_1^\bot-\Phi_2^\bot\|_\rho^2
 +
 c
 \left(
 \tau^{-2}|{\bf a}_1-{\bf a}_2|^2
 +
 \tau^{-\frac{9}{2}}
 \|\nu_1-\nu_2\|_Y^2
 \right).
 \end{align*}
Furthermore it holds from \eqref{7.5} that
 \begin{align*}
 |{\bf a}_1(\tau)-{\bf a}_2(\tau)|
 &\lesssim
 \left|
 M(\tau)
 \int_\tau^\infty I_1N(\tau')({\bf q}_1-{\bf q}_2)(\tau')d\tau'
 \right|
 +
 \left|
 M(\tau)
 \int_{\tau_0}^\tau I_2N(\tau')({\bf q}_1-{\bf q}_2)(\tau')d\tau'
 \right|
 \\
 &\lesssim
 \left|
 \begin{pmatrix}
 \tau^{-4}\\
 \tau^{-5}
 \end{pmatrix}
 \right|
 \|\nu_1-\nu_2\|_Y
 +
 \left|
 \begin{pmatrix}
 \tau^{-3}\log\tau \\
 \tau^{-2}\log\tau
 \end{pmatrix}
 \right|
 \|\nu_1-\nu_2\|_Y
 \\
 &\lesssim
 \tau^{-2}\log\tau
 \|\nu_1-\nu_2\|_Y.
 \end{align*}
Therefore since
 \begin{align*}
 \|(\Phi_1^\bot-\Phi_2^\bot)(\tau_0)\|_\rho
 &=
 \left\|
 \left\{
 (({\bf d}_1-{\bf d}_2)\cdot{\bf e})\chi_\text{out}
 \right\}^\bot
 \right\|_\rho
 \lesssim
 |{\bf d}_1-{\bf d}_2|
 \cdot
 \|\chi_\text{out}-1\|_\rho
 \\
 &\lesssim
 |({\bf a}_1-{\bf a}_2)(\tau_0)|
 e^{-\frac{1}{2}\tau_0}
 \lesssim
 e^{-\frac{1}{2}\tau_0}
 \|\nu_1-\nu_2\|_Y,
 \end{align*}
we deduce that
 \begin{align*}
 \|(\Phi_1^\bot-\Phi_2^\bot)(\tau)\|_\rho
 &<
 e^{-\frac{1}{2}(\tau-\tau_0)}
 \|(\Phi_1^\bot-\Phi_2^\bot)(\tau_0)\|_\rho
 +
 c\tau^{-\frac{9}{4}}
 \|\nu_1-\nu_2\|_Y
 \\
 &\lesssim
 \tau^{-\frac{9}{4}}
 \|\nu_1-\nu_2\|_Y.
 \end{align*}
This shows the claim.
From the Banach fixed point theorem,
there exits a unique fixed point $\Phi^\bot=\nu\in\ Y$ satisfying
 \[
 \begin{cases}
 \dis
 \pa_\tau\Phi^\bot
 =
 A_z\Phi^\bot
 +
 \Phi^\bot
 -
 \left(
 \frac{2\alpha\tau^{-1}e_1+2\chi_1}{1+\alpha\tau^{-1}e_1}
 ({\bf a}\cdot{\bf e})
 \right)^\bot
 -
 \left(
 \frac{2\alpha\tau^{-1}e_1+2\chi_1}{1+\alpha\tau^{-1}e_1}
 \Phi^\bot
 \right)^\bot
 \\[4mm]
 \hspace{15mm}
 +
 e^{-2\tau}F_\text{out}^\bot
 +
 e^{-2\tau}\mu^\bot
 +
 e^{-2\tau}g_\text{out}^\bot
 +
 e^{-2\tau}N_\text{out}^\bot
 & \text{in } \R^6\times(\tau_0,\infty),
 \\[3mm]
 \Phi^\bot(\tau_0)
 =
 \left\{
 ({\bf d}\cdot {\bf e})\chi_\text{out}
 \right\}^\bot.
 \end{cases}
 \]
As is stated before,
$\varphi={\bf a}\cdot{\bf e}+\Phi^\bot$ gives a solution of \eqref{7.2} - \eqref{7.3}.
Furthermore by the above construction,
it holds that
 \begin{align}\label{7.10}
 \|\varphi(\tau)\|_\rho
 &=
 |{\bf a}(\tau)|+\|\Phi^\bot(\tau)\|_\rho
 \lesssim
 \tau^{-2}\log\tau.
 \end{align}
We finally remark that
\eqref{7.2} - \eqref{7.3} admits a unique solution $\varphi(z,\tau)$
in the class $\varphi(z,\tau)$ satisfying
 \begin{enumerate}[(i)]
 \item $\varphi^\bot\in Y$,
 \item ${\bf a}(\tau)=(\varphi(\tau),{\bf e})_\rho$ is given by \eqref{7.5},
 \item ${\bf d}$ is given by \eqref{7.6},
 \end{enumerate}
where ${\bf q}(\tau)$ in \eqref{7.5} is defined by
 \[
 {\bf q}(\tau)
 =
 -
 \left(
 \frac{2e\alpha\tau^{-1}e_1}{1+\alpha\tau^{-1}e_1}\varphi^\bot(\tau),
 {\bf e}
 \right)_\rho.
 \]
The uniqueness is a consequence of the contractive property of the mapping
$\nu\in Y\mapsto\Phi^\bot\in Y$ (see above).
We will use this uniqueness result to confirm
the continuity of the mapping $(\lambda(t),\epsilon(y,t)) \mapsto \varphi(z,\tau)$.

\subsection{Pointwise estimate}
\label{S7.3}
Throughout this subsection,
$\varphi(z,\tau)$ represents a solution of \eqref{7.2} - \eqref{7.3}
constructed in Section \ref{S7.2}.
We here provide a pointwise estimate of $\varphi(z,\tau)$ to show $W(x,t)\in X_\delta$.
Let $r_1\in(0,1)$ and consider
 \begin{equation}\label{7.11}
 \begin{cases}
 \dis
 \pa_\tau\varphi_1
 =
 A_z\varphi_1
 +
 \left(
 1-\frac{2\alpha\tau^{-1}e_1+2\chi_1}{1+\alpha\tau^{-1}e_1}
 \right)
 \varphi_1
 +
 e^{-2\tau}F_\text{out}
 {\bf 1}_{|z|<\frac{1}{\tau}}
 +
 e^{-2\tau}N_\text{out}
 {\bf 1}_{|y|<2R}
 \\
 &\hspace{-30mm}\text{for } |z|<4r_1,\quad \tau>\tau_0,
 \\
 \varphi_1=0
 &\hspace{-30mm}\text{for } |z|=4r_1,\quad \tau>\tau_0,
 \\
 \varphi_1
 =
 0
 &\hspace{-30mm}\text{for } \tau=\tau_0.
 \end{cases}
 \end{equation}
We put
 \[
 \Psi_1(y,s)=
 e^\tau\varphi_1(e^\frac{\tau}{2}\lambda y,\tau),
 \qquad
 T-t=e^{-\tau},
 \qquad
 s=\int_0^t\frac{dt'}{\lambda(t')^2}.
 \]
The function $\Psi(y,s)$ solves
 \begin{align}\label{7.12}
 \begin{cases}
 \dis
 \pa_s\Psi_1
 =
 \Delta_y\Psi_1
 +
 e^\tau\lambda^2
 \left(
 2
 -
 \frac{2\alpha\tau^{-1}e_1+2\chi_1}{1+\alpha\tau^{-1}e_1}
 \right)
 \Psi_1
 +
 (\lambda\dot\lambda)
 y\cdot\nabla_y\Psi_1
 \\[2mm]
 \hspace{12.5mm}
 +
 \lambda^2F_\text{out}
 {\bf 1}_{|z|<\frac{1}{\tau}}
 +
 \lambda^2N_\text{out}
 {\bf 1}_{|y|<2R}
 &\hspace{-20mm}
 \text{for } |y|<4r_1e^{-\frac{\tau}{2}}\lambda^{-1},\quad s>0,
 \\[2mm]
 \Psi_1=0
 &\hspace{-20mm}
 \text{for } |y|=4r_1e^{-\frac{\tau}{2}}\lambda^{-1},\quad s>0,
 \\
 \Psi_1
 =
 0
 &\hspace{-20mm}\text{for } s=0.
 \end{cases}
 \end{align}
We here apply a comparison argument to  derive a priori estimate of $\Psi_1(y,s)$.
Let
 \[
 \bar\Psi_1(y,s)
 =
 \frac{e^\tau}{R^\frac{1}{8}\tau^\frac{3}{2}}\frac{1}{(1+|y|)^\frac{5}{4}}.
 \]
Since $e_1(z)={\sf c}_1(|z|^2-2n)>0$ for $|z|>\sqrt{2n}$ (see \eqref{3.4})
and $\lambda\dot\lambda=-(\frac{5}{4}+o(1))e^\tau\lambda^2$ (see \eqref{6.3}),
we verify that
 \begin{align*}
 {\cal L}\bar\Psi_1
 &=
 \pa_s\bar\Psi_1
 -
 \Delta_y\bar\Psi_1
 -
 e^\tau\lambda^2
 \left(
 2
 -
 \frac{2\alpha\tau^{-1}e_1+2\chi_1}{1+\alpha\tau^{-1}e_1}
 \right)
 \bar\Psi_1
 -
 (\lambda\dot\lambda)
 y\cdot\nabla_y\bar\Psi_1
 \\
 &=
 \left(
 \frac{\pa \tau}{\pa s}
 \left( 1-\frac{3}{2\tau} \right)
 -
 \frac{\frac{5}{4}(\frac{9}{4}-\frac{(n-1)(1+|y|)}{|y|})}{(1+|y|)^2}
 -
 e^\tau\lambda^2
 \left(
 2
 -
 \frac{2\alpha\tau^{-1}e_1}{1+\alpha\tau^{-1}e_1}
 \right)
 +
 \frac{5\lambda\dot\lambda|y|}{4(1+|y|)}
 \right)
 \bar\Psi_1
 \\
 &>
 \left(
 \frac{55}{16(1+|y|)^2}
 -
 \left( 2+c\tau^{-1}+\frac{25}{16}(1+o(1)) \right)e^\tau\lambda^2
 \right)
 \bar\Psi_1
 \\
 &>
 \left(
 \frac{3}{(1+|y|)^2}
 -
 4e^\tau\lambda^2
 \right)
 \bar\Psi_1
 \\
 &=
 \frac{3-4e^\tau\lambda^2(1+|y|)^2}{(1+|y|)^2}
 \bar\Psi_1.
 \end{align*}
We here note that (see \eqref{6.3}),
 \begin{align*}
 e^\tau\lambda^2(1+|y|)^2
 &\lesssim
 e^\tau\lambda^2
 +
 e^\tau\lambda^2
 \left( 4r_1e^{-\frac{\tau}{2}}\lambda^{-1} \right)^2
 \\
 &\lesssim
 \tau^{-\frac{15}{4}}e^{-\frac{3}{2}\tau}
 +
 16r_1^2
 \qquad
 \text{for } |y|<4r_1e^{-\frac{\tau}{2}}\lambda^{-1}.
 \end{align*}
Therefore
there exists $r_1\in(0,1)$ such that
 \[
 {\cal L}\bar\Psi_1>\frac{ \bar\Psi_1}{(1+|y|)^2}
 \qquad\text{for } |y|<4r_1e^{-\frac{\tau}{2}}\lambda^{-1},\quad s>0.
 \]
We fix $r_1\in(0,1)$ as above.
We estimate the last two terms in \eqref{7.12}.
From Lemma \ref{L7.1},
we see that
 \begin{align*}
 \lambda^2
 |F_\text{out}|
 {\bf 1}_{|z|<\frac{1}{\tau}}
 &\lesssim
 \frac{e^\tau}{\tau^\frac{3}{2}}
 \left(
 \frac{e^{-\tau}{\bf 1}_{|y|<2R}}{1+|y|^\frac{11}{2}}
 +
 \frac{\log |y|}{1+|y|^\frac{7}{2}}
 {\bf 1}_{R<|y|<2R}
 +
 \frac{{\bf 1}_{|y|>R}}{1+|y|^4}
 \right)
 \\
 &\lesssim
 R^\frac{1}{8}
 \left(
 \frac{e^{-\tau}{\bf 1}_{|y|<2R}}{1+|y|^\frac{9}{4}}
 +
 \frac{\log |y|}{1+|y|^\frac{1}{4}}
 {\bf 1}_{R<|y|<2R}
 +
 \frac{{\bf 1}_{|y|>R}}{1+|y|^\frac{3}{4}}
 \right)
 \frac{\bar\Psi_1}{(1+|y|)^2}
 \\
 &\lesssim
 \left(
 R^\frac{1}{8}e^{-\tau}
 +
 \frac{\log R}{R^\frac{1}{8}}
 +
 \frac{1}{R^\frac{5}{8}}
 \right)
 \frac{\bar\Psi_1}{(1+|y|)^2}.
 \end{align*}
Furthermore
we verify from Lemma \ref{L7.3} that
 \begin{align*}
 \lambda^2|N_\text{out}|
 {\bf 1}_{|y|<2R}
 &\lesssim
 \frac{e^{2\tau}\lambda^2}{\tau^3}
 \frac{R^8(\log R)^2}{1+|y|^{11}}
 {\bf 1}_{|y|<2R}
 \\
 &\lesssim
 \frac{e^\tau\lambda^2}{\tau^\frac{3}{2}}
 \frac{R^\frac{65}{8}(\log R)^2}{1+|y|^\frac{31}{4}}
 {\bf 1}_{|y|<2R}
 \frac{\bar\Psi_1}{(1+|y|)^2}.
 \end{align*}
Therefore
it follows that
 \begin{align*}
 \lambda^2|F_\text{out}|
 {\bf 1}_{|z|<\tau^{-1}}
 +
 \lambda^2|N_\text{out}|
 {\bf 1}_{|y|<2R}
 \lesssim
 \frac{1}{2}
 \frac{\bar\Psi_1}{(1+|y|)^2}.
 \end{align*}
A comparison argument shows
 \[
 |\Psi_1(y,s)|<\bar\Psi_1(y,s)
 =
 \frac{e^\tau}{R^\frac{1}{8}\tau^\frac{3}{2}}\frac{1}{(1+|y|)^\frac{5}{4}}
 \qquad\text{for } |y|<4r_1e^{-\frac{\tau}{2}}\lambda^{-1},\quad s>0.
 \]
We next derive a gradient estimate of $\Psi_1(y,s)$.
Since $|\lambda\dot\lambda y|\lesssim(T-t)^\frac{3}{4}\tau^{-\frac{15}{8}}$
for $|y|<4r_1e^{-\frac{\tau}{2}}\lambda^{-1}$,
a parabolic estimate (see Lemma \ref{L3.5} (i) - (ii)) implies
 \begin{align*}
 |\nabla_y\Psi_1(y,s)|
 &\lesssim
 \sup_{\max(0,s-1)<s'<s}\
 \sup_{|y-y'|<2}
 \left(
 |\Psi_1(y',s')|
 +
 \lambda^2|F_\text{out}|{\bf 1}_{|z|<\frac{1}{\tau}}
 +
 \lambda^2|N_\text{out}|{\bf 1}_{|y|<2R}
 \right)
 \\
 &\lesssim
 \frac{e^\tau}{R^\frac{1}{8}\tau^\frac{3}{2}}\frac{1}{(1+|y|)^\frac{5}{4}}
 \qquad\text{for } |y|<2r_1e^{-\frac{\tau}{2}}\lambda^{-1},\ s>0.
 \end{align*}
From these estimates,
it follows that
 \begin{align}\label{7.13}
 |\varphi_1(z,\tau)|
 \lesssim
 \frac{1}{R^\frac{1}{8}\tau^\frac{3}{2}}
 \frac{1}{1+|y|^\frac{5}{4}},
 \qquad
 |\nabla_z\varphi_1(z,\tau)|
 \lesssim
 \frac{e^{-\frac{\tau}{2}}}{R^\frac{1}{8}\lambda\tau^\frac{3}{2}}
 \frac{1}{1+|y|^\frac{5}{4}}
 \end{align}
for $|z|<2r_1$, $\tau>\tau_0$.
We put
 \[
 \varphi_2=\varphi-\varphi_1\chi_{r_1},
 \qquad
 \chi_{r_1}(z)=\eta\left( \frac{|z|}{r_1} \right).
 \]
From \eqref{7.1}, \eqref{7.3} and \eqref{7.11},
we find that
 \begin{align}\label{7.14}
 \begin{cases}
 \dis
 \pa_\tau\varphi_2
 =
 A_z\varphi_2
 +
 \left( 1-\frac{2\alpha\tau^{-1}e_1+2\chi_1}{1+\alpha\tau^{-1}e_1} \right)
 \varphi_2
 +
 \underbrace{e^{-2\tau}F_\text{out}{\bf 1}_{|z|>\frac{1}{\tau}}
 +e^{-2\tau}\mu+e^{-2\tau}g_\text{out}
 }
 \\[4mm]
 \hspace{15mm}
 \underbrace{
 +e^{-2\tau}N_\text{out}{\bf 1}_{|y|>2R}
 +
 \left(
 2\nabla_z\varphi_1\cdot\nabla_z\chi_{r_1}+\varphi_1\Delta_z\chi_{r_1}
 -\frac{\varphi_1}{2}z\cdot\nabla_z\chi_{r_1}
 \right)}_{=:S}
 & \text{in } \R^6\times(\tau_0,\infty),
 \\[2mm]
 \varphi_2(\tau_0)=({\bf d}\cdot{\bf e})\chi_\text{out}.
 \end{cases}
 \end{align}
From Lemma \ref{L7.1} - Lemma \ref{L7.3},
we verify that
 \begin{align*}
 e^{-2\tau}|F_\text{out}|{\bf 1}_{|z|>\frac{1}{\tau}}
 &\lesssim
 \tau^\frac{1}{4}
 e^{-\frac{3}{2}\tau}{\bf 1}_{|z|<\frac{2}{\tau}}
 +
 \left(
 \frac{e^{-\frac{3}{2}\tau}}{\tau^\frac{21}{4}}
 \frac{1}{|z|^2}
 +
 \frac{e^{-\frac{3}{2}\tau}}{\tau^\frac{15}{4}}
 \frac{1}{|z|^4}
 \right)
 {\bf 1}_{|z|>1},
 \\
 e^{-2\tau}|g_\text{out}|
 &\lesssim
 {\bf 1}_{|z|<\frac{2}{\tau}}
 +
 \tau^\frac{1}{4}e^{-\frac{3}{2}\tau}
 {\bf 1}_{|z|<1}
 +
 \frac{e^{-\frac{3}{2}\tau}}{\tau^\frac{15}{4}}
 \frac{1}{|z|^4}
 {\bf 1}_{|z|>1},
 \\
 e^{-2\tau}|\mu|
 &\lesssim
 \tau^{-2}
 (1+|z|^2)
 {\bf 1}_{|z|<K\sqrt{\tau}}
 +
 \frac{{\bf 1}_{|z|>K\sqrt{\tau}}}{|z|^2}
 \\
 &\lesssim
 \tau^{-2}
 (1+|z|^2)
 {\bf 1}_{|z|<K\sqrt{\tau}}
 +
 \frac{K^{-\frac{7}{4}}}{\tau^\frac{7}{8}}
 \frac{{\bf 1}_{|z|>K\sqrt{\tau}}}{|z|^\frac{1}{4}},
 \\
 e^{-2\tau}|N_\text{out}|{\bf 1}_{|y|>2R}
 &\lesssim
 {\bf 1}_{|z|<\frac{2}{\tau}}
 +
 \tau^{-3}
 (1+|z|^4)
 {\bf 1}_{|z|<K\sqrt\tau}
 +
 \frac{1}{\tau^\frac{7}{8}}
 \frac{{\bf 1}_{|z|>K\sqrt\tau}}{|z|^\frac{1}{4}}.
 \end{align*}
Furthermore
we get from \eqref{7.13} that
 \begin{align*}
 \left|
 2\nabla_z\varphi_1\cdot\nabla_z\chi_{r_1}+\varphi_1\Delta_z\chi_{r_1}
 -\frac{\varphi_1}{2}z\cdot\nabla_z\chi_{r_1}
 \right|
 &\lesssim
 \frac{e^{-\frac{\tau}{2}}}{R^\frac{1}{8}\lambda\tau^\frac{3}{2}}
 \frac{{\bf 1}_{r_1<|z|<2r_1}}{1+|y|^\frac{5}{4}}
 \lesssim
 \frac{e^{\frac{1}{8}\tau}\lambda^\frac{1}{4}}{R^\frac{1}{8}\tau^\frac{3}{2}}
 \frac{{\bf 1}_{r_1<|z|<2r_1}}{|z|^\frac{5}{4}}
 \\
 &\lesssim
 \frac{e^{-\frac{3}{16}\tau}}{R^\frac{1}{8}\tau^\frac{63}{32}}
 \frac{{\bf 1}_{r_1<|z|<2r_1}}{|z|^\frac{5}{4}}.
 \end{align*}
Therefore it follows that
 \begin{align}\label{7.15}
 |S(z,\tau)|
 &\lesssim
 {\bf 1}_{|z|<\frac{2}{\tau}}
 +
 \left(
 \tau^{-2}
 +
 \tau^{-2}|z|^2
 +
 \tau^{-3}|z|^4
 \right)
 {\bf 1}_{|z|<K\sqrt{\tau}}
 +
 \frac{1}{\tau^\frac{7}{8}}
 \frac{{\bf 1}_{|z|>K\sqrt{\tau}}}{|z|^\frac{1}{4}}
 \nonumber
 \\
 &\lesssim
 {\bf 1}_{|z|<\frac{2}{\tau}}
 +
 (\tau^{-2}+K^2\tau^{-2}|z|^2)
 {\bf 1}_{|z|<K\sqrt{\tau}}
 +
 \frac{1}{\tau^\frac{7}{8}}
 \frac{{\bf 1}_{|z|>K\sqrt{\tau}}}{|z|^\frac{1}{4}}. 
 \end{align}
Let $r_2\in(1,\sqrt{\tau_0})$ be a large constant which is determined later,
and fix $p=\frac{8n}{15}>\frac{n}{2}$.
From \eqref{7.15},
we see that
 \begin{equation}\label{7.16}
 \|S(\tau)\|_{L^p(|z|<2r_2)}
 \lesssim
 \|{\bf 1}_{|z|<\frac{2}{\tau}}\|_{L^p(|z|<2r_2)}
 +
 \tau^{-2}
 \lesssim
 \tau^{-\frac{n}{p}}+\tau^{-2}
 \lesssim
 \tau^{-\frac{15}{8}}.
 \end{equation}
We apply Lemma \ref{L3.4} (i) to get
 \begin{align*}
 \sup_{|z|<r_2}|\varphi_2(z,\tau)|
 &\lesssim
 \sup_{\tau'\in(\tau-1,\tau)}
 \left(
 \|\varphi_2(\tau')\|_{L^2(|z|<2r_2)}
 +
 \|S(\tau')\|_{L^p(|z|<2r_2)}
 \right)
 \\
 &\lesssim
 \sup_{\tau'\in(\tau-1,\tau)}
 \left(
 \|\varphi_2(\tau')\|_\rho+\tau'^{-\frac{15}{8}}
 \right)
 \qquad
 \text{for } \tau>\tau_0+1.
 \end{align*}
Since
 $\|\varphi_1\|_\rho\lesssim
 R^{-\frac{1}{8}}\tau^{-\frac{3}{2}}
 (e^{\frac{\tau}{2}}\lambda)^\frac{5}{4}
 \lesssim
 R^{-\frac{1}{8}}\tau^{-\frac{123}{32}}e^{-\frac{15}{16}\tau}$
 (see \eqref{7.13})
and
 $\|\varphi\|_\rho\lesssim\tau^{-2}\log\tau$ (see \eqref{7.10}),
it holds that $\|\varphi_2\|_\rho\lesssim\tau^{-2}\log\tau$.
Therefore we obtain
 \begin{align*}
 \sup_{|z|<r_2}|\varphi_2(z,\tau)|
 \lesssim
 \tau^{-\frac{15}{8}}
 \qquad
 \text{for } \tau>\tau_0+1.
 \end{align*}
To investigate the behavior of $\varphi_2(z,\tau)$ for
$\tau\in(\tau_0,\tau_0+1)$,
we put
 \[
 \varphi_3(z,\tau)=\varphi_2(z,\tau)-\varphi_2(z,\tau_0)
 =
 \varphi_2(z,\tau)-({\bf d}\cdot{\bf e})\chi_\text{out}.
 \]
Since $\varphi_3(z,\tau)$ satisfies a similar equation to \eqref{7.14} with $\varphi_3(\tau_0)=0$,
we verify from a local parabolic estimate (see Lemma \ref{L3.4} (ii)) that
 \begin{align}\label{7.17}
 \sup_{|z|<r_2}|\varphi_3(z,\tau)|
 &\lesssim
 \sup_{\tau'\in(0,\tau)}
 \left(
 \|\varphi_3(\tau')\|_\rho
 +
 \|S(\tau')\|_{L^p(|z|<2r_2)}
 \right)
 +
 |{\bf d}|
 \nonumber
 \\
 &\lesssim
 \sup_{\tau'\in(0,\tau)}
 \left(
 \|\varphi_2(\tau')\|_\rho
 +
 \|S(\tau')\|_{L^p(|z|<2r_2)}
 \right)
 +
 |{\bf d}|
 \nonumber
 \\
 &\lesssim
 \tau^{-\frac{15}{8}}+|{\bf d}|
 \lesssim
 \tau^{-\frac{15}{8}}
 \qquad
 \text{for } \tau\in(\tau_0,\tau_0+1).
 \end{align}
In the last inequality,
we use $|{\bf d}|\lesssim\tau_0^{-2}\log\tau_0$ (see \eqref{7.9}).
Therefore we conclude
 \begin{align}\label{7.18}
 \sup_{|z|<r_2}|\varphi_2(z,\tau)|
 \lesssim
 \tau^{-\frac{15}{8}}
 \qquad
 \text{for } \tau>\tau_0.
 \end{align}
We next consider the case $|z|>r_2$.
Let
 \[
 \bar\varphi_2(z,\tau)
 =
 \tau^{-\frac{7}{4}}|z|^\frac{17}{8}.
 \]
Since $e_1(z)={\sf c}_1(|z|^2-2n)$ for some ${\sf c}_1>0$,
we verify that
 \begin{align*}
 {\cal L}\bar\varphi_2
 &=
 \pa_\tau\bar\varphi_2
 -
 A_z\varphi_2
 -
 \left( 1-\frac{2\alpha\tau^{-1}e_1+2\chi_1}{1+\alpha\tau^{-1}e_1} \right)
 \bar\varphi_2
 \\
 &>
 \left(
 -\frac{7}{4\tau}
 -
 \frac{17}{8}\left( \frac{1}{8}+n \right)\frac{1}{|z|^2}
 +
 \frac{17}{16}
 -
 \left(
 1+\frac{c}{\tau}
 \right)
 \right)
 \bar\varphi_2
 \\
 &>
 \left(
 -\frac{c}{\tau}
 -
 \frac{c}{|z|^2}
 +
 \frac{1}{16}
 \right)
 \bar\varphi_2 \qquad\text{for } |z|>r_2.
 \end{align*}
We here fix $r_2>0$ large enough such that
 ${\cal L}\bar\varphi_2>\frac{1}{32}\bar\varphi_2$ for $|z|>r_2$.
On the other hand,
it holds from \eqref{7.15} that
 \begin{align*}
 |S(z,\tau)|
 &\lesssim
 \tau^{-2}|z|^2
 {\bf 1}_{|z|<K\sqrt{\tau}}
 +
 \frac{1}{\tau^\frac{7}{8}}
 \frac{{\bf 1}_{|z|>K\sqrt{\tau}}}{|z|^\frac{1}{4}}
 \\
 &\lesssim
 \tau^{-2}|z|^2
 {\bf 1}_{|z|<K\sqrt{\tau}}
 +
 K^{-\frac{19}{8}}
 \tau^{-\frac{33}{16}}
 |z|^{\frac{17}{8}}
 {\bf 1}_{|z|>K\sqrt{\tau}}
 \\
 &\lesssim
 \left(
 \tau^{-\frac{1}{4}}
 \frac{1}{|z|^\frac{1}{8}}
 {\bf 1}_{|z|<K\sqrt{\tau}}
 +
 \tau^{-\frac{5}{16}}
 {\bf 1}_{|z|>K\sqrt{\tau}}
 \right)
 \bar\varphi_2
 \qquad\text{for } |z|>r_2.
 \end{align*}
From definition of $\varphi_2(\tau_0)$ (see \eqref{7.14})
and $|{\bf d}|\lesssim\tau_0^{-2}\log\tau_0$ (see \eqref{7.9}),
we note that
 \begin{align*}
 |\varphi_2(\tau_0)|
 &<
 |{\bf d}\cdot{\bf e}|
 \chi_\text{out}
 \lesssim
 |{\bf d}|(1+|z|^2)
 {\bf 1}_{|z|<\frac{K}{2}\sqrt{\tau}_0}
 \lesssim
 \tau_0^{-2}\log\tau_0
 (1+|z|^2)
 {\bf 1}_{|z|<\frac{K}{2}\sqrt{\tau}_0}.
 \end{align*}
Furthermore since $|\varphi_2(z,\tau)|\lesssim\tau^{-\frac{15}{8}}$
for $|z|=r_2$ (see  \eqref{7.18}),
it holds that $|\varphi_2(z,\tau)|<\bar\varphi_2(z,\tau)$ on $|z|=r_2$, $\tau>\tau_0$.
Therefore a comparison argument gives
 \[
 |\varphi_2(z,\tau)|<\bar\varphi_2(z,\tau)
 \qquad\text{for } |z|>r_2,\quad \tau>\tau_0.
 \]
From this estimate,
we obtain
 \begin{equation}\label{7.19}
 |\varphi_2(z,\tau)|
 <
 \tau^{-\frac{7}{4}}|z|^\frac{17}{8}
 <
 K^\frac{1}{8}\tau^{-\frac{27}{16}}|z|^2
 \qquad\text{for } r_2<|z|<K\sqrt{\tau},\quad \tau>\tau_0.
 \end{equation}
We finally derive an estimate for $|z|>K\sqrt{\tau}$.
Since $e_1(z)={\sf c}_1(|z|^2-2n)$ for some ${\sf c}_1>0$,
we can fix $K>0$ such that
 \[
 \frac{2\alpha\tau^{-1}e_1}{1+\alpha\tau^{-1}e_1}
 >
 \frac{7}{4}
 \qquad\text{for } |z|>K\sqrt{\tau}.
 \]
We define a comparison function as
 \[
 \bar\varphi_{2,\text{out}}(z,\tau)
 =
 \frac{1}{\tau^\frac{1}{2}}\frac{1}{|z|^\frac{1}{8}}.
 \]
From this definition,
we find
 \begin{align*}
 {\cal L}\bar\varphi_{2,\text{out}}
 &=
 \pa_\tau\bar\varphi_{2,\text{out}}
 -
 A_z\bar\varphi_{2,\text{out}}
 -
 \left(
 1-\frac{2\alpha\tau^{-1}e_1+2\chi_1}{1+\alpha\tau^{-1}e_1}
 \right)
 \bar\varphi_{2,\text{out}}
 \\
 &>
 \left(
 -\frac{1}{2\tau}
 -
 \frac{1}{8}\left( \frac{17}{8}-n \right)
 \frac{1}{|z|^2}
 -
 \frac{1}{16}
 +
 \frac{3}{4}
 \right)
 \bar\varphi_{2,\text{out}}
 \\
 &>
 \left(
 -\frac{1}{2\tau}
 +
 \frac{11}{16}
 \right)
 \bar\varphi_{2,\text{out}}
 \qquad\text{for } |z|>K\sqrt{\tau}.
 \end{align*}
Since
$|S(z,\tau)|\lesssim\tau^{-\frac{7}{8}}|z|^{-\frac{1}{4}}
 \lesssim\tau^{-\frac{3}{8}}|z|^{-\frac{1}{8}}\bar\varphi_{2,\text{out}}$
for $|z|>K\sqrt\tau$ (see \eqref{7.15}),
it holds that ${\cal L}\bar\varphi_{2,\text{out}}>S$ for $|z|>K\sqrt\tau$, $\tau>\tau_0$.
Furthermore \eqref{7.19} implies
 \begin{align*}
 |\varphi_2(z,\tau)|
 \lesssim
 \tau^{-\frac{7}{4}}
 |z|^\frac{17}{8}
 \lesssim
 \frac{K^\frac{9}{4}}{\tau^\frac{5}{8}}
 \frac{1}{|z|^\frac{1}{8}}
 <
 \bar\varphi_{2,\text{out}}(z,\tau)
 \qquad\text{for } |z|=K\sqrt{\tau}.
 \end{align*}
Since $\varphi_2(\tau_0)=({\bf d}\cdot{\bf e})\chi_\text{out}=0$
for $|z|>\frac{K}{2}\sqrt{\tau_0}$ (see \eqref{7.3} for definition of $\chi_\text{out}$),
we apply a comparison argument to obtain
 \begin{equation}\label{7.20}
 |\varphi_2(z,\tau)|
 <
 \bar\varphi_{2,\text{out}}(z,\tau)
 =
 \frac{1}{\tau^\frac{1}{2}}
 \frac{1}{|z|^\frac{1}{8}}
 \qquad\text{for } |z|>K\sqrt{\tau},\quad \tau>\tau_0.
 \end{equation}
We now recall that $W(x,t)=e^\tau\varphi(z,\tau)$
and
$\varphi=\varphi_1\chi_{r_1}+\varphi_2$.
Therefore from \eqref{7.13}, \eqref{7.18} - \eqref{7.20},
we conclude
 \begin{align*}
 |W(x,t)|
 &\lesssim
 \frac{e^\tau}{R^\frac{1}{8}\tau^\frac{3}{2}}
 \frac{\chi_{r_1}}{1+|y|^\frac{5}{4}}
 +
 \frac{e^\tau}{\tau^\frac{15}{8}}
 {\bf 1}_{|z|<r_2}
 +
 \frac{e^\tau}{\tau^\frac{27}{16}}|z|^2
 {\bf 1}_{r_2<|z|<K\sqrt{\tau}}
 +
 \frac{1}{\tau^\frac{1}{2}}
 \frac{{\bf 1}_{|z|>K\sqrt{\tau}}}{|z|^\frac{1}{8}}.
 \end{align*}
This shows $W(x,t)\in X_\delta$.

\subsection{Gradient estimates}
\label{S7.4}
Let $\varphi_2(z,\tau)$ be as in Section \ref{S7.3}.
We here derive a gradient estimate of $\varphi_2(z,\tau)$,
We first consider the case $|z|<8$.
Due to \eqref{7.16} - \eqref{7.18} and Lemma \ref{L3.5} (i),
it holds that
\begin{align*}
 \sup_{|z|<8}
 |\nabla_z\varphi_2(z,\tau)|
 &\lesssim
 \sup_{\tau-1<\tau'<\tau}
 \left(
 \sup_{|z|<16}
 |\varphi_2(z,\tau')|
 +
 \|S(\tau')\|_{L_z^p(|z|<16)}
 \right)
 \\
 &\lesssim
 \tau^{-\frac{15}{8}}
 +
 \tau^{-\frac{n}{p}}
 \lesssim
 \tau^{-\frac{n}{p}}
 \qquad\text{for } \tau>\tau_0+1
 \qquad
 (p>n).
 \end{align*}
Furthermore let
 \[
 \varphi_3(z,\tau)
 =
 \varphi_2(z,\tau)-\varphi_2(z,\tau_0)
 =
 \varphi_2(z,\tau)-({\bf d}\cdot{\bf e})\chi_\text{out}.
 \]
From a similar computation to \eqref{7.17}
and Lemma \ref{L3.5} (ii),
we see that
 \begin{align*}
 \sup_{|z|<8}
 |\nabla_z\varphi_3(z,\tau)|
 &\lesssim
 \sup_{\tau_0<\tau'<\tau}
 \left(
 \sup_{|z|<16}
 |\varphi_3(z,\tau')|
 +
 \|S(\tau')\|_{L_z^p(|z|<16)}
 +
 |{\bf d}|
 \right)
 \\
 &\lesssim
 \tau^{-\frac{15}{8}}
 +
 \tau^{-\frac{n}{p}}
 +
 |{\bf d}|
 \lesssim
 \tau^{-\frac{n}{p}}
 \qquad\text{for } \tau>\tau_0+1
 \qquad
 (p>n).
 \end{align*}
As a consequence,
we obtain
 \begin{equation}\label{7.21}
 |\nabla_z\varphi_2(z,\tau)|
 \lesssim
 \tau^{-\frac{n}{p}}
 \qquad\text{for }
 |z|<8,\quad \tau>\tau_0
 \qquad
 (p>n).
 \end{equation}
We next consider the case $|z|>4$.
For any fixed $(z_1,\tau_1)\in\R^6\times(\tau_0+1,\infty)$ with $|z_1|>4$,
we put
 \[
 \tilde\varphi_2(\tilde z,\tilde\tau)
 =
 \varphi_2(\tilde z+e^\frac{\tilde\tau}{2}z_1,\tilde\tau+\tau_1).
 \]
This $\tilde\varphi_2(\tilde z,\tilde\tau)$ solves
 \[
 \pa_{\tilde\tau}\tilde\varphi_2
 =
 A_{\tilde z}\tilde\varphi_2
 +
 \left( 1-\frac{2\alpha\tau^{-1}e_1+2\chi_1}{1+\alpha\tau^{-1}e_1} \right)
 \tilde\varphi_2
 +
 \tilde S(\tilde z,\tilde\tau),
 \]
where $\tilde S(\tilde z,\tilde \tau)=S(z,\tau)$ with
$z=\tilde z+e^\frac{\tilde\tau}{2}z_1$ and $\tau=\tilde\tau+\tau_1$.
We note that
 \[
 \frac{1}{2}|z_1|
 <
 |z|
 =
 |\tilde z+e^\frac{\tilde\tau}{2}z_1|
 <
 2\sqrt{e}|z_1|
 \qquad\text{for } |\tilde z|<2,\quad \tilde\tau\in(0,1).
 \]
Therefore
from \eqref{7.15}, \eqref{7.18} - \eqref{7.19} and Lemma \ref{L3.5} (i),
it holds that if $4<|z_1|<(2\sqrt e)^{-1}K\sqrt\tau$
 \begin{align*}
 \sup_{|\tilde z|<1}
 |\nabla_{\tilde z}\tilde\varphi_2(\tilde z,\tilde\tau)|_{\tilde\tau=1}|
 &\lesssim
 \sup_{\tilde\tau\in(1,0)}
 \left(
 \|\tilde\varphi_2(\tilde\tau)\|_{L^\infty(|\tilde z|<2)}
 +
 \|\tilde S(\tilde\tau)\|_{L^\infty(|\tilde z|<2)}
 \right)
 \nonumber
 \\
 &\lesssim
 \sup_{\tilde\tau\in(0,1)}\
 \sup_{\frac{1}{2}|z_1|<|z|<2\sqrt e|z_1|}
 \left(
 |\varphi_2(z,\tilde\tau+\tau_1)|
 +
 |S(z,\tilde\tau+\tau_1)|
 \right)
 \nonumber
 \\
 &\lesssim
 \tau_1^{-\frac{15}{8}}
 +
 \tau_1^{-\frac{27}{16}}
 |z_1|^2
 +
 \tau_1^{-2}
 |z_1|^2
 \lesssim
 \tau_1^{-\frac{27}{16}}
 |z_1|^2.
 \end{align*}
Since
 $\nabla_{\tilde z}\tilde\varphi_2(\tilde z,\tilde\tau)|_{\tilde z=0, \tilde\tau=1}
 = \nabla_z\varphi_2(\sqrt{e}z_1,1+\tau_1)$,
from this relation,
we deduce that
 \[
 |\nabla_z\varphi_2(\sqrt{e}z_1,1+\tau_1)|
 \lesssim
 \tau_1^{-\frac{27}{16}}|z_1|^2
 \qquad\text{for }
 4<|z_1|<(2\sqrt e)^{-1}K\sqrt\tau,\quad \tau_1>\tau_0+1.
 \]
This gives
 \[
 |\nabla_z\varphi_2(z,\tau)|
 \lesssim
 \tau^{-\frac{27}{16}}
 |z|^2
 \qquad\text{for }
 4\sqrt{e}<|z|<\frac{K}{2}\sqrt\tau,\quad \tau>\tau_0+2.
 \]
In the same way as \eqref{7.17},
we obtain
 \[
 |\nabla_z\varphi_2(z,\tau)|
 \lesssim
 \tau^{-\frac{27}{16}}
 |z|^2
 \qquad\text{for }
 4\sqrt{e}<|z|<\frac{K}{2}\sqrt\tau,\quad \tau\in(\tau_0,\tau_0+2).
 \]
Therefore
it follows that
 \begin{align}\label{7.22}
 |\nabla_z\varphi_2(z,\tau)|
 \lesssim
 \tau^{-\frac{27}{16}}
 |z|^2
 \qquad\text{for }
 4\sqrt{e}<|z|<\frac{K}{2}\sqrt\tau,\quad \tau>\tau_0.
 \end{align}
Since $\varphi=\varphi_1\chi_{r_1}+\varphi_2$,
from \eqref{7.13} and \eqref{7.21} - \eqref{7.22},
we conclude
 \begin{align}\label{7.23}
 |\nabla_z\varphi(z,\tau)|
 &<
 |\nabla_z\varphi_1(z,\tau)|\chi_{r_1}
 +
 |\varphi_1(z,\tau)|\cdot|\nabla_z\chi_{r_1}|
 +
 |\nabla_z\varphi_2(z,\tau)|
 \nonumber
 \\
 &\lesssim
 \frac{e^{-\frac{\tau}{2}}}{R^\frac{1}{8}\lambda\tau^\frac{3}{2}}
 \frac{{\bf 1}_{|z|<1}}{1+|y|^\frac{5}{4}}
 +
 \frac{1}{R^\frac{1}{8}\tau^\frac{3}{2}}
 \frac{{\bf 1}_{r_1<|z|<2r_1}}{1+|y|^\frac{5}{4}}
 +
 \tau^{-\frac{n}{p}}
 {\bf 1}_{|z|<8}
 +
 \tau^{-\frac{27}{16}}
 |z|^2
 {\bf 1}_{|z|>4\sqrt{e}}
 \nonumber
 \\
 &\lesssim
 \frac{e^{-\frac{\tau}{2}}}{R^\frac{1}{8}\lambda\tau^\frac{3}{2}}
 \frac{{\bf 1}_{|z|<1}}{1+|y|^\frac{5}{4}} +
 \tau^{-\frac{n}{p}}
 {\bf 1}_{|z|<8}
 +
 \tau^{-\frac{27}{16}}
 |z|^2
 {\bf 1}_{|z|>8}
 \nonumber
 \\
 &\hspace{40mm}
 \text{for } |z|<\frac{K}{2}\sqrt\tau,\quad \tau>\tau_0
 \qquad
 (p>n).
 \end{align}

\section{Proof of theorems}
\label{S8}
\subsection{Proof of Theorem \ref{Thm1}}
We perform the argument mentioned in Section \ref{S5.2} rigorously.
Let $\tilde w(x,t)\in C(\R^6\times[0,T])$ be an extension of $w(x,t)\in X_\delta$
defined in Section \ref{S5.2},
$(\lambda(t),\epsilon(y,t))$ be a pair of functions constructed
in Section \ref{S6},
and $W(x,t)$ be a function obtained in Section \ref{S7}.
We now define
 \begin{equation}\label{8.1}
 w(x,t)\in X_\delta
 \mapsto
 W(x,t)\in C(\R^6\times[0,T-\delta])\subset C(\R^6\times[0,T)).
 \end{equation}
To apply the Schauder fixed point theorem in $X_\delta$,
we check
 \begin{enumerate}[(i)]
 \setlength{\parskip}{0cm} 
 \setlength{\itemsep}{1mm} 
 \item the mapping $(X_\delta,\|\cdot\|_{C(\R^n\times[0,T-\delta])})\to
 (X_\delta,\|\cdot\|_{C(\R^n\times[0,T-\delta])})$ is well defined,
 \item the mapping is continuous and
 \item the mapping is compact.
 \end{enumerate}
As is stated in the end of Section \ref{S7.3},
we already prove $W(x,t)\in X_\delta$.
Therefore (i) is verified.
As for (ii),
the continuity of the mapping
$(X_\delta,\|\cdot\|_{C(\R^n\times[0,T-\delta])})\to
 (X_\delta,\|\cdot\|_{C(\R^n\times[0,T-\delta])})$
is a consequence of construction of $W(x,t)$
(see Section \ref{S5} - Section \ref{S7}).
We finally check (iii).
Let $\{w_i(x,t)\}_{i=1}^\infty$ be a sequence in $X_\delta$ and
$\{W_i(x,t)\}_{i=1}^\infty$ be as in \eqref{8.1}. 
Since $\{W_i(x,t)\}_{i=1}^\infty\subset X_\delta$,
it holds that
 \[
 \sup_{i\in\N}\
 \sup_{(x,t)\in\R^6\times(0,T-\delta)}
 |W_i(x,t)|
 <
 \infty.
 \]
Therefore by a standard parabolic estimate,
there exists a subsequence $\{W_i(x,t)\}_{i=1}^\infty$ and a limiting function
$W_\infty(x,t)\in X_\delta$ such that
 \begin{equation}\label{8.2}
 \lim_{i\to\infty}\
 \sup_{0<t<T-\delta}\
 \sup_{|x|<R}|W_i(x,t)-W_\infty(x,t)|=0
 \quad\text{for any }  R>0.
 \end{equation}
Furthermore
since $|W_i(x,t)|<{\cal W}(x,t)$ (see \eqref{5.9} - \eqref{5.10}),
it holds that
 \begin{align}\label{8.3}
 |W_i(x,t)|
 &<
 \frac{K^\frac{17}{8}}{(T-t)\tau^\frac{7}{16}}
 \frac{1}{|z|^\frac{1}{8}}
 <
 \frac{K^\frac{17}{8}}{(T-t)^\frac{15}{16}\tau^\frac{7}{16}}
 \frac{1}{|x|^\frac{1}{8}}
 \nonumber
 \\
 &<
 \frac{K^\frac{17}{8}}{\delta^\frac{15}{16}\tau^\frac{7}{16}}
 \frac{1}{|x|^\frac{1}{8}}
 \qquad
 \text{for } |x|>K\tau\sqrt{T-t},\quad t\in(0,T-\delta).
 \end{align}
Combining \eqref{8.2} - \eqref{8.3},
we obtain
 \[
 \lim_{i\to\infty}\|W_i-W_\infty\|_{C(\R^6\times[0,T-\delta])}=0.
 \]
This shows (iii).
From the Schauder fixed point theorem,
there exists $w(x,t)\in X_\delta$ such that $w(x,t)=W(x,t)$.
Since
 $\tilde w(x,t)=w(x,t)$ for $(x,t)\in\R^6\times[0,T-\delta]$ (see Section \ref{S5.2}),
the function
 \[
 u(x,t)=\lambda^{-2}{\sf Q}(y)+\lambda_0^{-2}\sigma T_1(y)\chi_1
 -\Theta(x,t)\chi_2+\lambda^{-2}\epsilon(y,t)\chi_\text{in}+w(x,t)
 \]
gives a solution of \eqref{1.1} in $\R^6\times(0,T-\delta)$.
We denote this $u(x,t)$ by $u^{(\delta)}(x,t)$.
Then the limiting function
$u(x,t)=\lim_{\delta\to0}u^{(\delta)}(x,t)$
is the desired blowup solution described in Theorem \ref{Thm1}.

\subsection{Energy value}
\label{S8.2}
We compute an exact value of the local energy of the solution constructed above.
The local energy is defined by
 \[
 E_\text{loc}(u)
 =
 \frac{1}{2}
 \int_{|x|<1}|\nabla_xu|dx
 -
 \frac{1}{3}
 \int_{|x|<1}|u|^3dx.
 \]
The solution $u(x,t)$ is written as (see Section \ref{S5.1} and Section \ref{S7.3})
 \begin{align}\label{8.4}
 u(x,t)
 &=
 \frac{1}{\lambda^2}{\sf Q}(y)
 +
 \frac{\sigma}{\lambda_0^2}
 T_1(y)\chi_1
 -
 \Theta(x,t)\chi_2
 +
 \frac{1}{\lambda^2}\epsilon(y,t)\chi_\text{in}
 +
 w(x,t)
 \nonumber
 \\
 &=
 \frac{1}{\lambda^2}{\sf Q}(y)
 +
 \frac{\sigma}{\lambda_0^2}
 T_1(y)\chi_1
 -
 \Theta(x,t)\chi_2
 +
 \frac{1}{\lambda^2}\epsilon(y,t)\chi_\text{in}
 +
 w_1(x,t)\chi_{r_1}+w_2(x,t),
 \end{align}
 where
 $w_i(x,t)=(T-t)^{-1}\varphi_i(z,\tau)$ ($i=1,2$).

\subsubsection{Computation of $\dis\int_{|x|<1}|u|^3dx$}
 We decompose the integral into four parts.
 \begin{align*}
 \Omega_1
 &=
 \{x\in\R^6;\ |y|<(T-t)^{-\frac{1}{4}}\}
 =
 \{x\in\R^6;\ |x|<\lambda(T-t)^{-\frac{1}{4}}\},
 \\
 \Omega_2
 &=
 \{x\in\R^6;\ (T-t)^{-\frac{1}{4}}<|y|,\ |z|<2\tau^{-1}\},
 \\
 \Omega_3
 &=
 \{x\in\R^6;\ 2\tau^{-1}<|z|<\tau^\frac{19}{30}\},
 \\
 \Omega_4
 &=
 \{x\in\R^6;\ \tau^\frac{19}{30}<|z|,\ |x|<1\}.
 \end{align*}
 We first recall the following elementary relation.
 \begin{equation}\label{8.5}
 ||b|^3-|a|^3|\lesssim a^2|b-a|+|b-a|^3
 \qquad\text{for } a,b\in\R.
 \end{equation}
 From \eqref{8.5},
 we see that
 \begin{align*}
 \int_{\Omega_1}
 \left|
 |u|^3-\left( \frac{\sf Q}{\lambda^2} \right)^3
 \right|
 dx
 &<
 \int_{\Omega_1}
 \left( \frac{\sf Q}{\lambda^2} \right)^2
 \left|
 u-\frac{\sf Q}{\lambda^2}
 \right|
 +
 \int_{\Omega_1}
 \left|
 u-\frac{\sf Q}{\lambda^2}
 \right|^3
 dx
 \\
 &=
 \|{\sf Q}\|_3^2
 \left\|
 u-\frac{\sf Q}{\lambda^2}
 \right\|_{L_x^3(\Omega_1)}
 +
 \left\|
 u-\frac{\sf Q}{\lambda^2}
 \right\|_{L_x^3(\Omega_1)}^3.
 \end{align*}
Since $\chi_2=0$ in $\Omega_1$,
from \eqref{5.9} - \eqref{5.10} and \eqref{6.23},
the integral on the right-hand side is computed as
 \begin{align*}
 \int_{\Omega_1}
 \left|
 u-\frac{\sf Q}{\lambda^2}
 \right|^3
 dx
 &=
 \int_{\Omega_1}
 \left|
 \frac{\sigma}{\lambda_0^2}
 T_1(y)\chi_1
 +
 \frac{1}{\lambda^2}\epsilon(y,t)\chi_\text{in}
 +
 w(x,t)
 \right|^3dx
 \\
 &\lesssim
 \int_{\Omega_1}
 \left(
 \frac{1}{(T-t)^3}
 +
 \frac{1}{\tau^\frac{9}{2}(T-t)^3}
 \frac{R^12(\log R)^3}{1+|y|^\frac{33}{2}}
 +
 \frac{1}{\tau^\frac{9}{2}(T-t)^3}
 \right)
 dx
 \\
 &\lesssim
 \frac{1}{(T-t)^3}
 \left( \frac{\lambda}{(T-t)^\frac{1}{4}} \right)^n
 +
 \frac{R^{12}(\log R)^3}{\tau^\frac{9}{2}(T-t)^3}
 \lambda^n
 \\
 &\lesssim
 \tau^{-\frac{45}{4}}(T-t)^3.
 \end{align*}
Therefore it follows that
 \begin{align}\label{8.6}
 \int_{\Omega_1}
 \left|
 |u|^3-\left( \frac{\sf Q}{\lambda^2} \right)^3
 \right|
 dx
 \lesssim
 \tau^{-\frac{15}{4}}(T-t).
 \end{align}
 We next provide an estimate in $\Omega_3$.
 Since $\chi_1=0$ and $\chi_2=1$ in $\Omega_3$,
 we get from \eqref{8.5} that
 \[
 \left| |u|^3-\Theta^3 \right|
 \lesssim
 \Theta^2|u-\Theta|
 +
 |u-\Theta|^3
 \lesssim
 \Theta^2
 \left(
 \frac{\sf Q}{\lambda^2}+|w|
 \right)
 +
 \left( \frac{\sf Q}{\lambda^2} \right)^3
 +
 |w|^3
 \qquad\text{in } \Omega_3.
 \]
 Furthermore
 we note from \eqref{5.9} - \eqref{5.10} that
 \[
 |w(x,t)|\lesssim
 \begin{cases}
 K^4\tau^{-\frac{1}{2}}\Theta(x,t) & \text{for } |z|<K\sqrt\tau,
 \\
 K^\frac{17}{8}\tau^{-\frac{1}{4}}\Theta(x,t)
 & \text{for } K\sqrt\tau<|z|<\tau^\frac{19}{30}.
 \end{cases}
 \]
 Due to this relation and
 \begin{align*}
 \int_{\Omega_3}
 \left( \frac{\sf Q}{\lambda^2} \right)^3
 dx
 <
 \int_{|y|>\frac{2\sqrt{T-t}}{\tau\lambda}}
 {\sf Q}^3
 dy
 \lesssim
 \left( \frac{2\sqrt{T-t}}{\tau\lambda} \right)^{-6}
 \lesssim
 \tau^{-\frac{21}{4}}
 (T-t)^\frac{9}{2},
 \end{align*}
 we get
 \begin{align}\label{8.7}
 \int_{\Omega_3}
 \left| |u|^3-\Theta^3 \right|dx
 &\lesssim
 \int_{\Omega_3}
 \left(
 \Theta^2
 \left(
 \frac{\sf Q}{\lambda^2}+|w|
 \right)
 +
 \left( \frac{\sf Q}{\lambda^2} \right)^3
 +
 |w|^3
 \right)
 dx
 \nonumber
 \\
 &\lesssim
 \int_{\Omega_3}
 \Theta^2\left( \frac{\sf Q}{\lambda^2} \right)
 dx
 +
 \frac{1}{\tau^\frac{1}{4}}
 \int_{\Omega_4}\Theta^3dx 
 +
 \int_{\Omega_3}
 \left( \frac{\sf Q}{\lambda^2} \right)^3
 dx
 +
 \frac{1}{\tau^\frac{3}{4}}
 \int_{\Omega_3}\Theta^3dx 
 \nonumber
 \\
 &\lesssim
 \int_{\Omega_3}
 \left( \tau^{-\frac{1}{4}}\Theta^3
 +
 \tau^\frac{1}{2}
 \left( \frac{\sf Q}{\lambda^2} \right)^3
 \right)
 dx
 +
 \frac{1}{\tau^\frac{1}{4}}
 \int_{\Omega_3}\Theta^3dx
 +
 \int_{\Omega_3}
 \left( \frac{\sf Q}{\lambda^2} \right)^3
 dx
 \nonumber
 \\
 &\lesssim
 \frac{1}{\tau^\frac{1}{4}}
 \int_{\Omega_3}\Theta^3dx
 +
 \tau^{-\frac{19}{4}}
 (T-t)^\frac{9}{2}.
 \end{align}
 Therefore
 from \eqref{8.6} - \eqref{8.7} and
 \begin{align*}
 \int_{\Omega_1}
 \left( \frac{\sf Q}{\lambda^2} \right)^3
 dx
 &=
 \int_{\R^6}
 {\sf Q}^3
 dy
 -
 \int_{|y|>(T-t)^{-\frac{1}{4}}}
 {\sf Q}^3
 dy
 =
 \int_{\R^6}
 {\sf Q}^3
 dy
 +
 O((T-t)^\frac{3}{2}),
 \\
 \int_{\Omega_3}\Theta^3dx
 &=
 \int_{|z|<\tau^\frac{19}{30}}\Theta^3dx
 -
 \int_{|z|<\frac{2}{\tau}}\Theta^3dx
 =
 \int_{|z|<\tau^\frac{19}{30}}\Theta^3dx
 +
 O(\tau^{-6}),
 \end{align*}
 we obtain
 \begin{align}\label{8.8}
 \int_{|x|<1}|u|^3
 dx
 &>
 \int_{\Omega_1}|u|^3
 dx
 +
 \int_{\Omega_3}|u|^3
 dx
 \nonumber
 >
 \int_{\Omega_1}
 \left( \frac{\sf Q}{\lambda^2} \right)^3
 dx
 +
 (1-c\tau^{-\frac{1}{4}}) 
 \int_{\Omega_3}\Theta^3dx
 -
 o_1
 \nonumber
 \\
 &>
 \int_{\R^6}
 {\sf Q}^3
 dy
 +
 (1-c\tau^{-\frac{1}{4}}) 
 \int_{|z|<\tau^\frac{19}{30}}\Theta^3dx
 -
 o_1,
 \end{align}
where $o_1$ is a positive function satisfying $\lim_{t\to T}o_1=0$.

\subsubsection{Computation of $\dis\int_{|x|<1}|\nabla_xu|^2dx$}
We first investigate the behavior of $u(x,t)$ in $|x|>\frac{K}{2}\sqrt\tau\sqrt{T-t}$.
We put
 \[
 \bar u(x,t)
 =
 (T-t)^{-1}
 \left( 1+\frac{\alpha}{2\tau}e_1(z) \right)^{-1},
 \]
 where $\alpha$ is a positive constant defined by \eqref{4.6}.
 We now claim that
 \begin{equation}\label{8.9}
 |u(x,t)|<\bar u(x,t)
 \qquad \text{for } |x|>\frac{K}{2}\sqrt\tau\sqrt{T-t},\quad t\in(0,T).
 \end{equation}
 Since
 $w(x,t)|_{t=0}=w_2(x,t)|_{t=0}=e^\tau\varphi_2(z,\tau_0)$ for $|z|>2r_1$
 and
 $\varphi_2(z,\tau_0)=({\bf d}\cdot{\bf e})\chi_\text{out}=0$
 for $|z|>\frac{K}{2}\sqrt\tau$
 (see \eqref{7.3} for definition of $\chi_\text{out}$),
 we easily see that
 \begin{equation}\label{8.10}
 w(x,t)|_{t=0}=0
 \qquad \text{for } |x|>\frac{K}{2}\sqrt\tau\sqrt{T-t}.
 \end{equation}
Therefore due to \eqref{8.4} and $\Theta(x,t)<\bar{u}(x,t)$,
we get
 \begin{align*}
 |u(x,t)|_{t=0}|
 &=
 \left.
 \left|
 \frac{1}{\lambda^2}{\sf Q}(y)
 \right|_{t=0}
 -
 \Theta|_{t=0}
 \right|
 \\
 &<
 \Theta|_{t=0}
 <
 \bar{u}(x,t)|_{t=0}
 \qquad\text{for } |x|>\frac{K}{2}\sqrt{\tau}\sqrt{T-t},\quad t=0.
 \end{align*}
Furthermore
it holds from \eqref{7.19} that
 \begin{align*}
 |w_2(x,t)|
 &\lesssim
 (T-t)^{-1}\tau^{-\frac{27}{16}}
 |z|^2
 \\
 &\lesssim
 \tau^{-\frac{27}{16}}
 |z|^4
 \frac{1}{|x|^2}
 \lesssim
 K^4\tau^\frac{5}{16}
 \frac{1}{|x|^2}
\qquad\text{for } |x|=\frac{K}{2}\sqrt{\tau}\sqrt{T-t},\quad \tau>\tau_0.
 \end{align*}
From this estimate and
 \begin{align*}
 \Theta(x,t)
 &\approx
 \frac{\tau}{\alpha(T-t)}\frac{1}{e_1(z)}
 \approx
 \frac{\tau}{\alpha{\sf c}_1}\frac{1}{|x|^2}
 \quad\text{for } |z|>\frac{K}{2}\sqrt\tau,
 \\
 \bar{u}(x,t)
 &\approx
 \frac{2\tau}{\alpha(T-t)}\frac{1}{e_1(z)}
 \approx
 \frac{2\tau}{\alpha{\sf c}_1}\frac{1}{|x|^2}
 \quad\text{for } |z|>\frac{K}{2}\sqrt\tau,
 \end{align*}
we obtain
 \begin{align*}
 |u(x,t)|
 &<
 \frac{1}{\lambda^2}{\sf Q}(y) 
 +
 \Theta(x,t)
 +
 |w_2(x,t)|
 <
 \frac{c}{\lambda^2}\frac{1}{|y|^4}
 +
 \Theta(x,t)
 +
 c\tau^\frac{5}{16}\frac{1}{|x|^2}
 \\
 &<
 \bar{u}(x,t)
 \qquad
 \text{for } |x|=\frac{K}{2}\sqrt{\tau}\sqrt{T-t},\quad t\in(0,T).
 \end{align*}
In exactly the same computation as \eqref{5.3},
we can verify that
\begin{equation}\label{8.11}
 \pa_t\bar{u}-\Delta\bar{u}-\bar{u}^2
 =
 \underbrace{
 \frac{\alpha}{2}
 \tau^{-2}
 (T-t)^{-2}
 \left( 1+\frac{\alpha}{2\tau}e_1 \right)^{-2}
 e_1
 -
 \frac{\alpha^2}{2}
 \tau^{-2}
 (T-t)^{-2}
 \left( 1+\frac{\alpha}{2\tau}e_1 \right)^{-3}
 |\nabla_ze_1|^2
 }_{=:\bar\mu(x,t)}.
 \end{equation}
Since
$|\nabla_ze_1|^2=4{\sf c}_1e_1+8n{\sf c}_1^2$ (see \eqref{3.4})
it holds that
 \begin{align}\label{8.12}
 \bar\mu
 &=
 \tau^{-2}(T-t)^{-2}
 \left( 1+\frac{\alpha}{2\tau}e_1 \right)^{-3}
 \left(
 \frac{\alpha}{2}
 e_1
 \left( 1+\frac{\alpha}{2\tau}e_1 \right)
 -
 \frac{\alpha^2}{2}
 \left( 4{\sf c}_1e_1+8n{\sf c}_1^2 \right)
 \right)
 \nonumber
 \\
 &=
 \frac{\alpha}{2\tau^2}
 (T-t)^{-2}
 \left( 1+\frac{\alpha}{2\tau}e_1 \right)^{-3}
 \left(
 e_1
 \left( 1+\frac{\alpha}{2\tau}e_1-4\alpha{\sf c}_1 \right)
 -
 8n\alpha{\sf c}_1^2
 \right).
 \end{align}
Therefore since
$e_1(z)>\frac{{\sf c}_1K^2}{8}\tau$ if $|z|>\frac{K}{2}\sqrt\tau$
and $\alpha>0$,
it follows that
 \[
 \bar\mu>0
 \qquad\text{for } |z|>\frac{K}{2}\sqrt\tau
 \]
if $K$ is large enough.
A comparison argument shows \eqref{8.9}.
We next give an estimate of $|\nabla_xu(x,t)|$ in the same manner.
Let
 \[
 v(r,t)
 =
 \pa_ru(r,t),
 \qquad
 r=|x|.
 \]
The function $v(r,t)$ solves
 \[
 v_t=\pa_r^2v+\frac{n-1}{r}\pa_rv-\frac{n-1}{r^2}v+2|u|v
 \qquad\text{in } \R^6\times(0,T).
 \]
We put
 \[
 \bar{v}(r,t)
 =
 -\pa_r\bar{u}(x,t)
 =
 \alpha{\sf c}_1\tau^{-1}(T-t)^{-\frac{3}{2}}
 \left( 1+\frac{\alpha}{2\tau}e_1(z) \right)^{-2}
 |z|,
 \qquad
 r=|x|.
 \]
By a direct computation,
we verify that
 \begin{align}\label{8.13}
 \begin{array}{cl}
 \dis
 \pa_r\Theta(r,t)
 \approx
 -\frac{2\tau}{\alpha{\sf c}_1}\frac{1}{|x|^3}
 & \dis \text{for } |z|>\frac{K}{2}\sqrt\tau,
 \\[4mm]
 \dis
 \bar{v}(r,t)
 \approx
 \frac{4\tau}{\alpha{\sf c}_1}\frac{1}{|x|^3}
 & \dis \text{for } |z|>\frac{K}{2}\sqrt\tau.
 \end{array}
 \end{align}
From \eqref{8.10} and \eqref{8.13},
we easily see that
 \begin{align*}
 |v(r,t)|_{t=0}|
 &=
 \left|
 \left.
 \pa_r
 \left( \frac{{\sf Q}}{\lambda^2} \right)
 \right|_{t=0}
 -
 \pa_r\Theta(r,t)|_{t=0}
 \right|
 \\
 &<
 \left|
 \pa_r\Theta(r,t)|_{t=0}
 \right|
 <
 \bar{v}(r,t)|_{t=0}
 \qquad\text{for } |x|>\frac{K}{2}\sqrt\tau\sqrt{T-t},\quad t=0.
 \end{align*}
Furthermore \eqref{7.23} implies
 \begin{align*}
 |\pa_rw_2(x,t)|
 &\lesssim
 \tau^{-\frac{27}{16}}(T-t)^{-\frac{3}{2}}|z|^2
 =
 \tau^{-\frac{27}{16}}(T-t)^{-\frac{3}{2}}|z|^2|x|^3\frac{1}{|x|^3}
 \\
 &=
 \frac{K^5}{32}\tau^\frac{13}{16}\frac{1}{|x|^3}
 \lesssim
 \tau^{-\frac{3}{16}\tau}\bar{v}
 \qquad
 \text{for } |x|=\frac{K}{2}\sqrt\tau\sqrt{T-t},\quad t\in(0,T).
 \end{align*}
Therefore
it follows from \eqref{8.13} that
 \begin{align*}
 |v(r,t)|
 &<
 \frac{\lambda^2}{|x|^5}
 +
 |\pa_r\Theta(r,t)|
 +
 |\pa_rw_2(r,t)|
 \\
 &<
 \bar{v}(r,t)
 \qquad\text{for } |x|=\frac{K}{2}\sqrt\tau\sqrt{T-t},\quad t\in(0,T).
 \end{align*}
We differentiate \eqref{8.11} to get an equation for $\bar{v}(r,t)=-\pa_r\bar{u}(r,t)$.
 \begin{equation}\label{8.14}
 \pa_t\bar{v}-\pa_r^2\bar{v}-\frac{n-1}{r}\pa_r\bar{v}+\frac{n-1}{r^2}\bar{v}
 -
 2\bar{u}\bar{v}
 =
 -\pa_r\bar\mu.
 \end{equation}
We now compute $-\pa_r\bar\mu$.
For simplicity,
we write
 \[
 {\sf X}=\frac{\alpha}{2\tau}e_1.
 \]
From \eqref{8.12},
we see that
\begin{align*}
 \pa_r\bar\mu
 &=
 \frac{\alpha}{2\tau^2}(T-t)^{-2}
 \left( 1+\frac{\alpha}{2\tau}e_1 \right)^{-4}
 \pa_re_1
 \left\{
 \left( 1+\frac{\alpha}{\tau}e_1-4\alpha{\sf c}_1 \right)
 \left( 1+\frac{\alpha}{2\tau}e_1 \right)
 \right.
 \\
 &\qquad
 -
 \frac{3\alpha}{2\tau}
 \left.
 \left(
 e_1
 \left( 1+\frac{\alpha}{2\tau}e_1-4\alpha{\sf c}_1 \right)
 -
 8n\alpha{\sf c}_1^2
 \right)
 \right\}
 \\
 &=
 \frac{\alpha}{2\tau^2}(T-t)^{-2}
 \left( 1+\frac{\alpha}{2\tau}e_1 \right)^{-4}
 \pa_re_1
 \\
 &\qquad\times
 \left\{
 \left( 1+2{\sf X}-4\alpha{\sf c}_1 \right)
 \left( 1+{\sf X} \right)
 -
 3
 X
 \left( 1+X-4\alpha{\sf c}_1 \right)
 +
 12n\alpha^2{\sf c}_1^2\tau^{-1}
 \right\}
 \\
 &=
 \frac{\alpha}{2\tau^2}(T-t)^{-2}
 \left( 1+\frac{\alpha}{2\tau}e_1 \right)^{-4}
 \pa_re_1
 \left(
 -{\sf X}^2
 +
 8\alpha c_1{\sf X}
 +
 (1-4\alpha{\sf c}_1)
 +
 12n\alpha^2{\sf c}_1^2\tau^{-1}
 \right).
 \end{align*}
Since $\pa_re_1=2{\sf c}_1(T-t)^{-\frac{1}{2}}|z|>0$ and
${\sf X}=\frac{\alpha}{2\tau}e_1>\frac{\alpha{\sf c}_1}{16}K^2$
if $|z|>\frac{K}{2}\sqrt\tau$,
it holds that
 \[
 \pa_r\bar\mu<0
 \qquad\text{for } |z|>\frac{K}{2}\sqrt\tau
 \]
if $K$ is large enough.
Therefore
from \eqref{8.9} and \eqref{8.14},
we deduce that
 \begin{align*}
 \pa_t\bar{v}-\pa_r^2\bar{v}-\frac{n-1}{r}\pa_r\bar{v}+\frac{n-1}{r^2}\bar{v}
 -
 2|u|\bar{v}
 &>
 \pa_t\bar{v}-\pa_r^2\bar{v}-\frac{n-1}{r}\pa_r\bar{v}+\frac{n-1}{r^2}\bar{v}
 -
 2\bar{u}\bar{v}
 \\
 &=
 -\pa_r\bar\mu>0
 \qquad\text{for } |x|>\frac{K}{2}\sqrt\tau\sqrt{T-t},
 \quad
 t\in(0,T).
 \end{align*}
From a comparison argument,
we finally obtain
 \begin{equation}\label{8.15}
 |\pa_ru(r,t)|=|v(r,t)|<\bar{v}(r,t)
 \qquad\text{for } |x|>\frac{K}{2}\sqrt\tau\sqrt{T-t},\quad t\in(0,T).
 \end{equation}
By using the above estimates,
we now compute $\dis\int_{|x|<1}|\nabla_xu|^2dx$.
Let
 \[
 \Omega_5
 =
 \left\{ x\in\R^6;\ |x|<\frac{K}{2}\sqrt\tau\sqrt{T-t} \right\},
 \qquad
 \Omega_6
 =
 \left\{ x\in\R^6;\ \frac{K}{2}\sqrt\tau\sqrt{T-t}<|x|<1 \right\}.
 \]
From \eqref{8.4},
we see that
 \begin{align}\label{8.16}
 \int_{\Omega_5}
 &
 \left|
 |\nabla_xu|^2
 -
 \left| \nabla_x\left( \frac{\sf Q}{\lambda^2} \right) \right|^2
 -
 |\nabla_x(\Theta\chi_2)|^2
 \right|
 dx
 \\
 &\lesssim
 \int_{\Omega_5}
 \left|
 \nabla_x\left( \frac{\sf Q}{\lambda^2} \right)
 \cdot
 \nabla_x(\Theta\chi_2)
 \right|
 dx
 +
 \int_{\Omega_5}
 \left|
 \nabla_x\left(
 \frac{\sigma}{\lambda_0^2}T_1\chi_1
 +
 \frac{\epsilon\chi_\text{in}}{\lambda^2}
 +
 w
 \right)
 \right|^2
 dx
 \nonumber
 \\
 &\qquad
 +
 \int_{\Omega_5}
 \left|
 \nabla_x\left( \frac{\sf Q}{\lambda^2} \right)
 \cdot
 \nabla_x\left(
 \frac{\sigma}{\lambda_0^2}T_1\chi_1
 +
 \frac{\epsilon\chi_\text{in}}{\lambda^2}
 +
 w
 \right)
 \right|
 dx
 \nonumber
 \\
 &\qquad
 +
 \int_{\Omega_5}
 \left|
 \nabla_x\left( \Theta\chi_2 \right)
 \cdot
 \nabla_x\left(
 \frac{\sigma}{\lambda_0^2}T_1\chi_1
 +
 \frac{\epsilon\chi_\text{in}}{\lambda^2}
 +
 w
 \right)
 \right|
 dx.
 \nonumber
 \end{align}
We here note that
 \begin{equation}\label{8.17}
 |\nabla_x\chi_2|
 =
 (T-t)^{-\frac{1}{2}}|\nabla_z\chi_2|
 \lesssim
 \tau(T-t)^{-\frac{1}{2}}{\bf 1}_{\frac{1}{\tau}<|z|<\frac{2}{\tau}}.
 \end{equation}
From this relation,
we easily see that
$|\nabla_x(\Theta\chi_2)|\lesssim
 \tau^{-\frac{1}{2}}(T-t)^{-\frac{3}{2}}{\bf 1}_{|z|>\frac{1}{\tau}}
 +\tau(T-t)^{-\frac{3}{2}}{\bf 1}_{\frac{1}{\tau}<|z|<\frac{2}{\tau}}$.
Therefore
the first term is estimated as
 \begin{align*}
 \int_{\Omega_5}
 \left|
 \nabla_x\left( \frac{\sf Q}{\lambda^2} \right)
 \cdot
 \nabla_x(\Theta\chi_2)
 \right|
 dx
 &\lesssim
 \left(
 \int_{|z|>\frac{1}{\tau}}
 \left|
 \nabla_x\left( \frac{\sf Q}{\lambda^2} \right)
 \right|^2
 dx
 \right)^\frac{1}{2}
 (T-t)^{-\frac{3}{2}}
 \\
 &\qquad
 \times
 \left(
 \frac{1}{\sqrt\tau}
 \left(
 \int_{\Omega_5}
 dx
 \right)^\frac{1}{2}
 +
 \left(
 \int_{|z|<\frac{2}{\tau}}
 dx
 \right)^\frac{1}{2}
 \right)
 \\
 &\lesssim
 \tau^{-\frac{7}{4}}
 \tau
 (T-t)^\frac{3}{2}
 =
 \tau^{-\frac{3}{4}}
 (T-t)^\frac{3}{2}.
\end{align*}
Here we use
 \[
 \int_{|z|>\frac{1}{\tau}}
 \left|
 \nabla_x\left( \frac{\sf Q}{\lambda^2} \right)
 \right|^2
 dx
 <
 \int_{|y|>\frac{\sqrt{T-t}}{\tau\lambda}}
 |\nabla_y{\sf Q}|^2
 dy
 \lesssim
 \left(
 \frac{\sqrt{T-t}}{\tau\lambda}
 \right)^{-4}
 \lesssim
 \tau^{-\frac{7}{2}}
 (T-t)^3.
 \]
Furthermore
from \eqref{4.2} - \eqref{4.3} and \eqref{6.23},
we verify that
 \begin{align*}
 \int
 \left|
 \nabla_x\left( \frac{\sigma}{\lambda^2}T_1\chi_1 \right)
 \right|^2
 dx
 &<
 \left( \frac{\sigma}{\lambda^2} \right)^2
 \frac{1}{\lambda^2}
 \int_{|z|<\frac{2}{\tau}}
 |\nabla_yT_1|^2
 dx
 +
 \left( \frac{\sigma}{\lambda^2} \right)^2
 \frac{1}{T-t}
 \int
 |\nabla_z\chi_1|^2
 dx 
 \\
 &\lesssim
 \frac{\sigma^2}{\lambda^6}
 \int_{|z|<\frac{2}{\tau}}
 \frac{dx}{|y|^6}
 +
 \frac{\tau^2}{(T-t)^3}
 \int_{|z|<\frac{2}{\tau}}
 dx
 \\
 &\lesssim
 \sigma^2
 \int_{|y|<\frac{2\sqrt{T-t}}{\tau\lambda}}
 \frac{dy}{|y|^6}
 +
 \tau^{-4}
 \lesssim
 \tau^{-4},
 \\
 \int
 \left|
 \nabla_x\left( \frac{\epsilon\chi_\text{in}}{\lambda^2} \right)
 \right|^2
 dx
 &<
 \frac{1}{\lambda^4}
 \frac{1}{\lambda^2}
 \int_{|y|<2R}
 |\nabla_y\epsilon|^2
 dx
 +
 \frac{1}{\lambda^4}
 \frac{1}{\lambda^2}
 \int
 \epsilon^2|\nabla_y\chi_\text{in}|^2
 dx 
 \\
 &\lesssim
 \int_{|y|<2R}
 |\nabla_y\epsilon|^2
 dy
 +
 \frac{1}{R^2}
 \int_{R<|y|<2R}
 \epsilon^2 dy
 \\
 &\lesssim
 \frac{\lambda^4}{\tau^3(T-t)^2}
 R^8(\log R)^2
 \left(
 \int_{|y|<2R}
 \frac{dy}{1+|y|^{13}}
 +
 \frac{1}{R^2}
 \int_{R<|y|<2R}
 \frac{dy}{1+|y|^{11}}
 \right)
 \\
 &\lesssim
 R^8(\log R)^2\tau^{-\frac{21}{2}}(T-t)^3.
 \end{align*}
Since $|\nabla_x\Theta(x,t)|>cK^{-4}\tau^{-1}(T-t)^{-\frac{3}{2}}|z|$ for
$|z|<\frac{K}{2}\sqrt\tau$,
we verify from \eqref{7.23} with $p=8>n$ that
 \begin{align*}
 |\nabla_xw|
 &=
 (T-t)^{-\frac{3}{2}}|\nabla_z\varphi(z,\tau)|
 \\
 &\lesssim
 \frac{(T-t)^{-1}}{R^\frac{1}{8}\lambda\tau^\frac{3}{2}}
 \frac{{\bf 1}_{|z|<1}}{1+|y|^\frac{5}{4}}
 +
 \tau^{-\frac{n}{p}}(T-t)^{-\frac{3}{2}}
 {\bf 1}_{|z|<8}
 +
 \tau^{-\frac{27}{16}}(T-t)^{-\frac{3}{2}}
 |z|^2
 {\bf 1}_{|z|>8}
 \\
 &\lesssim
 \frac{(T-t)^{-1}}{R^\frac{1}{8}\lambda\tau^\frac{3}{2}}
 \frac{{\bf 1}_{|z|<1}}{1+|y|^\frac{5}{4}}
 +
 \tau^{-\frac{3}{4}}(T-t)^{-\frac{3}{2}}
 {\bf 1}_{|z|<8}
 +
 K^5\tau^{-\frac{3}{16}}|\nabla_x\Theta|
 {\bf 1}_{|z|>8}
 \end{align*}
for $|z|<\frac{K}{2}\sqrt\tau$, $\tau>\tau_0$.
From this relation and
 \begin{align*}
 \int_{|z|<1}
 \frac{dx}{1+|y|^\frac{5}{2}}
 &=
 \lambda^n
 \int_{|y|<\frac{\sqrt{T-t}}{\lambda}}
 \frac{dy}{1+|y|^\frac{5}{2}}
 \lesssim
 \lambda^6
 \left( \frac{\sqrt{T-t}}{\lambda} \right)^\frac{7}{2}
 =
 \lambda^\frac{5}{2}
 (T-t)^\frac{7}{4},
 \end{align*}
we get
 \begin{align}\label{8.18}
 \int_{\Omega_5}|\nabla_xw|^2
 dx
 &\lesssim
 \frac{(T-t)^{-2}}{R^\frac{1}{4}\lambda^2\tau^3}
 \int_{|z|<1}
 \frac{dx}{1+|y|^\frac{5}{2}}
 +
 \tau^{-\frac{3}{2}}(T-t)^{-3}
 \int_{|z|<8}
 dx
 +
 \frac{1}{\tau^\frac{3}{8}}
 \int_{\Omega_5}
 |\nabla_x\Theta|^2
 dx
 \nonumber
 \\
 &\lesssim
 \lambda^\frac{1}{2}
 \tau^{-3}(T-t)^{-\frac{1}{4}}
 +
 \tau^{-\frac{3}{2}}
 +
 \frac{1}{\tau^\frac{3}{8}}
 \int_{\Omega_5}
 |\nabla_x\Theta|^2
 dx
 \nonumber
 \\
 &\lesssim
 \tau^{-\frac{3}{2}}
 +
 \frac{1}{\tau^\frac{3}{8}}
 \int_{\Omega_5}
 |\nabla_x\Theta|^2
 dx.
 \end{align}
Furthermore
it holds from \eqref{8.17} that
 \begin{equation}\label{8.19}
 \int_{\Omega_5}
 \left|
 |\nabla_x\Theta|^2
 -
 |\nabla_x(\Theta\chi_2)|^2
 \right|
 dx
 \lesssim
 \int_{|z|<\frac{2}{\tau}}
 \left( |\nabla_x\Theta|^2+\tau|\nabla_x\Theta|\Theta+\tau^2\Theta^2 \right)
 dx
 \lesssim
 \tau^{-4}.
 \end{equation}
Therefore
we get from \eqref{8.18} - \eqref{8.19} that
 \begin{align*}
 \int_{\Omega_5}
 \left|
 \nabla_x(\Theta\chi_2)\cdot\nabla_xw
 \right|
 dx
 &\lesssim
 \left\|
 \nabla_x(\Theta\chi_2)
 \right\|_{L_x^2(\Omega_5)}
 \left\|
 \nabla_xw
 \right\|_{L_x^2(\Omega_5)}
 \\
 &\lesssim
 \tau^{-\frac{3}{4}}
 \left\|
 \nabla_x(\Theta\chi_2)
 \right\|_{L_x^2(\Omega_5)}
 +
 \tau^{-\frac{3}{16}}
 \left\|
 \nabla_x(\Theta\chi_2)
 \right\|_{L_x^2(\Omega_5)}^2.
 \end{align*}
Combining the above estimates,
we obtain
 \begin{align*}
 \text{(right-hand side of \eqref{8.16})}
 &\lesssim
 \int_{\Omega_5}|\nabla_xw|^2dx
 +
 \int_{\Omega_5}|\nabla_x(\Theta\chi_2)\cdot\nabla_xw|dx
 +
 \text{(the rest of terms)} \\
 &\lesssim
 \tau^{-\frac{3}{2}}
 +
 \tau^{-\frac{3}{8}}
 \left\|
 \nabla_x(\Theta\chi_2)
 \right\|_{L_x^2(\Omega_5)}^2
 +
 \tau^{-\frac{3}{4}}
 \left\|
 \nabla_x(\Theta\chi_2)
 \right\|_{L_x^2(\Omega_5)}
 \\
 &\qquad
 +
 \tau^{-\frac{3}{16}}
 \left\|
 \nabla_x(\Theta\chi_2)
 \right\|_{L_x^2(\Omega_5)}^2
 \\
 &\lesssim
 \tau^{-\frac{21}{16}}
 +
 \tau^{-\frac{3}{16}}
 \left\|
 \nabla_x(\Theta\chi_2)
 \right\|_{L_x^2(\Omega_5)}^2
 \\
 &\lesssim
 \tau^{-\frac{21}{16}}
 +
 \tau^{-\frac{3}{16}}
 \left\|
 \nabla_x\Theta
 \right\|_{L_x^2(\Omega_5)}^2.
 \end{align*}
In the last line,
we use \eqref{8.19}.
The integral in $\Omega_6$ is much simpler.
In fact,
due to \eqref{8.15} and $\bar{v}\lesssim|\nabla_x\Theta|$,
it holds that
$|\nabla_xu|\lesssim|\nabla_x\Theta|$ in $\Omega_6$.
Therefore we immediately see that
$\int_{\Omega_6}|\nabla_xu|\lesssim\int_{\Omega_6}|\nabla_x\Theta|$.
As a consequence,
we obtain
 \begin{align}\label{8.20}
 \int_{|x|<1}
 |\nabla_xu|^2
 dx
 &<
 \int_{\Omega_5}
 \left| \nabla_x\left( \frac{{\sf Q}}{\lambda^2} \right) \right|^2
 dx
 +
 c\tau^{-\frac{21}{16}}
 +
 (1+c\tau^{-\frac{3}{16}})
 \int_{\Omega_5}
 |\nabla_x\Theta|^2
 dx
 +
 \int_{\Omega_6}
 |\nabla_xu|^2
 dx
 \nonumber
 \\
 &<
 \|\nabla_y{\sf Q}\|_2^2
 +
 c
 \int_{|x|<1}
 |\nabla_x\Theta|^2
 dx
 +
 c\tau^{-\frac{21}{16}}.
 \end{align}
Combining \eqref{8.8} and \eqref{8.20},
we conclude
 \begin{align*}
 E_\text{loc}(u)
 &=
 \frac{1}{2}
 \int_{|x|<1}|\nabla_xu|^2
 dx
 -
 \frac{1}{3}
 \int_{|x|<1}|u|^3
 dx
 \\
 &<
 E({\sf Q})
 +
 c
 \int_{|x|<1}
 |\nabla_x\Theta|^2
 dx
 -
 \frac{1}{6}
 \int_{|z|<\tau^\frac{19}{30}}
 \Theta^3dx
 +
 o_2,
 \end{align*}
where $o_2$ is a positive function with $\lim_{t\to T}o_2=0$.
Since $e_1(z)={\sf c}_1(|z|^2-2n)$,
a direct computation shows
 \begin{align*}
 \int_{|z|<\tau^\frac{19}{30}}
 \Theta^3dx
 &=
 \int_{|z|<\tau^\frac{19}{30}}
 \left( 1+\frac{\alpha{\sf c}_1}{\tau}(|z|^2-2n) \right)^{-3}
 dz
 >
 \int_{|z|<\tau^\frac{19}{30}}
 \left( 1+\frac{\alpha{\sf c}_1}{\tau}|z|^2 \right)^{-3}
 dz
 \\
 &>
 \tau^3
 \int_{|\xi|<\tau^\frac{4}{30}}
 \left( 1+\alpha{\sf c}_1|\xi|^2 \right)^{-3}
 d\xi
 >
 C_1\tau^3\log\tau
 \qquad
 (C_1>0),
 \\
 \int_{|x|<1}
 |\nabla_x\Theta|^2
 dx
 &\lesssim
 \tau^{-2}
 \int_{|z|<\frac{1}{\sqrt{T-t}}}
 \left( 1+\frac{1}{\tau}|z|^2 \right)^{-4}
 |z|^2
 dz
 \\
 &\lesssim
 \tau^2
 \int_{|\xi|<\frac{1}{\sqrt\tau\sqrt{T-t}}}
 \left( 1+|\xi|^2 \right)^{-4}
 |\xi|^2
 d\xi
 \lesssim
 \tau^2\log\tau.
\end{align*}
We finally obtain
 \[
 E_\text{loc}(u)
 <
 E({\sf Q})
 +
 C_2\tau^2\log\tau
 -
 \frac{C_1}{6}\tau^3\log\tau
 +
 o_2.
 \]
This shows $\lim_{t\to\ T}E_\text{loc}(u)=-\infty$.

\section*{Acknowledgement}
The author is partly supported by
Grant-in-Aid for Young Scientists (B) No. 26800065.


\end{document}